\documentclass[11pt, reqno]{amsart}
\usepackage{amsmath,amssymb,amsthm, amsfonts}
\usepackage{enumerate}
\usepackage{caption}
 2

\usepackage{tikz}
\usepackage{bbm}
\usepackage{mathabx}
\usepackage{comment}
\usepackage{enumitem}
\usepackage{relsize}
\newcommand{\smallprod}{\mathop{\mathsmaller{\prod}}}
\usepackage[colorlinks=true,linkcolor=blue,citecolor=blue,pdfpagelabels=false]{hyperref}
\usepackage{xcolor}
\newtheorem{thm}{Theorem}[section]
\newtheorem{lem}[thm]{Lemma}

\newtheorem{cor}[thm]{Corollary}
\newtheorem{defn}[thm]{Definition}
\newcommand{\thmref}[1]{Theorem~\ref{#1}}
\newcommand{\lemref}[1]{Lemma~\ref{#1}}

\theoremstyle{remark}
\newtheorem{rmk}{Remark}[section]

\allowdisplaybreaks
\renewcommand{\geq}{\geqslant}
\renewcommand{\leq}{\leqslant}

\title[ ]{Reductions of $\mathrm{GL}_2(\mathbb Q_{p^f})$-Banach spaces of slopes in $(0,1)$}
\author{Eknath Ghate and Shivansh Pandey}
\address{School of Mathematics, Tata Institute of Fundamental Research, Homi Bhabha Road, Mumbai-400005, India.}
\email{eghate@math.tifr.res.in}
\email{shivansh@math.tifr.res.in}

\begin{document}
\date{\today}
\begin{abstract}
Let $p$ be an odd prime and $f \geq 1$. We consider a $p$-adic locally algebraic $\text {GL}_2(\mathbb Q_{p^f})$-representation attached to a tuple of $f$ weights $k = (k_i)$ for $0 \leq i \leq f-1$
and a $p$-adic integer $a_p$ with valuation in $(0,1)$. We give conditions under which the irreducible 
quotients of the subquotients in a filtration on 
the reduction mod $p$ of the natural integral structure on this space are supercuspidal.
We also check that for small $k$ and $f$  the integral structure is a lattice so that the mod $p$ reduction is nonzero.
\end{abstract}
\maketitle

\section{Introduction}

Let $p$ be a prime.
The $p$-adic and mod $p$ Local 
Langlands correspondences 
for $\mathrm{GL}_2({\mathbb Q}_p)$ were introduced now two decades ago by Breuil \cite{breuilmathjassieu}.  
These were subsequently made functorial by Colmez \cite{Col}.  
One nice application 
of these correspondences is that one can use them to give an alternative approach to computing the mod $p$ reduction of the local Galois representations attached to an elliptic modular form by computing instead the mod $p$ reduction of a certain $p$-adic 
$\mathrm{GL}_2({\mathbb Q}_p)$-Banach space, 
at least when $p$
does not divide the level of the form (and even in some cases when $p$ does)
and the form has positive slope (the case of slope $0$ is well understood). Here, the slope is as usual the 
$p$-adic valuation of the 
$p$-th Fourier coefficient of the form, where
the valuation is normalized so that the valuation of $p$ is $1$. 
The computation of the reduction by this method was initiated by 
Breuil in \cite{Br} for weights up to $2p$ and by 
Buzzard-Gee \cite{BGEE} for slopes in the range $(0,1)$,
at least if $p > 2$.

It seems natural to try and generalize some of these results to finite extensions of ${\mathbb Q}_p$. 
Much effort has gone into writing down the $p$-adic and mod $p$ Langlands correspondences for $\mathrm{GL}_2({\mathbb Q}_{p^f})$ 
where
${\mathbb Q}
_{p^f}$ is the unique unramified extension of ${\mathbb Q}_{p}$ of
degree $f \geq 1$.
For instance, Buzzard-Diamond-Jarvis \cite{BDJ}
foreshadow such a mod $p$ correspondence by 
providing an explicit recipe for the Serre weights which should occur in the ${\mathrm{GL}_2(\mathbb Z}_{p^f})$-socle
of the possible mod $p$
${\mathrm{GL}_2(\mathbb Q}_{p^f})$-representation 
corresponding to a two-dimensional mod $p$ representation of the Galois group of ${\mathbb Q}_{p^f}$. Subsequently,
Breuil and Pa\v{s}k\={u}nas \cite{BP}
constructed families of supercuspidal 
representations with
this property using  diagrams. More recently, 
Breuil, Herzig, Hu, Morra and Schraen (\cite{BHHMS}, \cite{BHHMS1}, \ldots) have
made further progress towards
the mod $p$ local Langlands correspondence in this setting by showing that certain candidate mod $p$ $\mathrm{GL}_2({\mathbb Q}_{p^f})$-representations have finite length.
For some recent progress for ramified extensions of ${\mathbb Q}_p$, see Schein \cite{Sch}. 

There is a natural
generalization of the
aforementioned Banach space to 
the case $f \geq 2$. 
In spite of the above correspondences not yet being completely available, it seems natural to ask what the mod $p$ reductions of these
Banach spaces are.
Let $\Theta_{k,a_p}$ be the standard lattice in the locally algebraic representation $\Pi_{k,a_p}$ whose completion is this Banach space, where
$k = (k_i)$ is a tuple of $f$ weights $k_i$ for $0 \leq i \leq f-1$ with $f \geq 1$ and $a_p$ lies in a finite extension of ${\mathbb Q}_{p^f}$.
Note that the reduction of this lattice is the same as
the reduction of its completion  in the Banach space. 
Call the normalized $p$-adic valuation of $a_p$ the slope. In the case of $f =1$, Buzzard-Gee \cite{BGEE} showed  that when $p>2$ and the slope is in $(0,1)$, the mod $p$ reduction of 
$\Theta_{k, a_p}$ is 
essentially  
supercuspidal
(and hence irreducible). This is as opposed to the reduction 
having length two (or three), where both the relevant Jordan-H\"older factors are principal series (or in a special case, 
a principal series, a character, and a twist of the  
Steinberg representation).

One might ask to what extent  this supercuspidality propagates to the case of general $f \geq 1$ (and $p > 2$) but 
the slope is still in $(0,1)$.  The goal of this paper is to lay some of the foundations needed to begin to answer this question and to then answer it positively to some extent.

Many complications arise. Let $V_r$ for $r = (r_i)$ with $r_i = k_i-2$ be the usual tensor product of mod $p$ symmetric power representations of 
${\mathrm {GL}}_2({\mathbb Z}_{p^f})$, modeled as usual on certain multi-variable polynomials. It turns out that the mod $p$ reduction of $\Theta_{k,a_p}$ is uniformized by the compact induction of a quotient $V_r/V_r^*$  of $V_r$, where $V_r^*$ is the subspace spanned by certain polynomials defined in \cite{GJ} and which generalize the Dickson polynomial when $f = 1$. When $f = 1$, the quotient  $V_r/V_r^*$ has $2$ weights. 
One of the weights (essentially the socle of 
$V_r/V_r^*$) is known, by
work of Glover \cite{G}, to be generated by the highest monomial in (the polynomial model of) $V_r$. Buzzard-Gee \cite{BGEE}  show that this weight dies
because the slope is positive, leaving one with a single weight. Since one knows that the mod $p$ reduction
of $\Theta_{k,a_p}$ is in the image of the mod $p$ local Langlands correspondence,
this already forces the reduction to be supercuspidal (except in the very special case that the remaining weight has dimension $p-1$).

However, for general $f \geq 1$, it turns out that there are up to $2^f$ weights in the symmetric power representation quotient $V_r/V_r^*$. We start this paper by showing that as in the case of $f =1$, the representation generated by the highest monomial in (the
polynomial model) of $V_r/V_r^*$  is a single weight 
(it is again essentially the socle of $V_r/V_r^*$) which dies when the slope is positive. Let $Q$ be the 
resulting quotient
of $V_r/V_r^*$ by this weight. Then $Q$ has up to $2^{f}-1$ weights, so this  still leaves one with many weights to handle when $f > 1$. Moreover, one has no recourse to a mod $p$ local  Langlands correspondence with which one might obtain some sort of handle on the 
possibilities for the weights contributing to the reduction of $\Theta_{k,a_p}$.

Nonetheless, in this paper we show that for general $f \geq 1$, the 
part of the reduction corresponding to the topmost weight
(essentially the cosocle of the aforementioned quotient $Q$) has the property that it is a quotient of a universal supersingular representation, so that all its irreducible quotients 
are supercuspidal, except possibly in 
the particular case 
when the weight $k$ is exceptional and 
the slope is $\frac{1}{2}$. 
We can even prove this in this exceptional case if
we assume that
a certain quantity 
involving $a_p$ is a unit. Let $q := p^f$. We have:

\begin{thm} 
Assume $p > 2$ and $f \geq 1$.
Assume $r_i = k_i-2 >q+p-2$ and $r = r_0 + r_1p + \cdots + r_{f-1}p^{f-1}$ and  $v(a_p) \in (0,1)$. 
Then all irreducible
quotients 
of the subquotient of the reduction mod $p$ of $\Theta_{k,a_p}$ corresponding to the `cosocle' of $Q$ are supercuspidal if
either
\begin{enumerate}
    \item
 $r \not\equiv p^h \textnormal{ mod } (q-1)$ for all
$0 \leq h \leq f-1$, or 
\item
$r\equiv p^h \textnormal{ mod } (q-1)$ for some
$0 \leq h \leq f-1$, 
with the additional caveat that
$$\frac{a_p^2 - pr_h}{p}$$ is a unit
if  
$v(a_p) = \frac{1}{2}$.
\end{enumerate}
\end{thm}
\noindent When $r \equiv 0 \textnormal{ mod } (q-1)$, it is almost immediate that $Q$ has only one weight (i.e., $Q$ is irreducible) 
for any  $f \geq 1$. We immediately obtain:

\begin{cor}  
Let $p > 2$ and $f \geq 1$.
If $r \equiv 0 \textnormal{ mod } (q-1)$, $r_i> q+p-2$ and $v(a_p) \in (0,1)$, then all the irreducible 
quotients of the mod $p$ reduction of $\Theta_{k,a_p}$
are supercuspidal.
\end{cor}

We now turn
to the second highest factors in the socle filtration of $Q$. 
For simplicity, we assume that there is only one such factor in this layer of the socle filtration.
We then show that the irreducible quotients of the reduction corresponding to
this weight are again supercuspidal for general $f \geq 1$, at least 
if $v(a_p) > \frac{1}{2}$ and 
some divisibility conditions on the weights hold. More precisely, we prove:

\begin{thm}
Let $p>2$ and $f \geq 1$. Assume that $r_i \geq q$ for all $0 \leq i \leq f-1$ and that $r\equiv a_hp^h ~\textnormal{mod }(q-1)$ with  $a_h\in \{1,2,3,\ldots, p-1\}$ for some $0 \leq h \leq f-1$. Assume that $p\nmid r_{h+1}$ and $p\mid r_{h}$ if $a_h=1$ and that $v(a_p)\in (\frac{1}{2},1)$. 
Then all the  
irreducible quotients of the subquotient of the reduction mod $p$ of $\Theta_{k,a_p}$ corresponding to the (unique) second highest weight in the socle filtration of $Q$ are supercuspidal. 
\end{thm}

If $f = 2$, 
and we assume the second layer of the socle filtration of $Q$ consists of only one weight, this gives a reasonable description of the reduction. Indeed
in this case there are only two layers in the socle filtration of $Q$, the `cosocle' of $Q$ and the `socle' of $Q$. We obtain:

\begin{cor}
   Let $p > 2$ and $f = 2$. Let $r=r_0+pr_1$, $r_i>q+p-2$ and $a_p\in \bar{\mathbb Z}_p$ with $v(a_p)\in (0,1)$. Then there is an at most 
   two-step filtration
   on the mod $p$ reduction of $\Theta_{k,a_p}$
   such that every 
   irreducible quotient 
   of the successive  quotients in this filtration
   are supercuspidal if one of the following conditions is satisfied:
   \begin{itemize}
       \item $r\equiv 0$ \textnormal{mod} $(q-1)$
   \end{itemize}
   or $v(a_p) \in (\frac{1}{2},1)$ and
   \begin{itemize}  
     \item $r\equiv 1 \textnormal{ mod } (q-1)$, $p\nmid r_1$ and  $p\mid r_0$
     \item $r\equiv 2,3,\ldots, p-1 \textnormal{ mod } (q-1)$ and $p\nmid r_1$
    \item $r\equiv p \textnormal{ mod } (q-1) $, $p\nmid r_0$ and $p\mid r_1$
    \item  $r\equiv 2p,3p,\ldots, (p-1)p  \textnormal{ mod } (q-1)$ and $p\nmid r_0$.
   \end{itemize}
\end{cor}

The results above follow from Theorems~\ref{requivzero}, \ref{nonexcep}, \ref{thmexcept}, \ref{middle1}, \ref{middle2}, \ref{gmiddle1} and \ref{gmiddle2} which are proved in the main body of the paper using
explicit computations with the Hecke operator $T$. These theorems show that the subquotient of
the reduction mod $p$
of $\Theta_{k,a_p}$ in question is a quotient
of the compact induction of a weight modulo the image of $T$. It is a fact going
back to \cite{BL}
that the irreducible quotients of 
such universal supersingular representations are supercuspidal. We remark
here that it is not clear
what the other Jordan-H\"older factors of the subquotients in the reduction might be.
Indeed, in the totally
ramified (not equal to ${\mathbb Q}_p$) case, Schein
\cite[Remark 3.10]{SchJNT} remarks that the socle of the corresponding universal supersingular representation vanishes (he shows  this using his computation of the $I(1)$-invariants - see also \cite{Hen}, \cite{CJS} - that every proper submodule is reducible). See the appendix for some details. Even worse, it follows from
the proof of \cite[Theorem 3.3]{BreuilICM} that
in this totally ramified
setting, a universal supersingular representation has
non-supercuspidal Jordan-H\"older factors. Subsequently, Morra  \cite[Theorem 1.3]{Mor12} showed that even in the unramified case  if $f > 1$ and $r \neq 0 \text{ mod } (q-1)$, then the universal supersingular 
representation has non-supercuspidal Jordan-H\"older factors. See also the recent work
\cite{FLLMR} for a generalization of this
statement to the case of $\mathrm{GL}_n({\mathbb Q}_{p^f})$ 
for $n \geq 3$.

During the preparation of this paper,  an interesting side question arose,
namely whether the mod $p$ reduction 
of $\Theta_{k,a_p}$ is nonzero! This is related to the freeness of $\Theta_{k,a_p}$, that is, whether it is indeed a lattice in $\Pi_{k,a_p}$. In the
case of $f = 1$, the non-vanishing of the reduction of $\Theta_{k,a_p}$ follows immediately from
the fact that the reduction is in the
image of the mod $p$ local Langlands correspondence, which is an injective map.  However, historically, before this fact was known, Breuil directly checked in the case $f = 1$ that the reduction is nonzero for weights $k \leq 2p$, at least if $p > 2$. Following his ideas,
we similarly check this non-vanishing 
for small $f$ and $k$. For instance,
if  all $k_i \leq p+1$, we recover the
non-vanishing for $f\geq 1$, $p \geq 2$ and all $a_p$ of positive slope, and 
if all $k_i \leq 2p$  with at least one $k_i \leq p+2$, we prove
the non-vanishing if $f = 2$, $p > 2$ and $v(a_p) \in (0,1)$. This is explained in Section \ref{thetanon-zero} at
the beginning of the paper.

An interesting consequence is that when all $k_i \leq p+1$, the reduction
of $\Theta_{k,a_p}$ is {\it equal} to the 
universal supercuspidal representation,
and so has infinite length if $f > 1$.
Presumably, this infinite length property is true for some
of the other general subquotients discussed above.
We remark also that the
non-vanishing
of the mod $p$ reduction of
$\Theta_{k,a_p}$ is related to 
the Breuil-Schneider conjecture \cite{BS}  which describes when a locally algebraic representation contains a lattice. This conjecture has now been proven  in various settings
by Vign\'eras \cite{Vig} (smooth tame principal series), de Ieso \cite{Ies} (when the weights are small), Assaf-Kazhdan-de Shalit \cite{AKS} 
(using Kirillov models), Caraiani-Emerton-Gee-Geraghty-Pa\v{s}k\={u}nas-Shin \cite{CEGGPS}
(arbitrary dimension),
Assaf \cite{Ass}
(larger weights, but with all but one weight zero), Nadimpalli-Sheth
\cite{NS} (for $\mathrm{GL}_2$ over a division algebra).

Finally, as noted at the start, the
behavior of the reduction of $\Theta_{k,a_p}$ precisely predicts
the shape of the reduction of the local
Galois representation attached to an
elliptic modular form (with no nebentypus, and $p$ not dividing the level of the form). We similarly expect
that the results in this paper should have
some connection to the description of
the mod $p$ reduction of the local Galois representation attached to a Hilbert modular form over a totally real field $F$ of degree 
$f$ over ${\mathbb Q}$ at a prime $p$  which is inert (and unramified) in $F$.

\section{Formulas for the Hecke operator}

Let $p$ be an odd prime. Let $G=\text{GL}_2(\mathbb Q_{p^f})$,  $K=\text{GL}_2(\mathbb{Z}_{p^f})$ be the maximal compact subgroup of $G$ and $Z=\mathbb Q_{p^f}^{\times}$ be the center of $G$. 

Let $R$ be a $\mathbb Z_{p^f}$-algebra and let $V=(\otimes _{i=0}^{f-1}\text{Sym}^{r_i}R^2)\otimes D^s$ be the usual symmetric power representation of $KZ$ twisted by a power of the determinant character $D$, modeled on polynomials $v(X_i,Y_i)$ over $R$ in the variables $X_0,Y_0, \ldots ,X_{f-1},Y_{f-1}$ and which are homogeneous of degree $r_i$ in the variables $X_i$, $Y_i$ for $i\in \{0,1,...,f-1\}$. Let $\mathrm{Fr}$ be the Frobenius in $\mathrm{Gal}(\mathbb Q_{p^f}/ \mathbb Q_p)$. Here $\left(\begin{smallmatrix}
    a&b\\c&d
\end{smallmatrix}\right)\in KZ$ acts on the $i$-th factor in the tensor product via $\mathrm{Fr^i}$, i.e.,   
$$\left(\begin{smallmatrix}
    a&b\\c&d
\end{smallmatrix}\right)\cdot v(X_i,Y_i)=v(a^{\mathrm{Fr}^i}X_i+c^{\mathrm{Fr}^i}Y_i,b^{\mathrm{Fr}^i}X_i+d^{\mathrm{Fr}^i}Y_i).$$

Let $\text{ind}_{KZ}^{G} V$ denote 
the compact induction of $V$ from
$KZ$ to $G$. For $g\in G$ and $v \in V,$ let $[g,v]\in \text{ind}_{KZ}^{G} V$ be the function with support in $KZg^{-1}$ given by 

\begin{equation*}
    g' \mapsto
\begin{cases}
g'g \cdot v & \text{if } g' \in KZg^{-1}, \\
0 & \text{otherwise}.
\end{cases}
\end{equation*}
Any function in $\text{ind}_{KZ}^{G} V$ is a finite linear combination of functions of the form $[g,v],$ for $g\in G$ and $v\in V$. 

Let $\mathcal{H}(G, V) := \mathrm{End}_G(\text{ind}_{KZ}^{G} V)$ be the corresponding Hecke algebra. There is an isomorphism of algebras between
$\mathcal{H}(K,V)$ and the  algebra under convolution of
functions $\varphi$ from $G$ to $\mathrm{End}_R(V)$ with compact support modulo $KZ$ satisfying $\varphi(k_1gk_2)=k_1\circ\varphi(g)\circ k_2$ for $k_1$, $k_2\in KZ$ and $g\in G$.  
The action of the Hecke operator $T_\varphi$ corresponding to a function $\varphi :G\rightarrow \mathrm{End}_R(V)$ is given by 
 \begin{equation}\label{Heckeop}
       T_\varphi([g,v]) =\sum\limits_{yKZ\in G/KZ } [gy, \varphi(y^{-1})v] \quad \forall g \in G,\; v \in V.
   \end{equation}

The goal of this section is to make the formula for $T_\varphi$  explicit for a particular function $\varphi$. We give the details to make this paper self-contained.

First we prove the following lemma:

\begin{lem}\label{independence}
    Let $\sigma$ denote the Frobenius automorphism of $\mathbb Q_{p^f} / \mathbb Q_{p}$ and $P\in \mathbb Q_{p^f}[X_0,\ldots, X_{f-1}]$ be such that $P(t,\sigma(t),...,\sigma^{f-1}(t))=0$ for every $t \in \mathbb Z_{p^f}$. Then $P=0.$
\end{lem}

\begin{proof}
    Let $\mathbb Z_{p^f}=\mathbb Z_p[\alpha]$.  We define a map $\phi:\mathbb Q_{p^f}[X_0,\ldots, X_{f-1}]\rightarrow \mathbb Q_{p^f}[A_0,\ldots, A_{f-1}]$ by change of variables 
    \begin{eqnarray*}\begin{pmatrix}
        X_0\\X_1\\\vdots \\X_{f-1}
    \end{pmatrix}=\begin{pmatrix}
        1&\alpha&\cdots&\alpha^{f-1}\\
        1&\sigma(\alpha)&\cdots&\sigma(\alpha^{f-1})\\
        \vdots&\vdots&\cdots&\vdots\\1&\sigma^{f-1}(\alpha)&\cdots&\sigma^{f-1}(\alpha^{f-1})
    \end{pmatrix}\begin{pmatrix}
        A_0\\A_1\\\vdots \\A_{f-1}
    \end{pmatrix}\end{eqnarray*}
i.e., $P(X_0,X_1,...,X_{f-1})\mapsto \phi_P(A_0,A_1,...,A_{f-1})= P(A_0 +\alpha A_1+ \cdots+ \alpha^{f-1}A_{f-1},A_0 +\sigma(\alpha) A_1+ \cdots+ \sigma(\alpha^{f-1})A_{f-1},...,A_0 +\sigma^{f-1}(\alpha) A_1+ \cdots+ \sigma^{f-1}(\alpha^{f-1})A_{f-1}).$
Since the determinant of the matrix above is nonzero (its square is the discriminant of $\mathbb Q_{p^f} / \mathbb Q_{p}$), we can define the inverse map $\phi^{-1}:\mathbb Q_{p^f}[A_0,\ldots, A_{f-1}] \rightarrow  \mathbb Q_{p^f}[X_0,\ldots, X_{f-1}]$ by the change of variables 
\begin{eqnarray*}\begin{pmatrix}
        A_0\\A_1\\\vdots \\A_{f-1}
    \end{pmatrix}=\begin{pmatrix}
        1&\alpha&\cdots&\alpha^{f-1}\\
        1&\sigma(\alpha)&\cdots&\sigma(\alpha^{f-1})\\
        \vdots&\vdots&\cdots&\vdots\\1&\sigma^{f-1}(\alpha)&\cdots&\sigma^{f-1}(\alpha^{f-1})
    \end{pmatrix}^{-1}\begin{pmatrix}
        X_0\\X_1\\\vdots \\X_{f-1}
    \end{pmatrix}.\end{eqnarray*}
    Hence the map $\phi$ is injective. Now suppose that $P(t,\sigma(t),...,\sigma^{f-1}(t))=0$ for every $t \in \mathbb Z_{p^f}$. Then $\phi_P(A_0,...,A_{f-1})=0$ for all $(A_0,\ldots, A_{f-1})\in \mathbb Z_p^f$ and hence $\phi_P=0.$ But since the map $\phi$ is injective, we have $P=0.$ 
\end{proof}

\begin{lem}
    \label{hecke operator}
    Let $\alpha=\left({\begin{smallmatrix}
        1&0\\0&p
    \end{smallmatrix}} \right)$. Let $\varphi :G\rightarrow \mathrm{End}_R(V)$ be a function with support in $KZ\alpha^{-1}KZ$ such that $\varphi(k_1\alpha^{-1}k_2)=k_1\circ \varphi(\alpha^{-1})\circ k_2$ for all $k_1$, $k_2 \in KZ$. Then up to a multiplication by a scalar $a \in R$, the action of $\varphi(\alpha^{-1})$ is given by $\varphi(\alpha^{-1})v(X_i,Y_i)=v(pX_i,Y_i).$
 \end{lem}

 \begin{proof}
      For ease of notation, we shall restrict to the case $f = 2$, i.e., we assume that  $V=\mathrm{Sym}^{r_0}R^2\otimes \mathrm{Sym}^{r_1}R^2.$ 
      The case of general $f$ is proved similarly. 
     The property $\varphi(k_1\alpha^{-1}k_2)=k_1\circ \varphi(\alpha^{-1})\circ k_2$  for $k_1$, $k_2 \in KZ$ implies that  
     if $k_1\alpha^{-1}=\alpha^{-1}k_2$, then we have $k_1\circ\varphi(\alpha^{-1})=\varphi(\alpha^{-1})\circ k_2$.
     Since
     \begin{eqnarray*}
       \begin{pmatrix}
           a'&pb'\\c'&d'
       \end{pmatrix} \begin{pmatrix}
           1&0\\0&p^{-1}
       \end{pmatrix} =\begin{pmatrix}
           a'&b'\\c'&p^{-1}d'
       \end{pmatrix}=\begin{pmatrix}
           1&0\\0&p^{-1}
       \end{pmatrix} \begin{pmatrix}
           a'&b'\\pc'&d'
       \end{pmatrix}
     \end{eqnarray*} 
   for every $a',d' \in \mathbb Z_{p^2}^{\times}$ and $b',c'\in \mathbb Z_{p^2}$, we have 
     \begin{eqnarray}\label{phi1}
       \begin{pmatrix}
           a'&pb'\\c'&d'
       \end{pmatrix} \circ \varphi(\alpha^{-1}) =\varphi(\alpha^{-1})\circ \begin{pmatrix}
           a'&b'\\pc'&d'
       \end{pmatrix}.
       \end{eqnarray}
       In particular, taking $a'=d'=1, c'=0,b'=t,$ we have 
        \begin{eqnarray*}
       \begin{pmatrix}
           1&pt\\0&1
       \end{pmatrix} \circ \varphi(\alpha^{-1}) =\varphi(\alpha^{-1})\circ \begin{pmatrix}
           1&t\\0&1
       \end{pmatrix}.
       \end{eqnarray*}
       Equip $\mathrm{Sym}^{r_i} R^2$ with the natural basis $X_i^{r_i}, X_i^{r_i-1}Y_i, \ldots, Y_i^{r_i}$ for $i = 0,1$. Then the action of $\left(\begin{smallmatrix}
           1 &pt \\0&1
       \end{smallmatrix}\right)$ on $\mathrm{Sym}^{r_0}R^2\otimes \mathrm{Sym}^{r_1}R^2$ is given by  
       $$\begin{pmatrix}
           \binom{j_0}{i_0}(pt)^{(j_0-i_0)}
       \end{pmatrix}_{\substack{0\le i_0\le r_0\\0\le j_0\le r_0}}\otimes \begin{pmatrix}
           \binom{j_1}{i_1}(pt^{\sigma})^{(j_1-i_1)}
       \end{pmatrix}_{\substack{0\le i_1\le r_1\\0\le j_1\le r_1}} $$
       which in 
       the ordered basis $$X_0^{r_0} X_1^{r_1}, X_0^{r_0} X_1^{r_1 - 1} Y_1, \ldots, X_0^{r_0} Y_1^{r_1}, X_0^{r_0-1} Y_0 X_1^{r_1}, \ldots, X_0^{r_0-1} Y_0 Y_1^{r_1}, \ldots, Y_0^{r_0} Y_1^{r_1}$$ is given 
       by the matrix $$ \begin{pmatrix}
        1& pt^{\sigma}&* &\cdots&* &(pt)^{r_0}(pt^{\sigma})^{r_1}\\
        0&1&*&\cdots&*& (pt)^{r_0}(pt^{\sigma})^{r_1-1}\\
        0 &0&1 &\cdots&\vdots &\vdots\\
        \vdots &\vdots&\vdots &\cdots&\vdots &\vdots\\
        0 &0 &0&\cdots&* &(pt)^{r_0}\\
        0&0&0&\cdots&* & (pt)^{r_0-1}(pt^{\sigma})^{r_1}\\
        \vdots&\vdots&\vdots &\cdots&\vdots &\vdots \\
        0&0&0&\cdots& 1& pt^{\sigma}\\
        0&0&0&\cdots &0& 1
    \end{pmatrix},$$
The matrix of $\left(\begin{smallmatrix}
           1 &t \\0&1
       \end{smallmatrix}\right)$ 
in this basis is very similar (one drops the powers of $p$).     
       Let $\varphi(\alpha^{-1})=(a_{ij})$ be the matrix of $\varphi(\alpha^{-1})$ in 
       the above basis. All these matrices are square 
       matrices of size $(r_0 +1) \times (r_1 + 1)$.
       We have 
\begin{eqnarray*}
&&  \tiny{  \begin{pmatrix}
        1& pt^{\sigma} &\cdots&* &(pt)^{r_0}(pt^{\sigma})^{r_1}\\
        0&1&\cdots&*& (pt)^{r_0}(pt^{\sigma})^{r_1-1}\\
        \vdots &\vdots &\cdots&\vdots &\vdots\\
        0 &0 &\cdots&* &(pt)^{r_0}\\
        0&0&\cdots&* & (pt)^{r_0-1}(pt^{\sigma})^{r_1}\\
        \vdots&\vdots &\cdots&\vdots &\vdots \\
        0&0&\cdots& 1& pt^{\sigma}\\
        0&0&\cdots &0& 1
    \end{pmatrix}\begin{pmatrix}
        a_{1,1}&a_{1,2}&\cdots &a_{1,(r_0+1)(r_1+1)} \\
        a_{2,1}&a_{2,2}&\cdots& a_{2,(r_0+1)(r_1+1)}\\
\vdots &\vdots &\cdots &\vdots\\
        a_{\tiny{(r_0+1)(r_1+1),1}}&a_{\tiny{(r_0+1)(r_1+1),2}}&\cdots& a_{{\tiny(r_0+1)(r_1+1),(r_0+1)(r_1+1)}}
    \end{pmatrix}}\\
  \qquad && =\tiny{\begin{pmatrix}
        a_{1,1}&a_{1,2}&\cdots &a_{1,(r_0+1)(r_1+1)} \\
        a_{2,1}&a_{2,2}&\cdots &a_{2,(r_0+1)(r_1+1)}\\
\vdots &\vdots &\cdots &\vdots\\
        a_{(r_0+1)(r_1+1),1}&a_{(r_0+1)(r_1+1),2}&\cdots& a_{(r_0+1)(r_1+1),(r_0+1)(r_1+1)}
    \end{pmatrix} \begin{pmatrix}
        1& t^{\sigma} &\cdots&* &t^{r_0}(t^{\sigma})^{r_1}\\
        0&1&\cdots&*& t^{r_0}(t^{\sigma})^{r_1-1}\\
        \vdots &\vdots &\cdots&\vdots &\vdots\\
        0 &0 &\cdots &*&t^{r_0}\\
        0&0&\cdots&* & t^{r_0-1}(t^{\sigma})^{r_1}\\
        \vdots&\vdots &\cdots&\vdots &\vdots \\
        0&0&\cdots& 1& t^{\sigma}\\
        0&0&\cdots &0& 1
    \end{pmatrix}}.
\end{eqnarray*}
Comparing the $(r_0+1)(r_1+1),(r_0+1)(r_1+1)$ entry of both sides, we get 
\begin{eqnarray*}
  a_{(r_0+1)(r_1+1),(r_0+1)(r_1+1)} &=& a_{(r_0+1)(r_1+1),1} t^{r_0}(t^{\sigma})^{r_1} + a_{(r_0+1)(r_1+1),2} t^{r_0}(t^{\sigma})^{r_1-1}\\
  &&+\cdots +a_{(r_0+1)(r_1+1),(r_0+1)(r_1+1)-1} t^{\sigma} +a_{(r_0+1)(r_1+1),(r_0+1)(r_1+1)}.
\end{eqnarray*}
Hence, for all $t \in {\mathbb Z}_{p^2}$, we have
\begin{eqnarray*}
  a_{(r_0+1)(r_1+1),1}t^{r_0}(t^{\sigma})^{r_1} + a_{(r_0+1)(r_1+1),2}t^{r_0}(t^{\sigma})^{r_1-1} + \cdots+ a_{(r_0+1)(r_1+1),(r_0+1)(r_1+1)-1}t^{\sigma} & = & 0.
\end{eqnarray*}
\lemref{independence} implies $a_{(r_0+1)(r_1+1),j}=0$ if $j\neq (r_0+1)(r_1+1).$ Similarly, comparing the $(r_0+1)(r_1+1)-1,(r_0+1)(r_1+1)$ entry, we obtain
\begin{eqnarray*}
  pt^{\sigma}a_{(r_0+1)(r_1+1),(r_0+1)(r_1+1)}&=&a_{(r_0+1)(r_1+1)-1,1}t^{r_0}(t^{\sigma})^{r_1} + a_{(r_0+1)(r_1+1)-1,2}t^{r_0}(t^{\sigma})^{r_1-1}\\
  &&+\cdots +a_{(r_0+1)(r_1+1)-1,(r_0+1)(r_1+1)-1}t^{\sigma}.\\
\end{eqnarray*}
Again \lemref{independence} implies
\begin{itemize}
    \item  $a_{(r_0+1)(r_1+1)-1,(r_0+1)(r_1+1)-1}=pa_{(r_0+1)(r_1+1),(r_0+1)(r_1+1)}$, and
    \item $a_{(r_0+1)(r_1+1)-1,j}=0$ for $j \neq (r_0+1)(r_1+1)-1, (r_0+1)(r_1+1)$.
\end{itemize}
Similarly, computing all the $i,(r_0+1)(r_1+1)$ entries we obtain that the matrix of $\varphi(\alpha^{-1})$ is an upper triangular matrix of the form 
\[
\begin{pmatrix}
p^{r_0+r_1}a & * & \cdots & * & * & \cdots & * & * \\
0 & p^{r_0+r_1-1}a & \cdots & *& * & \cdots & * & * \\
\vdots & \vdots & \ddots &  & \vdots & \vdots & \vdots & \vdots \\
0 & 0 & \cdots & p^{r_0}a & * & \cdots & * & * \\
0 & 0 & \cdots & 0 & p^{r_0-1+r_1}a & \cdots & * & * \\
\vdots & \vdots & \vdots & \vdots & \vdots & \ddots & \vdots & \vdots \\ 
0 & 0 & \cdots & 0 & 0 & \cdots & pa & * \\
0 & 0 & \cdots & 0 & 0 & \cdots & 0 & a
\end{pmatrix}
\]
where $a=a_{(r_0+1)(r_1+1),(r_0+1)(r_1+1)}$. 
Similarly, taking $a'=1,d'=1,b'=0,c'=t$ in \eqref{phi1}, we have 
       \begin{eqnarray*}
       \begin{pmatrix}
           1&0\\t&1
       \end{pmatrix} \circ \varphi(\alpha^{-1}) =\varphi(\alpha^{-1})\circ \begin{pmatrix}
           1&0\\pt&1
       \end{pmatrix}.
       \end{eqnarray*}
A similar calculation shows that matrix of $\varphi(\alpha^{-1})$ in the same basis is lower triangular with the same diagonal entries. Hence the matrix of $\varphi(\alpha^{-1})$ is the diagonal matrix given by  
     \begin{eqnarray*}
  \left(  \begin{smallmatrix}
    p^{r_0+r_1}a&&& &&\\
        &p^{r_0+r_1-1}a&&&&\\
      &\ddots&\\  &&  p^{r_0}a\\
        & & &p^{r_0-1+r_1}a& &&  \\      
        &&&\ddots&\\
        &&&&pa &\\
        &&&& &a
    \end{smallmatrix}\right).
\end{eqnarray*}  
Thus up to multiplication by 
the scalar $a$, we see that 
$\varphi(\alpha^{-1})$ takes $X_i$ to $pX_i$ and fixes $Y_i$.
\end{proof}

\begin{defn}
    Let $T$ be the Hecke operator corresponding to $\varphi$ 
in Lemma~\ref{hecke operator}
with $a=1$.
\end{defn}

\begin{rmk}
    Working mod $p$,
    one recovers the matrix for $\varphi(\alpha^{-1})$ in \cite[Lemma 7]{BL}.
\end{rmk}

We now write down an explicit formula for $T$, which will be used repeatedly below. Let $\mathbb F_q$ for $q= p^f$ denote the residue field of $\mathbb Q_{p^f}$. For $\lambda \in \mathbb F_q$, denote by $[\lambda] \in \mathbb Z_{p^f}$ its
Teichm\"uller lift. Let $w = \left( \begin{smallmatrix} 0 & 1 \\ 1 & 0 \end{smallmatrix} \right) \in G$ denote the Weyl element.

\begin{lem} The action of $T$ on  $[g,\prod_{i=0}^{f-1}v_i(X_i,Y_i)]\in \mathrm{ind}_{KZ}^GV$ is given by    \begin{eqnarray*}
   T\left(\bigg[g,\prod_{i=0}^{f-1}v_i(X_i,Y_i)\bigg]\right)&=& \sum\limits_{\lambda \in \mathbb F_q}\bigg[g\left(\begin{smallmatrix}
       p &[\lambda]\\ 0&1
   \end{smallmatrix} \right),\prod_{i=0}^{f-1}v_i(X_i,(-[\lambda])^{p^i}X_i+pY_i)\bigg] +\bigg[g\left(\begin{smallmatrix}
       1 &0\\ 0&p
   \end{smallmatrix}\right),\prod_{i=0}^{f-1}v_i(pX_i,Y_i)\bigg].
\end{eqnarray*}
\end{lem}
\begin{proof}
    Decompose $KZ\alpha^{-1}KZ$ as follows
   \begin{equation*}
    KZ\alpha^{-1}KZ =\bigsqcup_{\lambda \in \mathbb{F}_q} \tiny{\begin{pmatrix}
        p&[\lambda]\\0&1
    \end{pmatrix}} KZ \sqcup \tiny{\begin{pmatrix}
      1&0\\0&p  
    \end{pmatrix} } KZ. 
   \end{equation*}
   In order to use \eqref{Heckeop}, we need to compute $\varphi(\left(\begin{smallmatrix}p&[\lambda]\\0&1\end{smallmatrix}\right)^{-1})$ and $\varphi(\left(\begin{smallmatrix}1&0\\0&p\end{smallmatrix}\right)^{-1})=\varphi(\alpha^{-1}).$ We observe that 
   $w\alpha^{-1}w\tiny{\begin{pmatrix}
       1&-[\lambda]\\0&1
   \end{pmatrix}}=\tiny{\begin{pmatrix}p&[\lambda]\\0&1\end{pmatrix}}^{-1}.$ Hence using Lemma~\ref{hecke operator} we have 
   \begin{eqnarray*}
       \varphi(\left(\begin{smallmatrix}p&[\lambda]\\0&1\end{smallmatrix}\right)^{-1})\prod_{i=0}^{f-1}v_i(X_i,Y_i)&=&w\varphi(\alpha^{-1})w\tiny{\begin{pmatrix}
       1&-[\lambda]\\0&1
   \end{pmatrix}}\prod_{i=0}^{f-1}v_i(X_i,Y_i)\\
   &=&w\varphi(\alpha^{-1})w\prod_{i=0}^{f-1}v_i(X_i,(-[\lambda])^{p^i}X_i+Y_i)\\
   &=&w\varphi(\alpha^{-1})\prod_{i=0}^{f-1}v_i(Y_i,(-[\lambda])^{p^i}Y_i+X_i)\\
   &=&w\prod_{i=0}^{f-1}v_i(Y_i,(-[\lambda])^{p^i}Y_i+pX_i)\\
&=&\prod_{i=0}^{f-1}v_i(X_i,(-[\lambda])^{p^i}X_i+pY_i).
   \end{eqnarray*}

Also, again by Lemma~\ref{hecke operator}, we have 
\begin{eqnarray*}
  \varphi(\left(\begin{smallmatrix}1&0\\0&p\end{smallmatrix}\right)^{-1})  \prod_{i=0}^{f-1}v_i(X_i,Y_i)=\prod_{i=0}^{f-1}v_i(pX_i,Y_i).
\end{eqnarray*}
The above calculations show that the  formula \eqref{Heckeop} for $T$ becomes 
\begin{eqnarray*}
   T\left(\bigg[g,\prod_{i=0}^{f-1}v_i(X_i,Y_i)\bigg]\right)&=& \sum\limits_{\lambda \in \mathbb F_q}\bigg[g\left(\begin{smallmatrix}
       p &[\lambda]\\ 0&1
   \end{smallmatrix}\right),\prod_{i=0}^{f-1}v_i(X_i,(-[\lambda])^{p^i}X_i+pY_i)\bigg] \\
   &&\qquad +\bigg[g\left(\begin{smallmatrix}
       1 &0\\ 0&p
   \end{smallmatrix}\right),\prod_{i=0}^{f-1}v_i(pX_i,Y_i)\bigg]. \qquad\qquad\qquad \qquad \qquad \qquad \qquad \qquad \qquad \qquad \qedhere
\end{eqnarray*}
\end{proof}

 More explicitly, we may write $$T = T^+ + T^-$$ 
 where if $v=\sum\limits_{0\le j_i\le r_i, }c_{j_0,...,j_{f-1}}\prod_{i=0}^{f-1}X_i^{r_i-j_i}Y_i^{j_i},$
then 
{\tiny \begin{eqnarray}
    \label{T+}
    T^+([g,v])\!\!&=&\!\!\sum\limits_{\lambda \in \mathbb F_q}\bigg[g\left(\begin{smallmatrix}
       p &[\lambda]\\ 0&1
   \end{smallmatrix}\right),\sum\limits_{0\le j_i \le r_i}p^{j_0+\cdots + j_{f-1}} \left( \sum\limits_{l_i\ge j_i}c_{l_0,...,l_{f-1}}(-[\lambda])^{l-j}\binom{l_0}{j_0}\cdots\binom{l_{f-1}}{j_{f-1}} \right) \prod_{i=0}^{f-1}X_i^{r_i-j_i}Y_i^{j_i}\bigg], \\\label{T-}
   T^-([g,v])\!\!&=&\!\!\sum\limits_{\lambda \in \mathbb F_q}\bigg[g\left(\begin{smallmatrix}
       1 &0\\ 0&p
\end{smallmatrix}\right),\sum\limits_{i=0}^{f-1}\sum\limits_{0\le j_i\le r_i, }p^{(r_0-j_0)+\cdots+(r_{f-1}-j_{f-1})}c_{j_0,...,j_{f-1}}\prod_{i=0}^{f-1}X_i^{r_i-j_i}Y_i^{j_i}\bigg],
   \end{eqnarray} }
where $l=\sum\limits_{i=0}^{f-1}l_ip^i$ and $j=\sum\limits_{i=0}^{f-1}j_ip^i.$

\section{The standard lattice in a locally algebraic representation }

Consider the locally algebraic representation of $G$
\begin{equation*}
\Pi_{k,a_p}:=\frac{\text{ind}_{KZ}^{G}(\otimes_{i=0}^{f-1}\text{Sym}^{r_i}\bar{\mathbb Q}_p^2)}{T-a_p},
\end{equation*}
where  $k = (k_0, k_1, \ldots, k_{f-1})$ with $k_i \geq 2$, $r_i=k_i-2$ and  $a_p \in \bar{\mathbb Q}_p$.  Here $\text{ind}_{KZ}^{G}$ is compact induction and $T$ is the Hecke operator defined in the previous section. Actually we work over a sufficiently large finite extension 
$E$ of $\mathbb Q_{p^f}$ containing $a_p$, though to simplify notation we sometimes write $\bar{\mathbb Q}_p$ instead of $E$.

Consider the $\bar{\mathbb Z}_p$-module $\Pi_{k,a_p}$ given by 
\begin{eqnarray}
  \Theta_{k,a_p}&:=&\text{image}~(\text{ind}_{KZ}^{G}(\otimes_{i=0}^{f-1}\text{Sym}^{r_i}\bar{\mathbb Z}_p^2)\rightarrow \Pi_{k,a_p})\notag \\
  &\cong&  \frac{\text{ind}_{KZ}^{G}(\otimes_{i=0}^{f-1}\text{Sym}^{r_i}\bar{\mathbb Z}_p^2)}{(T-a_p)(\text{ind}_{KZ}^{G}(\otimes_{i=0}^{f-1}\text{Sym}^{r_i}\bar{\mathbb Q}_p^2)) \cap \text{ind}_{KZ}^{G}(\otimes_{i=0}^{f-1}\text{Sym}^{r_i}\bar{\mathbb Z}_p^2)}.
\end{eqnarray}
As above, more precisely, we work over
$\mathcal O_E$, though often write $\bar{\mathbb Z}_{p}$ instead.

We denote by $\bar  \Theta_{k,a_p}$ the mod $p$ reduction of $ \Theta_{k,a_p}$, i.e., $\bar{\Theta}_{k,a_p}=\Theta_{k,a_p}\otimes_{\bar{\mathbb Z}_p}\bar{\mathbb F}_p$, or more precisely, $\Theta_{k,a_p}\otimes_{{\mathcal O}_E} {k}_E$ for $k_E$ the residue field of $E$.  Our goal is to compute the mod $p$ reduction of $ \Theta_{k,a_p}$. 

By the injectivity of the mod $p$ Local Langlands correspondence (and the
fact that it commutes with the $p$-adic Local Langlands correspondence) we know that $\bar{\Theta}_{k,a_p} \neq 0$ when $f = 1$. Indeed, it is the image of 
the mod $p$ reduction of a crystalline
representation $V_{k,a_p}$ which is
clearly nonzero. However, such 
a correspondence is not yet available when $f > 1$.  So one may ask 
whether $\bar{\Theta}_{k,a_p} \neq 0$
when $f > 1$.  This will be addressed
in the next section.

We remark that the 
non-vanishing of the mod $p$ reduction
of $\Theta_{k,a_p}$ is shown by showing 
that this module does not have a $\bar{\mathbb Q}_p$-line. This is equivalent to showing that
$\Theta_{k,a_p}$ is
a lattice (i.e., it is 
$\bar{\mathbb Z}_p$-free). When
$\Theta_{k,a_p}$ is a lattice one can define a $\mathrm{GL}_2({\mathbb Q}_{p^f})$-Banach space by taking the completion of $\Pi_{k,a_p}$ with
respect to this lattice. However, for now, this Banach
space will be  mostly auxiliary to this paper, and so we do not mention it further.

\section{$\bar{\Theta}_{k,a_p} $ is nonzero}\label{thetanon-zero}

In this section we show that
$\bar{\Theta}_{k,a_p} \neq 0$ at least
for small $f$ and small weights $k$. 

Indeed, in the
early days (before Colmez \cite{Col} made the 
Langlands correspondences functorial) Breuil \cite[Corollary 4.1.2]{breuilmathjassieu}  directly checked  that $\bar{\Theta}_{k,a_p}\neq 0$ 
in the case $f = 1$ for weights $k \leq 2p$ by proving that $\Theta_{k,a_p}$ does not contain a $\bar{\mathbb Q}_p$-line, at least if $p >2$. 

We  wish to similarly check  that for general $f \geq 1$ and $p > 2$ we have $\bar{\Theta}_{k,a_p}\neq 0$ if $k_i \le 2p$ for all $0 \leq i \leq f-1$. 
For simplicity, we eventually restrict to the case $f=2$, but 
our arguments here (which go back
to \cite{breuilmathjassieu}) should  work for general $f \geq 1$. See also \cite{Ies} and \cite{Ass}.

First we prove that $\bar{\Theta}_{k,a_p}\neq 0$ for $p \geq 2$ when $0 \leq r_i\leq p-1$ for every $i=0,1,\ldots,f-1$. 
We need the following lemma.

\begin{lem}
\label{smallweights}
    Let $p \geq 2$. Suppose that $0\le r_i\le p-1$ for all $0 \leq i \leq f-1$. Then, we have 
    \begin{eqnarray*}
        (T-a_p)(\mathrm{ind}_{KZ}^{G}(\otimes_{i=0}^{f-1}\mathrm{Sym}^{r_i}\bar{\mathbb Q}_p^2)) \cap \mathrm{ind}_{KZ}^{G}(\otimes_{i=0}^{f-1}\mathrm{Sym}^{r_i}\bar{\mathbb Z}_p^2)=(T-a_p)(\mathrm{ind}_{KZ}^{G}(\otimes_{i=0}^{f-1}\mathrm{Sym}^{r_i}\bar{\mathbb Z}_p^2)).
    \end{eqnarray*}
\end{lem}
\begin{proof}
    The inclusion $\supseteq$ is obvious. We now prove the other inclusion $\subseteq$. For that let $f\in \mathrm{ind}_{KZ}^{G}(\otimes_{i=0}^{f-1}\mathrm{Sym}^{r_i}\bar{\mathbb Q}^2_p)$ be such that $(T-a_p)f \in \mathrm{ind}_{KZ}^{G}(\otimes_{i=0}^{f-1}\mathrm{Sym}^{r_i}\bar{\mathbb Z}^2_p).$
Write $f=\sum\limits_{m=0}^M f_m$ where each $f_m$ is supported on radius $m$, for some $M \geq 0$. 
Here the radius means the distance
from the vertex $[\mathbb Z_{p^f}^2]$
in the positive direction of the tree, or the distance from the adjacent
vertex $[\alpha \mathbb Z_{p^f}^2]$ in 
the negative half of the tree. These two vertices corresponds to the matrices $g_{0,0}^0 = \mathrm{Id}$ 
and $g_{0,0}^1 = \alpha$ respectively
under the natural bijection between $KZ \backslash G$ and the vertices of the
tree. We first show that $f_M \in \mathrm{ind}_{KZ}^{G}(\otimes_{i=0}^{f-1}\mathrm{Sym}^{r_i}\bar{\mathbb Z}_p^2)$. Since $T^+(f_M)$ is the only function  in $(T-a_p)f$ supported on radius $M+1$, we have $$T^+f_M \in \mathrm{ind}_{KZ}^{G}(\otimes_{i=0}^{f-1}\mathrm{Sym}^{r_i}\bar{\mathbb Z}_p^2).$$
But any elementary function supported on radius $M$ is of the form $[g_{M,\mu}^0, v]$ or $[g_{M,\mu}^1, v]$. Since $\tiny{\begin{pmatrix}
    0&1\\p&0
\end{pmatrix}}g_{M,\mu}^0=g_{M,\mu}^1 w$, it is enough to show that $$T^+([g_{M,\mu}^0,v])\in \mathrm{ind}_{KZ}^{G}(\otimes_{i=0}^{f-1}\mathrm{Sym}^{r_i}\bar{\mathbb Z}_p^2) \text{ implies } v\in \otimes_{i=0}^{f-1}\mathrm{Sym}^{r_i}\bar{\mathbb Z}_p^2.$$ Let $v=\sum\limits_{0\le j_i\le r_i}c_{j_0,\ldots,j_{f-1}}X_0^{r_0-j_0}Y_0^{j_0}\cdots X_{f-1}^{r_{f-1}-j_{f-1}}Y_{f-1}^{j_{f-1}}.$
Using the formula \eqref{T+} for $T^+,$ we have 
\begin{eqnarray*}
   p^{j_0+\cdots+j_{f-1}} \sum\limits_{j_l\le i_l\le r_l}c_{i_0,\ldots, i_{f-1}}\binom{i_0}{j_0}\cdots\binom{i_{f-1}}{j_{f-1}} (-[\lambda])^{(i_0-j_0)+\cdots +p^{f-1}(i_{f-1}-j_{f-1})} \in \bar{\mathbb Z}_p
\end{eqnarray*}
for all $0 \leq j_l \leq r_l$ and for every  $\lambda \in \mathbb F_{p^f}.$ For $(j_0,\ldots, j_{f-1})=(0,\ldots,0),$ we have
\begin{eqnarray*}
     \sum\limits_{0\le i_l\le r_l}c_{i_0,\ldots, i_{f-1}} (-[\lambda])^{i_0+\cdots +p^{f-1}i_{f-1}} \in \bar{\mathbb Z}_p 
\end{eqnarray*}
for every $\lambda \in \mathbb F_{p^f}.$ Since the $r_i$'s are bounded by $p-1$ we see that 
the exponents of $(-[\lambda])$ are distinct and are at most $q-1$. Writing this system of equations in matrix form and observing that the coefficient matrix is a Vandermonde matrix (so its determinant is in $\bar{\mathbb Z}_p^\times$), we deduce that each $c_{i_0,\ldots, i_{f-1}}\in \bar{\mathbb Z}_p.$ Hence $v \in \otimes_{i=0}^{f-1}\mathrm{Sym}^{r_i}\bar{\mathbb Z}_p^2$ and so $[g_{M,\mu}^0, v] \in \mathrm{ind}^G_{KZ}(\otimes_{i=0}^{f-1}\mathrm{Sym}^{r_i}\bar{\mathbb Z}_p^2)$. So $f_M$ is integral. Now applying the same method to $f-f_{M},$ we observe that $f_{M-1}$ is integral. 
Continuing in this manner, we conclude that $f\in\mathrm{ind}^G_{KZ}(\otimes_{i=0}^{f-1}\mathrm{Sym}^{r_i}\bar{\mathbb Z}_p^2)$. 
\end{proof}

\begin{rmk}
  A more general version of this result can be found in \cite{Ies}
  who determines exactly when the equality in Lemma~\ref{smallweights} holds.
\end{rmk}
  
\begin{cor}
 \label{nonzeroleqp-1}
   Let $p \geq 2$ and $0\le r_i \le p-1$ for $0 \leq i \leq f-1$ and $v(a_p)>0.$ Then we have 
   \begin{eqnarray*}
       \bar{\Theta}_{k,a_p}\cong \frac{\mathrm{ind}^G_{KZ}(\otimes_{i=0}^{f-1}\mathrm{Sym}^{r_i}\bar{\mathbb F}_p^2)}{T}
   \end{eqnarray*}
    and hence $\bar{\Theta}_{k,a_p} \neq 0$.
\end{cor}
\begin{proof}
    By the previous lemma, we have 
    \begin{eqnarray*}
     \Theta_{k,a_p}=  \frac{\text{ind}_{KZ}^{G}(\otimes_{i=0}^{f-1}\text{Sym}^{r_i}\bar{\mathbb Z}_p^2)}{(T-a_p)(\text{ind}_{KZ}^{G}(\otimes_{i=0}^{f-1}\text{Sym}^{r_i}\bar{\mathbb Z}_p^2))}. 
    \end{eqnarray*}
    Tensoring with $\bar{\mathbb F}_p$, we get 
     \begin{eqnarray*}
     \bar{\Theta}_{k,a_p}&=&  \frac{\text{ind}_{KZ}^{G}(\otimes_{i=0}^{f-1}\text{Sym}^{r_i}\bar{\mathbb F}_p^2)}{(T-a_p)(\text{ind}_{KZ}^{G}(\otimes_{i=0}^{f-1}\text{Sym}^{r_i}\bar{\mathbb F}_p^2))}
      \> = \>  \frac{\text{ind}_{KZ}^{G}(\otimes_{i=0}^{f-1}\text{Sym}^{r_i}\bar{\mathbb F}_p^2)}{T(\text{ind}_{KZ}^{G}(\otimes_{i=0}^{f-1}\text{Sym}^{r_i}\bar{\mathbb F}_p^2))}. 
     \end{eqnarray*}
    where in the last equality we have used the fact that $v(a_p)>0$. The last quotient
    is well-known to be nonzero.  
      \end{proof}

 The corollary shows that 
 when $0 \leq r_i \leq p-1$ for all 
 $0 \leq i \leq f-1$, then $\bar{\Theta}_{k,a_p}$ is a universal supersingular module.
 Thus all its irreducible quotients are supercuspidal. Note that it is 
 well-known that this module
 is irreducible if $f = 1$ and has infinite length if $f > 1$. 
 
We now wish to show that   $\bar{\Theta}_{k,a_p} \neq 0$ when 
all $k_i \leq 2p$, at least if $p > 2$ and under the  
assumption that $\min{k_i} \leq p+2$. This was done
in \cite{breuilmathjassieu} for the case $f = 1$ (without the last assumption) and we follow the ideas in that paper. Since the proof is slightly {\it ad hoc} and involves treating various cases, for simplicity we shall assume that $f = 2$. The general case should be treatable in a similar fashion. We start with a useful lemma.

\begin{lem}\label{equationssolution}
    Let $p \geq 2$ and $0\le r_0,r_1\le 2p-2$. Assume that $\min({r_0,r_1)} \leq p$. For $l = 0,1$ and $0 \leq i_l \leq r_l$, let $c_{i_0,i_1} \in \bar{\mathbb Q}_{p}$ be such that 
    \begin{eqnarray*}
        \sum\limits_{\substack{0\le i_0\le r_0\\0\le i_1\le r_1}}c_{i_0,i_1}[\lambda]^{i_0+pi_1}\in p^n \bar{\mathbb Z}_p\\
        \sum\limits_{\substack{0\le i_0\le r_0\\0\le i_1\le r_1}} i_0 c_{i_0,i_1}[\lambda]^{i_0-1+pi_1}\in p^n \bar{\mathbb Z}_p\\
        \sum\limits_{\substack{0\le i_0\le r_0\\0\le i_1\le r_1}} i_1 c_{i_0,i_1}[\lambda]^{i_0+p(i_1-1)}\in p^n \bar{\mathbb Z}_p
    \end{eqnarray*}
    for all $\lambda \in \mathbb F_q$ and some fixed $n \geq 0$.
    Then each $c_{i_0,i_1}\in p^n \bar{\mathbb Z}_p.$
\end{lem}
\begin{proof}
    Assume that for some $(i_0,i_1)$, we have  $c_{i_0,i_1}\notin p^n\bar{\mathbb Z}_p.$ Let $m=\text{inf} \> \{ v(c_{i_0,i_1}) \}$. Then $m < n$. Divide the three equations by $p^m$. Setting $c'_{i_0,i_1} = c_{i_0,i_1}/p^m$,  
  and reducing these equations mod $p$, we obtain the following system 
  of equations over $\bar{\mathbb F}_p$:
    \begin{eqnarray*}
         \sum\limits_{\substack{0\le i_0\le r_0\\0\le i_1\le r_1}}\bar c'_{i_0,i_1}\lambda^{i_0+pi_1}&=&0\\
        \sum\limits_{\substack{0\le i_0\le r_0\\0\le i_1\le r_1}}i_0\bar c'_{i_0,i_1}\lambda^{i_0-1+pi_1}&=&0\\
        \sum\limits_{\substack{0\le i_0\le r_0\\0\le i_1\le r_1}}i_1\bar c'_{i_0,i_1}\lambda^{i_0+p(i_1-1)}&=&0,
        \end{eqnarray*}
        for every $\lambda \in \mathbb F_{p^2}$ with at least one  $\bar c'_{i_0,i_1}\neq 0$. Rearranging these equations, we have 
\begin{eqnarray}\label{eq1}
  \sum\limits_{i=p}^{r_0}\sum\limits_{k=1}^{r_1}(\bar c'_{i,k-1}+\bar c'_{i-p,k})\lambda^{i+(k-1)p} + \sideset{}{'}\sum\limits_{\substack{0\le i_0\le r_0\\0\le i_1\le r_1}}\bar c'_{i_0,i_1}\lambda^{i_0+pi_1}&=&0\\\label{eq2}
      \sum\limits_{i=p}^{r_0}\sum\limits_{k=1}^{r_1}(i\bar c'_{i,k-1}+(i-p)\bar c'_{i-p,k})\lambda^{i- 1+(k-1)p} + \sideset{}{'}\sum\limits_{\substack{0\le i_0\le r_0\\0\le i_1\le r_1}}i_0\bar c'_{i_0,i_1}\lambda^{i_0-1+pi_1}&=&0\\
         \sum\limits_{i=p}^{r_0}\sum\limits_{k=1}^{r_1}((k-1)\bar c'_{i,k-1}+k\bar c'_{i-p,k})\lambda^{i+(k-2)p} + \sideset{}{'}\sum\limits_{\substack{0\le i_0\le r_0\\0\le i_1\le r_1}}i_1\bar c'_{i_0,i_1}\lambda^{i_0+p(i_1-1)}&=&0.\label{eq3}
        \end{eqnarray}
 Using \eqref{eq1} and \eqref{eq2}, we observe that each $\lambda\in \mathbb F_{p^2}$ is a double root of the polynomial 
 $$\sum\limits_{i=p}^{r_0}\sum\limits_{k=1}^{r_1}(\bar c'_{i,k-1}+\bar c'_{i-p,k})x^{i+(k-1)p} + \sideset{}{'}\sum\limits_{\substack{0\le i_0\le r_0\\0\le i_1\le r_1}}\bar c'_{i_0,i_1}x^{i_0+pi_1}$$ 
 and this polynomial has degree at most $2p^2-2.$ Hence  we obtain 
 \begin{equation}\label{coesum1}
 \bar c'_{i,k-1}+\bar c'_{i-p,k}=0
 \end{equation}
 for $i=p,\ldots, r_0$, $k=1,\ldots, r_1$ and $\bar c'_{i_0,i_1}=0$ for other indices $(i_0,i_1).$ 
 
 We also claim that
\begin{equation}\label{coesum2}
(k-1)\bar c'_{i,k-1}+k\bar c'_{i-p,k} = 0
\end{equation}
for all $i = p, \ldots, r_0$, $k = 1, \ldots, r_1$. 
 From \eqref{eq3}, we observe that each $\lambda$ is a root of the polynomial 
$$\sum\limits_{i=p}^{r_0}\sum\limits_{k=1}^{r_1}((k-1)\bar c'_{i,k-1}+k\bar c'_{i-p,k})x^{i+(k-2)p}.$$

If this polynomial has degree $< p^2$, then \eqref{coesum2} holds immediately. This happens, for instance, if $r_1\le p$, since then the degree $\le 2p-2+(p-2)p=p^2-2$. 

If the degree of the above polynomial $\geq p^2$, we may write
$$\sum\limits_{i=p}^{r_0}\sum\limits_{k=1}^{r_1}((k-1)\bar c'_{i,k-1}+k\bar c'_{i-p,k})x^{i+(k-2)p}=(x^{p^2}-x)\left(\sum\limits_{l=0}^{r_0+pr_1-p^2-2p}d_lx^l\right),$$
where $d_l \in \bar{\mathbb F}_p$. By the previous paragraph and the hypothesis that
$\min(r_0, r_1) \leq p$, we may now assume that $r_0 \leq p$.

%

We claim that all the monomials in $x^{p^2}\sum\limits_{l=0}^{r_0+pr_1-p^2-2p}d_lx^l$ and $x\sum\limits_{l=0}^{r_0+pr_1-p^2-2p}d_lx^l$ have different degrees. Indeed, 
if
    $$p^2+l'=1+l''$$
    for some $0\le l',l''\le r_0+r_1p-p^2-2p$, then
    $$p^2-1=l''-l'\le r_0+pr_1-p^2-2p\le p+(2p-2)p-p^2-2p=p^2-3p,$$
    which is a contradiction. This prove the claim.

    We now prove that each $d_l=0$. Indeed if $d_l \neq 0$, then the  
    monomials $x^{p^2 +l}$ and
    $x^{1+l}$ survive on the right hand side and don't cancel each other by the previous paragraph. Comparing with the left hand side, we see that
    there are $(i,k)$ and $(i',k')$ such that 
    \begin{eqnarray*}
        i + (k -2)p=p^2+l \quad \text{ and }\quad i'+(k'-2)p=1+l. 
    \end{eqnarray*}
    Since $r_0\le p$, we have
    $i = i' = p$, and so we have 
    \begin{eqnarray*}
        p+(k-2)p=p^2+l \quad \text{ and }\quad p+(k'-2)p=1+l. 
    \end{eqnarray*}
    Subtracting, this implies that $p\mid p^2-1$,
    a contradiction.
    Hence each $d_l=0$ and so the polynomial on the right hand side, hence left hand side, vanishes. Thus we obtain \eqref{coesum2}
       for $i=p = r_0$, $k=1,\ldots, r_1$, 
     in this case as well.

Now \eqref{coesum1} and \eqref{coesum2} imply that $\bar c'_{i,k-1}=0=\bar c'_{i-p,k}$ for every $i=p,\ldots,r_0$ and $k=1,\ldots, r_1.$ This implies that all the $\bar c'_{i_0,i_1}$ (at the beginning of the proof) vanish which is a contradiction to the fact that one of the $\bar c'_{i_0,i_1}\neq 0.$ Hence $c_{i_0,i_1}\in p^n \bar{\mathbb Z}_p$ for every $(i_0,i_1).$
\end{proof}

\begin{rmk}
    The lemma is false without  the assumption that $\min(r_0,r_1) \leq p$.
    Take $p = 3$, $q = 9$, $r_0 = r_1 = 4$. For the following four indices $(i_0, i_1)$,
	let $c_{0,4} = 1,
    c_{4,0} = 1,
	c_{3,3} = -1,
	c_{1,1} = -1$ and    let all other $c_{i_0,i_1} = 0$. Then the three relations in the lemma hold for $n=1$, but 
    clearly none of the above four $c_{i_0,i_1}$
    are in $p\bar{\mathbb Z}_p$.
\end{rmk}

We now state the following lemma
for $f = 2$.

\begin{lem}\label{coefficientformula}
 Let $p \geq 2$. For each $m>0$ and  $\mu \in I_m$, let $$v_\mu^m=\sum\limits_{0\le j_i\le r_i}c^{m}_{j_0,j_1,\mu}X_0^{r_0-j_0}Y_0^{j_0}X_1^{r_1-j_1}Y_1^{j_1}\in \mathrm{Sym}^{r_0}\bar{\mathbb Q}_p^2\otimes \mathrm{Sym}^{r_1}\bar{\mathbb Q}_p^2$$ and let 
 \begin{eqnarray*}
    f_{m-1}&=&\sum\limits_{\mu \in I_{m-1}}[g^0_{m-1,\mu},v_\mu^{m-1}]\\
    f_m&=&\sum\limits_{\mu \in I_m}[g^0_{m,\mu},v_\mu^m]\\
    f_{m+1}&=&\sum\limits_{\mu \in I_{m+1}}[g^0_{m+1,\mu},v_\mu^{m+1}].\\    
 \end{eqnarray*}
 Then we have 
    \begin{eqnarray*}
        T^-(f_{m+1})+T^+(f_{m-1})-a_pf_m=\sum\limits_{\mu \in I_m}[g^0_{m,\mu},\sum\limits_{\substack{0\le j_0\le r_0\\0\le j_1\le r_1}}C^{m}_{j_0,j_1,\mu }X_0^{r_0-j_0}Y_0^{j_0}X_1^{r_1-j_1}Y_1^{j_1}],
    \end{eqnarray*}
    where 
    \begin{eqnarray*}
   C^{m}_{j_0,j_1,\mu}&=&\sum\limits_{\substack{j_0\le i_0\le r_0\\j_1\le i_1\le r_1}} p^{r_0-i_0+r_1-i_1} \binom{i_0}{j_0}\binom{i_1}{j_1}\sum\limits_{\lambda \in \mathbb F_{p^2}}   c_{i_0,i_1,\mu+p^m[\lambda]}^{m+1} [\lambda]^{(i_0-j_0)+p(i_1-j_1)}\\
   &&+p^{j_0+j_1} \sum\limits_{\substack{j_0\le i_0\le r_0\\j_1\le i_1\le r_1}} c_{i_0,i_1,[\mu]_{m-1}}^{m-1}\binom{i_0}{j_0}\binom{i_1}{j_1}\left(\frac{[\mu]_{m-1}-\mu}{p^{m-1}}\right)^{(i_0-j_0)+p(i_1-j_1)}-a_pc_{j_0,j_1,\mu}^{m}.
    \end{eqnarray*}
\end{lem}
\begin{proof}
    The proof follows from the explicit formulas for the Hecke operator 
    given in \eqref{T+} and \eqref{T-}.
\end{proof}

Let $B_N(\bar{\mathbb Q}_p)$ denote the space of functions supported at vertices whose distance is at most $N$ from the vertex $[\mathbb Z_{p^f}^2]$
in the positive direction of the tree or at most $N$ from the adjacent
vertex $[\alpha \mathbb Z_{p^f}^2]$ in 
the negative half of the tree.
\begin{thm}\label{latticetheorem1}
  Let $p>2$. Let $(r_0,r_1)\in \{0,1,\ldots, 2p-2\}^2$ with one of them strictly bigger  than $p-1$ and $\min(r_0,r_1) \leq p$, $a_p\in \bar{\mathbb Z}_p$ with $0<v(a_p)<1$ and $N\in \mathbb N.$ Then there exists a constant $c$ depending only on $a_p$ such that for every $n \in {\mathbb Z}$ and $f\in \mathrm{ind}_{KZ}^{G}(\mathrm{Sym}^{r_0}\bar{\mathbb Q}_p^2\otimes \mathrm{Sym}^{r_1}\bar{\mathbb Q}_p^2)$, we have
  \begin{eqnarray}
  \label{latticethm1}
   &&(T-a_p)f\in B_{N}(\bar{\mathbb Q}_p)+ p^n\mathrm{ind}_{KZ}^{G}(\mathrm{Sym}^{r_0}\bar{\mathbb Z}_p^2\otimes \mathrm{Sym}^{r_1}\bar{\mathbb Z}_p^2)\nonumber\\
   &&\qquad\implies f\in B_{N-1}(\bar{\mathbb Q}_p)+ p^{n-c}\mathrm{ind}_{KZ}^{G}(\mathrm{Sym}^{r_0}\bar{\mathbb Z}_p^2\otimes \mathrm{Sym}^{r_1}\bar{\mathbb Z}_p^2).
  \end{eqnarray}
\end{thm}

\begin{proof}
(i) First we assume that $(r_0,r_1)\neq (p-1,p)$, $(p,p-1)$. We may assume 
that $f=\sum\limits_{m=N}^{M}f_m$, for some $M$, where $f_m=\sum\limits_{\mu \in I_m}[g^0_{m,\mu},v_\mu^m]$
where the $v_\mu^m$ are as in Lemma~\ref{coefficientformula}.
Assume the following
claims hold for some 
$m > N$:
\begin{eqnarray*}
 c_{0,0}^{m+1}, c_{r_0,r_1}^{m+1}\in \frac{p^n}{a_p}\bar{\mathbb Z}_p, \sum\limits_{i_0+pi_1\equiv 1 \> {(p^2-1)}}c_{i_0,i_1}^{m+1}\in\frac{p^n}{a_p}\bar{\mathbb Z}_p, \sum\limits_{i_0+pi_1\equiv p \> {(p^2-1)}}c_{i_0,i_1}^{m+1}\in\frac{p^n}{a_p}\bar{\mathbb Z}_p, c_{i_0,i_1}^{m+1}\in \frac{p^{n-1}}{a_p}\bar{\mathbb Z}_p, \\
         c_{0,0}^{m}, c_{r_0,r_1}^{m}\in \frac{p^n}{a_p}\bar{\mathbb Z}_p, \sum\limits_{i_0+pi_1\equiv 1 \> {(p^2-1)}}c_{i_0,i_1}^{m}\in\frac{p^n}{a_p}\bar{\mathbb Z}_p, \sum\limits_{i_0+pi_1\equiv p \> {(p^2-1)}}c_{i_0,i_1}^{m}\in\frac{p^n}{a_p}\bar{\mathbb Z}_p, c_{i_0,i_1}^{m}\in \frac{p^{n-1}}{a_p}\bar{\mathbb Z}_p,
\end{eqnarray*}
where we drop $\mu$ from the notation. We shall prove that these claims also hold for $m-1$. Applying
this recursively (essentially) all
the way down to $m - 1 = N$, the last
claim in each line implies that \eqref{latticethm1} holds with 
$c = 1+v(a_p)$.

By assumption, for $m > N$ we have
\begin{equation}\label{mthcoeff}
T^-({f_{m+1}})+T^+(f_{m-1})-a_pf_m\in p^n\mathrm{ind}_{KZ}^{G}(\mathrm{Sym}^{r_0}\bar{\mathbb Z}_p^2\otimes \mathrm{Sym}^{r_1}\bar{\mathbb Z}_p^2).
\end{equation}
Using \eqref{mthcoeff} and \lemref{coefficientformula} for $C_{0,0}^m, C_{1,0}^m$ and $C_{0,1}^m$, for each $\lambda \in {\mathbb F}_{p^2}$ we have 
\begin{eqnarray*}
  \sum\limits_{\substack{0\le i_0\le r_0\\0\le i_1\le r_1}}c^{m-1}_{i_0,i_1}(-[\lambda])^{i_0+pi_1}\in \frac{p^n}{a_p}\bar{\mathbb Z}_p\\
  \sum\limits_{\substack{1\le i_0\le r_0\\0\le i_1\le r_1}}i_0c^{m-1}_{i_0,i_1}(-[\lambda])^{i_0-1+pi_1}\in p^{n-2}\bar{\mathbb Z}_p\\
  \sum\limits_{\substack{0\le i_0\le r_0\\1\le i_1\le r_1}}i_1c^{m-1}_{i_0,i_1}(-[\lambda])^{i_0+p(i_1-1)}\in p^{n-2}\bar{\mathbb Z}_p.
\end{eqnarray*}
Indeed, taking $(j_0,j_1) = (0,0)$ 
in Lemma~\ref{coefficientformula}, and noting that
$c^{m}_{0,0} \in \frac{p^n}{a_p} \bar{\mathbb Z}_p$ and
$p^{r_0-i_0+r_1-i_1}  c_{i_0,i_1}^{m+1} \in \frac{p^n}{a_p}\bar{\mathbb Z}_p$ (since $c^{m+1}_{i_0,i_1} \in \frac{p^{n-1}}{a_p} \bar{\mathbb Z}_p$ and $c^{m+1}_{r_0,r_1} \in \frac{p^n}{a_p} \bar{\mathbb Z}_p$),
we obtain the first
statement above. The second and third statements follow similarly by taking $(j_0, j_1) = (1,0)$ and $(0,1)$, respectively.

Putting $\lambda=0$ in the first statement we get $c_{0,0}^{m-1}\in \frac{p^n}{a_p}\bar{\mathbb Z}_p$,
proving the first claim for $m-1$. Also using the fact that $(-[\lambda])^{p^2}=-[\lambda]$ (since $p > 2$),  we have $\sum\limits_{i_0+pi_1\equiv 1 \> {(p^2-1)}}c_{i_0,i_1}^{m-1}\in\frac{p^n}{a_p}\bar{\mathbb Z}_p$ and $\sum\limits_{i_0+pi_1\equiv p \> {(p^2-1)}}c_{i_0,i_1}^{m-1}\in\frac{p^n}{a_p}\bar{\mathbb Z}_p$,
by a Vandermonde type argument, proving the third and fourth claims for $m-1$.

Also, by \lemref{equationssolution}, we obtain $c_{i_0,i_1}^{m-1}\in p^{n-2}\bar{\mathbb Z}_p.$
Now from the formula for $C^m_{j_0,j_1}$ in Lemma~\ref{coefficientformula} with $j_0+pj_1\equiv 1 \> (p^2-1)$, we have 
$$p^{j_0+j_1} \sum\limits_{\substack{1\le i_0\le r_0\\0\le i_1\le r_1}}\binom{i_0}{j_0}\binom{i_1}{j_1}c^{m-1}_{i_0,i_1}(-[\lambda])^{i_0-j_0+p(i_1-j_1)}-a_p c^m_{j_0,j_1}\in \frac{p^n}{a_p}\bar{\mathbb Z}_p.$$
If $(j_0,j_1)\neq (1,0)$, then $j_0+j_1\ge2$,   
so  $a_pc^m_{j_0,j_1}\in \frac{p^n}{a_p}\bar{\mathbb Z}_p.$ The relation $\sum\limits_{i_0+pi_1\equiv 1{(p^2-1)}}c_{i_0,i_1}^{m-1}\in\frac{p^n}{a_p}\bar{\mathbb Z}_p$ then implies $a_pc^{m}_{1,0}\in\frac{p^n}{a_p}\bar{\mathbb Z}_p.$ A similar calculation with $(j_0,j_1)$ running over the congruence class $p$ mod $(p^2-1)$ shows that $a_pc^{m}_{0,1}\in\frac{p^n}{a_p}\bar{\mathbb Z}_p.$ Using the formulas for  
$C_{1,0}^{m}$ and $C_{0,1}^{m}$ once again, but with these improvements, we get that 
\begin{eqnarray}
    \label{first relation}
    \sum\limits_{\substack{0\le i_0\le r_0\\0\le i_1\le r_1}}c^{m-1}_{i_0,i_1}(-[\lambda])^{i_0+pi_1}\in \frac{p^n}{a_p}\bar{\mathbb Z}_p\\
  \sum\limits_{\substack{1\le i_0\le r_0\\0\le i_1\le r_1}}i_0c^{m-1}_{i_0,i_1}(-[\lambda])^{i_0-1+pi_1}\in \frac{p^{n-1}}{a_p}\bar{\mathbb Z}_p \nonumber \\
  \sum\limits_{\substack{0\le i_0\le r_0\\1\le i_1\le r_1}}i_1c^{m-1}_{i_0,i_1}(-[\lambda])^{i_0+p(i_1-1)}\in \frac{p^{n-1}}{a_p}\bar{\mathbb Z}_p.
  \nonumber
\end{eqnarray}
Again using \lemref{equationssolution}, we have $c^{m-1}_{i_0,i_1}\in \frac{p^{n-1}}{a_p}\bar{\mathbb Z}_p$,
proving the fifth claim for $m-1$. 

Thus it remains to show
the second claim $c^{m-1}_{r_0.r_1}\in \frac{p^{n}}{a_p}\bar{\mathbb Z}_p.$ 
 Note that in the above calculation we did not use the fact that $c^m_{r_0,r_1}\in \frac{p^n}{a_p}\bar{\mathbb Z}_p$. Since for $m > N+1$
 $$T^-(f_m)+T^+(f_{m-2})-a_pf_{m-1}\in p^n \mathrm{ind}_{KZ}^{G}(\mathrm{Sym}^{r_0}\bar{\mathbb Z}_p^2\otimes \mathrm{Sym}^{r_1}\bar{\mathbb Z}_p^2),$$
 we obtain as above that
 \begin{eqnarray*}
     c_{0,0}^{m-2} \in \frac{p^n}{a_p}\bar{\mathbb Z}_p, \sum\limits_{i_0+pi_1\equiv 1 \> {(p^2-1)}}c_{i_0,i_1}^{m-2}\in\frac{p^n}{a_p}\bar{\mathbb Z}_p, \sum\limits_{i_0+pi_1\equiv p \> {(p^2-1)}}c_{i_0,i_1}^{m-2}\in\frac{p^n}{a_p}\bar{\mathbb Z}_p, c_{i_0,i_1}^{m-2}\in \frac{p^{n-1}}{a_p}\bar{\mathbb Z}_p.
 \end{eqnarray*}
 We use these relations to prove that $c_{r_0,r_1}^{m-1}\in \frac{p^{n}}{a_p}\bar{\mathbb Z}_p.$ For this we notice that 
 \eqref{first relation} implies that $c_{r_0,r_1}^{m-1}+c_{r_0-p+1,r_1-p+1}^{m-1}\in \frac{p^{n}}{a_p}\bar{\mathbb Z}_p$, 
 by a Vandermonde type argument. Hence it is enough to show that $c_{r_0-p+1,r_1-p+1}^{m-1}\in \frac{p^{n}}{a_p}\bar{\mathbb Z}_p.$

If $r_0,r_1\ge p$, or one of them is $p-1$ and the other is strictly bigger than $p$, then 
$p^{r_0-p+1+r_1-p+1}c^{m-2}_{i_0,i_1}\in p^n\bar{\mathbb Z}_p$.
 By the formula for $C_{r_0-p+1,r_1-p+1}^{m-1}$, we obtain that
{\tiny
\begin{eqnarray*} \sum\limits_{\substack{r_0-p+1\le i_0\le r_0\\r_1-p+1\le i_1\le r_1}} p^{r_0-i_0+r_1-i_1} \binom{i_0}{r_0-p+1}\binom{i_1}{r_1-p+1}\sum\limits_{\lambda \in \mathbb F_{p^2}}   c_{i_0,i_1}^{m} [\lambda]^{(i_0-r_0+p-1)+p(i_1-r_1+p-1)}-a_pc_{r_0-p+1,r_1-p+1}^{m-1}\in p^n \bar{\mathbb Z}_p.
\end{eqnarray*}}
\!\!Moreover, for such $(r_0, r_1)$, by Lucas's theorem, $p$ divides each of the following products of binomial coefficients  $$\binom{r_0}{r_0-p+1}\binom{r_1}{r_1-p+1},\binom{r_0-1}{r_0-p+1}\binom{r_1}{r_1-p+1},\binom{r_0}{r_0-p+1}\binom{r_1-1}{r_1-p+1}.$$
Since $c^m_{r_0,r_1}\in \frac{p^n}{a_p}\bar{\mathbb Z}_p$ and $c^m_{i_0,i_1}\in \frac{p^{n-1}}{a_p}\bar{\mathbb Z}_p$, we see that $a_pc^{m-1}_{r_0-p+1,r-1-p+1}\in p^n\bar{\mathbb Z}_p$,
as desired. This proves the second claim for $m - 1$ but only for $m - 1 > N$.

We now run the recursion
above till $m = N+2$ to
obtain that $c^{N+1}_{i_0,i_1} \in \frac{p^{n-1}}{a_p} \bar{\mathbb Z}_p$. 
Finally, running it 
partially one last time for $m = N+1$ shows 
that $c^{N}_{i_0,i_1} \in \frac{p^{n-1}}{a_p} \bar{\mathbb Z}_p$ also. 
\\

Now let's assume that $r_0\le p-2$ and $r_1\ge p$ (or vice versa). We run the recurrence as above except that this time the relation $c_{r_0,r_1}^{m-1}+c_{r_0-p+1,r_1-p+1}^{m-1}\in \frac{p^{n}}{a_p}\bar{\mathbb Z}_p$
collapses to 
$c_{r_0,r_1}^{m-1}\in \frac{p^{n}}{a_p}\bar{\mathbb Z}_p$ 
since $c_{r_0-p+1,r_1-p+1}^{m-1}=0$. Thus the
technicality in the last step above disappears.\\

(ii) We now prove the theorem for $(r_0,r_1)=(p-1,p)$. The case $(r_0, r_1) = (p, p-1)$ is similar. Write $f = \sum_{m=N}^M f_m$ as before. Assume that
the following seven claims hold for $m+1$ for some 
$m > N$
{\small \begin{eqnarray*}
 &&c_{i_0,i_1}^{m+1}\in \frac{p^n}{a_p}\bar{\mathbb Z}_p \text{ if } i_0+pi_1\not\equiv 1,\ldots,p \> (p^2-1),\sum\limits_{i_0+pi_1\equiv 1 \> {(p^2-1)}}c_{i_0,i_1}^{m+1}\in\frac{p^n}{a_p}\bar{\mathbb Z}_p,
 \sum\limits_{i_0+pi_1
 \equiv p \> {(p^2-1)}}c_{i_0,i_1}^{m+1}\in\frac{p^n}{a_p}\bar{\mathbb Z}_p, \\
 && c_{i_0,i_1}^{m+1}\in \frac{p^{n-1}}{a_p}\bar{\mathbb Z}_p, \sum\limits_{\lambda\in \mathbb F_{p^2}}c_{p-1,p,\mu +p^m[\lambda]}^{m+1}[\lambda]^p\in\frac{p^n}{a_p}\bar{\mathbb Z}_p, \sum\limits_{\lambda\in \mathbb F_{p^2}}c_{p-1,p,\mu +p^m[\lambda]}^{m+1}[\lambda]^{p-1}\in\frac{p^n}{a_p}\bar{\mathbb Z}_p\\
 &&\sum\limits_{\lambda\in \mathbb F_{p^2}}c_{p-1,p,\mu +p^m[\lambda]}^{m+1}\in\frac{p^n}{a_p}\bar{\mathbb Z}_p.
\end{eqnarray*} } 
\!\!Similarly assume these claims hold for $m$ as well. Using 
\begin{equation*}
T^-({f_{m+1}})+T^+(f_{m-1})-a_pf_m\in p^n\mathrm{ind}_{KZ}^{G}(\mathrm{Sym}^{r_0}\bar{\mathbb Z}_p^2\otimes \mathrm{Sym}^{r_1}\bar{\mathbb Z}_p^2),
\end{equation*}
and the formulas for $C^m_{0,0}$, $C^m_{1,0}$ and $C^m_{0,1}$ from Lemma~\ref{coefficientformula},   we obtain
\begin{eqnarray}\label{firsteq}
  \sum\limits_{\substack{0\le i_0\le r_0\\0\le i_1\le r_1}}c^{m-1}_{i_0,i_1}(-[\lambda])^{i_0+pi_1}\in \frac{p^n}{a_p}\bar{\mathbb Z}_p\\
  \label{secondeq}
  \sum\limits_{\substack{1\le i_0\le r_0\\0\le i_1\le r_1}}i_0c^{m-1}_{i_0,i_1}(-[\lambda])^{i_0-1+pi_1}\in p^{n-2}\bar{\mathbb Z}_p \\
  \label{thirdeq} \sum\limits_{\substack{0\le i_0\le r_0\\1\le i_1\le r_1}}i_1c^{m-1}_{i_0,i_1}(-[\lambda])^{i_0+p(i_1-1)}\in p^{n-2}\bar{\mathbb Z}_p. 
\end{eqnarray}
 Here we have used $\sum\limits_{\lambda\in \mathbb F_{p^2}}c_{p-1,p,\mu +p^m[\lambda]}^{m+1}[\lambda]^p\in\frac{p^n}{a_p}\bar{\mathbb Z}_p$ and $\sum\limits_{\lambda\in \mathbb F_{p^2}}c_{p-1,p,\mu +p^m[\lambda]}^{m+1}[\lambda]^{p-1}\in\frac{p^n}{a_p}\bar{\mathbb Z}_p$ to obtain \eqref{firsteq} and \eqref{secondeq}, respectively.
 
 Using the fact that $(-[\lambda])^{p^2}=-[\lambda]$ in \eqref{firsteq}, by a Vandermonde type argument we obtain that the first three claims for $m-1$ hold, namely
\begin{eqnarray*}
 c_{i_0,i_1}^{m-1}\in \frac{p^n}{a_p}\bar{\mathbb Z}_p \text{ if } i_0+pi_1\not\equiv 1,\ldots,p \> (p^2-1), \sum\limits_{i_0+pi_1\equiv 1 \> {(p^2-1)}}c_{i_0,i_1}^{m-1}\in\frac{p^n}{a_p}\bar{\mathbb Z}_p, 
\sum\limits_{i_0+pi_1\equiv p \> {(p^2-1)}}c_{i_0,i_1}^{m-1}\in\frac{p^n}{a_p}\bar{\mathbb Z}_p.
\end{eqnarray*}

 By  \lemref{equationssolution} applied to \eqref{firsteq}, \eqref{secondeq}, \eqref{thirdeq}, we have $c_{i_0,i_1}^{m-1}\in p^{n-2}\bar{\mathbb Z}_p$.
Now from the formula for $C^m_{0,p}$ in Lemma~\ref{coefficientformula}, we have 
$$p^{p} \sum\limits_{\substack{1\le i_0\le r_0\\0\le i_1\le r_1}}\binom{i_0}{j_0}\binom{i_1}{j_1}c^{m-1}_{i_0,i_1}(-[\lambda])^{i_0-j_0+p(i_1-j_1)}-a_p c^m_{0,p}\in \frac{p^n}{a_p}\bar{\mathbb Z}_p.$$
Again, we have used $\sum\limits_{\lambda\in \mathbb F_{p^2}}c_{p-1,p,\mu +p^m[\lambda]}^{m+1}[\lambda]^{p-1}\in\frac{p^n}{a_p}\bar{\mathbb Z}_p$. This implies $a_p c^m_{0,p}\in \frac{p^n}{a_p}\bar{\mathbb Z}_p$ since $c^{m-1}_{i_0,i_1}\in p^{n-2} \bar{\mathbb Z}_p$. Using $ c^m_{1,0}+c^m_{0,p}\in \frac{p^n}{a_p}\bar{\mathbb Z}_p$ we have $a_p c^m_{1,0}\in \frac{p^n}{a_p}\bar{\mathbb Z}_p$. A similar calculation using $C_{p-1,p}^m$ and $\sum\limits_{\lambda\in \mathbb F_{p^2}}c_{p-1,p,\mu +p^m[\lambda]}^{m+1}\in\frac{p^n}{a_p}\bar{\mathbb Z}_p$ shows that $a_p c^m_{0,1}\in \frac{p^n}{a_p}\bar{\mathbb Z}_p$. 

From the formulas for $C^m_{1,0}$ and $C^m_{0,1}$ from
Lemma~\ref{coefficientformula}, we obtain
\begin{eqnarray*}
  \sum\limits_{\substack{1\le i_0\le r_0\\0\le i_1\le r_1}}i_0c^{m-1}_{i_0,i_1}(-[\lambda])^{i_0-1+pi_1}\in \frac{p^{n-1}}{a_p}\bar{\mathbb Z}_p\\
  \sum\limits_{\substack{0\le i_0\le r_0\\1\le i_1\le r_1}}i_1c^{m-1}_{i_0,i_1}(-[\lambda])^{i_0+p(i_1-1)}\in \frac{p^{n-1}}{a_p}\bar{\mathbb Z}_p.
\end{eqnarray*}
Along with \eqref{firsteq}, we have $c_{i_0,i_1}^{m-1}\in \frac{p^{n-1}}{a_p}\bar{\mathbb Z}_p$
by  \lemref{equationssolution}. Thus the fourth claim holds for $m-1$. 

Note that in the above calculations, we did not use the last three claims for $m$. That is, the seven claims for $m+1$ plus the first four claims for $m$ imply the first four claims for $m-1$. Now assume $m-1>N$. Using the seven claims for $m$ and the just proved four claims for $m-1$ we can similarly prove the first four claims for $m-2$.
To complete the recursive step, it remains to show the last three claims for $m-1$ hold, namely that: 
\begin{eqnarray*}
&&\sum\limits_{\lambda\in \mathbb F_{p^2}}c_{p-1,p,\mu +p^{m-2}[\lambda]}^{m-1}[\lambda]^p\in\frac{p^n}{a_p}\bar{\mathbb Z}_p, \sum\limits_{\lambda\in \mathbb F_{p^2}}c_{p-1,p,\mu +p^{m-2}[\lambda]}^{m-1}[\lambda]^{p-1}\in\frac{p^n}{a_p}\bar{\mathbb Z}_p, \\
&&\sum\limits_{\lambda\in \mathbb F_{p^2}}c_{p-1,p,\mu +p^{m-2}[\lambda]}^{m-1}\in\frac{p^n}{a_p}\bar{\mathbb Z}_p \text{ for every }\mu \in I_{m-2}.
\end{eqnarray*}
 To do this, observe that $pc_{i_0,i_1}^{m-1}\binom{i_1}{1}\in p^n\bar{\mathbb Z}_p$ if $i_0+pi_1 \not\equiv 1,2,\ldots, p~ (p^2-1)$, by the first claim
 for $m-1$. If 
 $i_0+pi_1\equiv 1,2,\ldots, p~ (p^2-1)$, then $(i_0, i_1)$ must be of the form $(a,0)$, $(a-1,p)$ or $(0,1)$ so that again
 $pc_{i_0,i_1}^{m-1}\binom{i_1}{1}\in p^n\bar{\mathbb Z}_p$ by the fourth claim for $m-1$, if $(i_0,i_1)\neq (0,1)$.
 Similar observations hold for $m-2$ (this is where we need the just proved first
 four claims for $m-2$). By hypothesis, since $m - 1>N$, we have $C_{0,1}^{m}$, $C_{0,1}^{m-1}\in p^n \bar{\mathbb Z}_p$, so that 
 by Lemma~\ref{coefficientformula}, we have
 \begin{eqnarray}
   && p\sum\limits_{\lambda\in \mathbb F_{p^2}}c_{p-1,p,\mu +p^m[\lambda]}^{m+1}[\lambda]^{p^2-1}+p c_{0,1,[\mu]_{m-1} }^{m-1} -a_pc_{0,1,\mu }^{m}\in p^n\bar{\mathbb Z}_p \text{ for every }\mu \in I_m\\
  \label{Im'}  &&  p\sum\limits_{\lambda\in \mathbb F_{p^2}}c_{p-1,p,\mu +p^m[\lambda]}^{m}[\lambda]^{p^2-1}+p c_{0,1,[\mu]_{m-2} }^{m-2} -a_pc_{0,1,\mu }^{m-1}\in p^n\bar{\mathbb Z}_p \text{ for every }\mu \in I_{m-1}
\end{eqnarray}
which by the third claims for $m+1$ and $m$ can be written as
\begin{eqnarray}\label{Im}
   && p\sum\limits_{\lambda\in \mathbb F_{p^2}}c_{0,1,\mu +p^m[\lambda]}^{m+1}[\lambda]^{p^2-1}+p c_{0,1,[\mu]_{m-1} }^{m-1} -a_pc_{0,1,\mu }^{m}\in p^n\bar{\mathbb Z}_p \text{ for every }\mu \in I_m\\
  \label{Im'}  &&  p\sum\limits_{\lambda\in \mathbb F_{p^2}}c_{0,1,\mu +p^m[\lambda]}^{m}[\lambda]^{p^2-1}+p c_{0,1,[\mu]_{m-2} }^{m-2} -a_pc_{0,1,\mu }^{m-1}\in p^n\bar{\mathbb Z}_p \text{ for every }\mu \in I_{m-1}.
\end{eqnarray}
Assume that there exists $U_{m+1}\in \bar{\mathbb Z}_p^{\times}$ such that 
\begin{eqnarray}
    \label{Um+1}
p c_{0,1,[\mu]_{m} }^{m} -a_pU_{m+1}c_{0,1,\mu }^{m+1}\in p^n\bar{\mathbb Z}_p \text{ for every }\mu \in I_{m+1}.
\end{eqnarray}
This can be shown for $m=M$ using \eqref{Im} with $m = M+1$.
We show that there exists
$U_m \in \bar{\mathbb Z}_p^\times$ such that \eqref{Um+1} holds with
$m+1$ replaced by $m$.  Multiplying \eqref{Um+1} by $[\lambda]^{p^2-1}$ and summing over all $\lambda\in \mathbb F_{p^2}$, we have 
$$p(p^2-1) c_{0,1,[\mu]_{m} }^{m} -a_pU_{m+1}\sum\limits_{\lambda \in \mathbb F_{p^2}}c_{0,1,[\mu]_{m}+p^m[\lambda] }^{m+1}[\lambda]^{p^2-1}\in p^n\bar{\mathbb Z}_p.$$ 
Using the above equation and  \eqref{Im}, we have 
\begin{eqnarray*}
p^2(p^2-1)c_{0,1,\mu }^{m}+a_pU_{m+1}p c_{0,1,[\mu]_{m-1} }^{m-1} -a_p^2U_{m+1}c_{0,1,\mu }^{m}&\in& a_pp^n\bar{\mathbb Z}_p \text{ for every }\mu \in I_m.
\end{eqnarray*}
Hence
\begin{eqnarray*}
p c_{0,1,[\mu]_{m-1} }^{m-1}-a_p\left(1-\frac{p^2(p^2-1)}{a_p^2U_{m+1}}\right)c_{0,1,\mu}^m &\in& p^n\bar{\mathbb Z}_p \text{ for every }\mu \in I_m.
\end{eqnarray*}
Hence we can take $U_m=\left(1-\frac{p^2(p^2-1)}{a_p^2U_{m+1}}\right).$ Similarly, using \eqref{Im'}, and \eqref{Um+1} with $m+1$ replaced by $m$, we obtain $U_{m-1} \in \bar{\mathbb Z}_p^{\times}$ such that
$$p c_{0,1,[\mu]_{m-2} }^{m-2} -a_pU_{m-1}c_{0,1,[\mu]_{m-2}+p^{m-2}[\lambda] }^{m-1}\in p^n\bar{\mathbb Z}_p \text{ for every }\mu \in I_{m-1}.$$
Using the facts that $\sum\limits_{\lambda \in \mathbb F_{p^2}}1=p^2$, $\sum\limits_{\lambda \in \mathbb F_{p^2}}[\lambda]^{p}=0$ and $\sum\limits_{\lambda \in \mathbb F_{p^2}}[\lambda]^{p-1}=0$, we have 
\begin{eqnarray*}
  p^3 c_{0,1,[\mu]_{m-2} }^{m-2} -a_pU_{m-1}\sum\limits_{\lambda \in \mathbb F_{p^2}}c_{0,1,[\mu]_{m-2}+p^{m-2}[\lambda] }^{m-1}\in p^n\bar{\mathbb Z}_p \text{ for every }\mu \in I_{m-1}\\
   -a_pU_{m-1}\sum\limits_{\lambda \in \mathbb F_{p^2}}c_{0,1,[\mu]_{m-2}+p^{m-2}[\lambda] }^{m-1}[\lambda]^{p}\in p^n\bar{\mathbb Z}_p \text{ for every }\mu \in I_{m-1}\\
  -a_pU_{m-1}\sum\limits_{\lambda \in \mathbb F_{p^2}}c_{0,1,[\mu]_{m-2}+p^{m-2}[\lambda] }^{m-1}[\lambda]^{p-1}\in p^n\bar{\mathbb Z}_p \text{ for every }\mu \in I_{m-1}.
\end{eqnarray*}
By the fourth claim for $m-2$ and the third claim for $m-1$, namely $c_{0,1}^{m-1}+c_{p-1,p}^{m-1}\in \frac{p^n}{a_p}\bar{\mathbb Z}_p$, we obtain the last three claims for $m-1$. 

We now run the recursion
above till $m = N+2$ to
obtain that $c^{N+1}_{i_0,i_1} \in \frac{p^{n-1}}{a_p} \bar{\mathbb Z}_p$. 
Finally, running it 
partially one last time for $m = N+1$ shows 
that $c^{N}_{i_0,i_1} \in \frac{p^{n-1}}{a_p} \bar{\mathbb Z}_p$ also.
This proves the theorem
with $c = 1+ v(a_p)$ in
case (ii) as well.
\end{proof}

\begin{cor}
    Let $p\geq 2$. Let $(r_0,r_1)\in \{0,1,\ldots, 2p-2\}^2$ with $\min(r_0,r_1) \leq p$  and $ a_p\in \bar{\mathbb Z}_p$ with $0<v(a_p)<1$.
    Assume that both $r_i \neq 2p-2$ if $p = 2$. Then $\Theta_{k,a_p}$ does not contain a $\bar{\mathbb Q}_p$-line.
\end{cor}
\begin{proof}
    We may assume that
    one of the $r_i$ is strictly bigger than $p-1$ by Corollary~\ref{nonzeroleqp-1}.
    By hypothesis, this implies $p >2$. 
    We work over a sufficiently large finite extension $E$ of $\mathbb Q_{p^2}$ rather than  $\bar{\mathbb Q}_p$  and over the ring of integers $\mathcal O_E$ of $E$ rather than $\bar{\mathbb Z}_p$. We prove that $\Theta_{k,a_p}$ does not contain an $E$-line. Let us assume that $\Theta_{k,a_p}$ contains an $E$-line. This implies that there exists an $f \in \mathrm{ind}_{KZ}^{G}(\mathrm{Sym}^{r_0}\mathcal O_E^2\otimes \mathrm{Sym}^{r_1}\mathcal O_E^2)$ such that for every $n\ge0$ there exists $f_n' \in \mathrm{ind}_{KZ}^{G}(\mathrm{Sym}^{r_0} E^2\otimes \mathrm{Sym}^{r_1} E^2)$ satisfying 
    \begin{eqnarray*}
        \frac{1}{p^n}f-(T-a_p)f_n'\in \mathrm{ind}_{KZ}^{G}(\mathrm{Sym}^{r_0}\mathcal O_E^2\otimes \mathrm{Sym}^{r_1}\mathcal O_E^2).
    \end{eqnarray*}
    Thus, for each $n \geq 0$, we have
    \begin{eqnarray*}
        f-(T-a_p)f_n\in p^n\mathrm{ind}_{KZ}^{G}(\mathrm{Sym}^{r_0}\mathcal O_E^2\otimes \mathrm{Sym}^{r_1}\mathcal O_E^2),
    \end{eqnarray*}
    where $f_n=p^n f_n'.$ If $f$ is supported on radius at most $N$, then 
    $$(T-a_p)f_n\in B_N(E)+p^n\mathrm{ind}_{KZ}^{G}(\mathrm{Sym}^{r_0}\mathcal O_E^2\otimes \mathrm{Sym}^{r_1}\mathcal O_E^2).$$
    By \thmref{latticetheorem1} (over $E$), we obtain that $f_n\in B_{N-1}(E)+p^{n-c}\mathrm{ind}_{KZ}^{G}(\mathrm{Sym}^{r_0}\mathcal O_E^2\otimes \mathrm{Sym}^{r_1}\mathcal O_E^2)$ for some $c$ independent of  $n$.  
    Thus we have 
    \begin{eqnarray*}
        f\in (T-a_p)(B_{N-1}(E))+p^{n-c}\mathrm{ind}_{KZ}^{G}(\mathrm{Sym}^{r_0}\mathcal O_E^2\otimes \mathrm{Sym}^{r_1}\mathcal O_E^2)~~~~~~~~~\text{ for every } n\ge 0.
    \end{eqnarray*}
    Since $(T-a_p)(B_{N-1}(E))$ is complete (being a finite dimensional vector space over $E$), we get that $f\in (T-a_p)(B_{N-1}(E)) \cap \mathrm{ind}_{KZ}^{G}(\mathrm{Sym}^{r_0}\mathcal O_E^2\otimes \mathrm{Sym}^{r_1}\mathcal O_E^2).$ Hence $f$ is zero in $\Theta_{k,a_p}$, a contradiction.
\end{proof}

\begin{lem}
Let $E$ be a finite extension of $\mathbb Q_{p^f}$ and $\mathcal O_E$ be the ring of integers of $E$. Define $$\Theta_{k,a_p}=\frac{\mathrm{ind}_{KZ}^{G}(\otimes_{i=0}^{f-1}\mathrm{Sym}^{r_i}\mathcal O_E^2)}{(T-a_p)(\mathrm{ind}_{KZ}^{G}(\otimes_{i=0}^{f-1}\mathrm{Sym}^{r_i}E^2)) \cap \mathrm{ind}_{KZ}^{G}(\otimes_{i=0}^{f-1}\mathrm{Sym}^{r_i}\mathcal O_E^2)}.$$ 
If $\Theta_{k,a_p}$ does not contain an $E$-line, then $\bar{\Theta}_{k,a_p} = \Theta_{k,a_p} \otimes_{\mathcal O_E} k_E$ is nonzero.
\end{lem}

\begin{proof}
    Let $\pi_E$ be a uniformizer of $E$. Suppose, towards a contradiction, that $\bar{\Theta}_{k,a_p}=0$.
Let $x\in \Theta_{k,a_p}$ be nonzero. Since $\Theta_{k,a_p}=\pi_E\Theta_{k,a_p}$, 
there exists $x_1\in \Theta_{k,a_p}$ such that $x = \pi_E x_1$. Iterating this, we get that for every $n\ge 1$ there exists an
$x_n\in \Theta_{k,a_p}$ such that $x = \pi_E^n x_n$. Equivalently,
$$\pi_E^{-n}x \in \Theta_{k,a_p}
\quad \text{for all } n\in \mathbb N.$$
Thus, $\Theta_{k,a_p}$ contains the $E$-line generated by $x$, a contradiction. Therefore $\bar{\Theta}_{k,a_p}\neq 0$.
\end{proof}

\begin{cor}
 Let $p \geq 2$. Let $(r_0,r_1)\in \{0,1,\ldots, 2p-2\}^2$ with $\min(r_0,r_1) \leq p$  and $ a_p\in \bar{\mathbb Z}_p$ with $0<v(a_p)<1$. Assume that $r_i \neq 2p-2$ if $p = 2$. Then $\bar\Theta_{k,a_p} \neq 0$.
\end{cor}

\begin{rmk}
  There is some overlap between our results in this section and those in \cite{Ass} and
  \cite{Ies}.
  However, the former  assumes that 
  (all but) one weight $r_i$ is zero, and when both (all) weights $r_i$ are nonzero, the latter assumes that all $r_i \leq p-1$. 
\end{rmk}  

\section{Modular representations of $\mathrm{GL}_2(\mathbb{F}_q)$ and Kernel $X_{k,a_p}$}
\label{kernel}

Let $p \geq 2$ in this section. 
Let $ V_{r}=\otimes_{i=0}^{f-1}\text{Sym}^{r_i}\bar{\mathbb F}_p^2 = (r_0,r_1, \ldots, r_{f-1})$ be the symmetric power representation of 
$\Gamma :=\mathrm{GL}_2(\mathbb{F}_q)$ (hence of $KZ$, with $p\in Z$ acting trivially). Thus $V_r=V$ for $R=\bar{\mathbb F}_p$. Clearly, there is a surjective map

\begin{equation}\label{surjectivemap}
 \text{ind}_{KZ}^{G}  V_{r} \twoheadrightarrow \bar{\Theta}_{k,a_p}. 
\end{equation}
Thus, to compute the reduction of $\Theta_{k,a_p},$ we need to study this map. Denote the kernel of this map by $X_{k,a_p}$. We begin by recalling some details on modular representations.

\subsection{Modular representations} Let $M$ be the set of all $2\times 2$ matrices over $\mathbb F_q.$ Then $M$ acts on $V_{r}$ via 
\begin{equation*}
    \tiny{\begin{pmatrix}
       a&b\\c&d 
    \end{pmatrix}}\cdot \smallprod_{i=0}^{f-1}v_i(X_i,Y_i)=\smallprod_{i=0}^{f-1}v_i(a^{p^i}X_i+c^{p^i}Y_i,b^{p^i}X_i+d^{p^i}Y_i).
\end{equation*}
This action makes $V_{r}$ into an $\bar{\mathbb F}_p[M]$-module. Let $N$ denote the set of all singular matrices in $M.$ We call an $\bar{\mathbb F}_p[M]$-submodule $W\subset V_r$ singular if $t\cdot W=0$ for all $t\in N$.

For $i=0,1,\ldots,f-1$, let $\theta_i=X_iY_{i-1}^p-Y_iX_{i-1}^p$ be the generalized Dickson polynomials (or Ghate-Jana polynomials) defined in \cite{GJ}. We follow the convention that $0-1=f-1$.
Let $$V_{r}^*=\langle\theta_0,\theta_1,\ldots, \theta_{f-1}\rangle$$ be
the module generated by the $\theta_i$ for $0 \leq i \leq f-1$.
Then $V_{r}^*$ is an 
$\bar{\mathbb F}_p[M]$-submodule of
$V_r$ since $M$ acts on $\theta_i$ via $D^{p^i}$, Indeed, if $\left(
\begin{smallmatrix}
a & b \\
c & d
\end{smallmatrix}
\right) \in M$, then  
\begin{eqnarray*}
  \left(
\begin{smallmatrix}
a & b \\
c & d
\end{smallmatrix}
\right)\cdot\theta_i&=&\left(
\begin{smallmatrix}
a & b \\
c & d
\end{smallmatrix}
\right)\cdot (X_i Y_{i-1}^p - Y_i X_{i-1}^p)\\
&=&(a^{p^{i}}X_i+c^{p^{i}}Y_i)(b^{p^{i-1}}X_{i-1}+d^{p^{i-1}}Y_{i-1})^p-(b^{p^{i}}X_i+d^{p^{i}}Y_i)(a^{p^{i-1}}X_{i-1}+c^{p^{i-1}}Y_{i-1})^p\\
&=&(a^{p^{i}}X_i+c^{p^{i}}Y_i)(b^{p^{i}}X_{i-1}^p+d^{p^{i}}Y_{i-1}^p)-(b^{p^{i}}X_i+d^{p^{i}}Y_i)(a^{p^{i}}X_{i-1}^p+c^{p^{i}}Y_{i-1}^p)\\
&=&X_iY_{i-1}^p(ad-bc)^{p^i}+Y_iX_{i-1}^p(bc-ad)^{p^i}\\
&=&(ad-bc)^{p^i}(X_iY_{i-1}^p-Y_iX_{i-1}^p) \\
&=& (ad-bc)^{p^i} \theta_i.
\end{eqnarray*}
The module $V_r^*$ was studied
by \cite[(4.1)]{G} in the case $f = 1$.
In general, we have:

\begin{lem}
Let $r_i\ge q$ for $i=0,1,\ldots, f-1$. The submodule $V_{r}^*$ is the maximal singular submodule of $V_{r}$ as an $\bar{\mathbb F}_p[M]$-module.
\end{lem} 
\begin{proof}
By the computation above,
we see that if   $\left(
\begin{smallmatrix}
a & b \\
c & d
\end{smallmatrix}
\right) \in N$, then $  \left(
\begin{smallmatrix}
a & b \\
c & d
\end{smallmatrix}
\right) \cdot \theta_i = 0$.
Hence, the submodule  $\langle \theta_0,...,\theta_{f-1} \rangle$ is contained in the maximal singular submodule of $V_r$. 

We now prove the reverse inclusion. 
By the work of Ghate and Jana (\cite{GJ}, Theorem $1.3)$, the quotient $V_{r}/V_{r}^*$ is isomorphic to a principal series of dimension $q+1$. In particular, the maximal singular submodule has codimension at most $q+1$ (since $V_r^*$ is contained in the maximal singular submodule). Therefore, it suffices to show that its codimension is at least $q+1$.

We claim that the monomials 
$$\left\{\smallprod_{i=0}^{f-1}Y_i^{r_i}, \smallprod_{i=0}^{f-1}{X_i^{r_i-j_i}Y_i^{j_i}}: 0\le j_i\le p-1\right\}$$ are linearly independent modulo the maximal singular submodule. Suppose, towards a contradiction, there exists a nonzero singular element in $V_r$ of the form
 $$v=\sum\limits_{i=0}^{f-1}\sum\limits_{j_i=0}^{p-1}a_{j_0,\ldots,j_{f-1} }\smallprod_{i}X_i^{r_i-j_i}Y_i^{j_i}+a_{r_0,\ldots ,r_{f-1}}\smallprod_{i=0}^{f-1}Y_i^{r_i}.$$
 As $v$ is a singular element, $t\cdot v=0$ for   $t\in N$. By considering the action of matrices $\left(
\begin{smallmatrix}
1 & 0 \\
0 & 0
\end{smallmatrix}
\right)$ and $\left(
\begin{smallmatrix}
0 & 0 \\
0 & 1
\end{smallmatrix}
\right)$ on $v$, we conclude $a_{0,\ldots,0}=0=a_{r_0,\ldots,r_{f-1}}$. Hence we can rewrite $v$ as 
$$v=\sum\limits_{1\le l\le p^f-1}^{}a_{l_0,\ldots,l_{f-1}}\smallprod_{i=0}^{f-1}X_i^{r_i-l_i}Y_i^{l_i},$$
where $l=\sum\limits_{i=0}^{f-1}l_ip^i$. Let $j_0,\ldots, j_{f-1}$ be such that $a_{j_0,\ldots, j_{f-1}}\neq0$ and $a_{l_0,\ldots, l_{f-1}}=0$ for $l=\sum\limits_{i=0}^{f-1}l_ip^i<j=\sum\limits_{i=0}^{f-1}j_ip^i$. For $0\neq b\in \mathbb F_q$, we have 
\begin{eqnarray*}
0=\left(
\begin{smallmatrix}
1 & b \\
0 & 0
\end{smallmatrix}
\right)\cdot v=\biggl(\sum\limits_{1\le l \le p^f-1}^{}a_{l_0,\ldots,l_{f-1}}b^{l}\biggl)X_0^{r_0}\cdots X_{f-1}^{r_{f-1}}.
\end{eqnarray*}
Hence, we have $0=\sum\limits_{j\le l\le p^f-1}^{}a_{l_0,...,l_{f-1}}b^{l}$  and therefore
$$-a_{j_0,...,j_{f-1}}=\sum\limits_{j<l\le p^f-1}^{}a_{l_0,...,l_{f-1}}b^{l-j}.$$
Summing the above equation for every $ b ~(\neq 0) \in \mathbb F_q,$ we obtain
\begin{equation*}
    -(q-1)a_{j_0,...,j_{f-1}}=\sum\limits_{j< l\le p^f-1}^{}a_{l_0,...,l_{f-1}} \sum\limits_{b\in \mathbb F^\times_q}b^{l-j}.
\end{equation*}
The inner sum vanishes for every $(l_0,...,l_{f-1})$, since $l-j \not\equiv 0~(q-1)$. Hence we get $a_{j_0,...,j_{f-1}}=0,$ which is a contradiction. It follows that the codimension of the maximal singular submodule in $V_{r}$ is at least $q+1.$ This completes the proof of the lemma.
\end{proof}

 \subsection{Kernel $X_{k,a_p}$} Let  $X_{r} \subset V_{r}$ be the $\Gamma$-submodule (hence a $KZ$-submodule) generated by $X_0^{r_0}\cdots X_{f-1}^{r_{f-1}}.$ We show that 
\begin{itemize}
    \item $v(a_p)>0\Rightarrow \text{ind}_{KZ}^{G} X_{r} \subseteq X_{k,a_p},$
    \item $v(a_p)<1\Rightarrow \text{ind}_{KZ}^{G} V_{r}^* \subseteq X_{k,a_p}.$
\end{itemize}
\begin{lem}\label{Xrinkernel}
Let $r_i\ge q$ for $i=0,1,\ldots, f-1$ and $v(a_p)>0$. Then we have $\textnormal{ind}_{KZ}^{G} X_{r} \subset X_{k,a_p}.$
\end{lem}
\begin{proof}
  Consider the elementary function $f_0=[\mathrm{Id}, v]$, where $$v=\smallprod_{i=0}^{f-1}Y_i^{r_i}-\smallprod_{i=0}^{f-1}X_i^{p-1}Y_i^{r_i-p+1}.$$ 
  Clearly, $a_pf_0\equiv 0$ mod $p$ since $v(a_p)> 0$. We compute $T^+f_0$ and $T^-f_0.$ By \eqref{T+}, we obtain
  \begin{eqnarray*}
   T^+f_0 &=&\sum\limits_{\lambda \in \mathbb F_q} [\left( \begin{smallmatrix}
       p&[\lambda]\\0&1
   \end{smallmatrix}\right), v(X_0,(-[\lambda])X_0+pY_0,\ldots,X_{f-1},(-[\lambda])^{p^{f-1}}X_{f-1}+pY_{f-1})].  
  \end{eqnarray*}
  First, we compute 
  \begin{eqnarray*}
 && v(X_0,(-[\lambda])X_0+pY_0,\ldots,X_{f-1},(-[\lambda])^{p^{f-1}}X_{f-1}+pY_{f-1})  \\   &&\quad=\smallprod_{i=0}^{f-1}((-[\lambda])^{p^i}X_i+pY_i)^{r_i}-\smallprod_{i=0}^{f-1}X_i^{p-1}((-[\lambda])^{p^i}X_i+pY_i)^{r_i-p+1}\\
    && \quad=((-[\lambda])^r)X_0^{r_0}\cdots X_{f-1}^{r_{f-1}}-((-[\lambda])^{r-(q-1)})X_0^{r_0}\cdots X_{f-1}^{r_{f-1}}+O(p).
  \end{eqnarray*}
Since $(-[\lambda])^{r}\equiv(-[\lambda])^{r-(q-1)}$ mod $p$ ($r=\sum\limits_{i=0}^{f-1}r_ip^i \ge q$), we obtain $T^+f_0$ dies mod $p.$
Now we calculate $T^-f_0$ using \eqref{T-}. We have
\begin{eqnarray*}
  T^-f_0&=& [\left(\begin{smallmatrix}
      1&0\\0&p
  \end{smallmatrix}\right),v(pX_0,Y_0,\ldots,pX_{f-1},Y_{f-1})] \\
  &=& [\left(\begin{smallmatrix}
      1&0\\0&p
\end{smallmatrix}\right),Y_0^{r_0}\cdots Y_{f-1}^{r_{f-1}}]+O(p).
\end{eqnarray*}
Combining $T^+f_0$, $T^-f_0$ and $-a_pf_0$, we obtain
$$(T-a_p)f_0=[\left(\begin{smallmatrix}
      1&0\\0&p
\end{smallmatrix}\right),Y_0^{r_0}\cdots Y_{f-1}^{r_{f-1}}] \text{ mod }p.$$
This shows that $\mathrm {ind}^G_{KZ}X_{r}\subseteq X_{k,a_p}$.
\end{proof}
\begin{lem}\label{thetainkernel}
   Let $v(a_p)<1.$ Then we have $\textnormal{ind}^G_{KZ} V_{r}^* \subseteq X_{k,a_p}.$
\end{lem}
\begin{proof}
    Recall that $\theta_i=X_i Y_{i-1}^p - Y_i X_{i-1}^p$ for $i=0,\ldots, f-1$. Consider $\theta_iP_i \in V_{{r}}^*.$ Consider the function 
    $$f_i=\biggl[\mathrm{Id},\frac{\theta_iP_i}{a_p}\biggl].$$
    We first calculate 
    \begin{eqnarray*}
      &&\theta_i (X_{i-1},(-[\lambda])^{p^{i-1}}X_{i-1}+pY_{i-1},X_{i},(-[\lambda])^{p^{i}}X_{i}+pY_{i})\\
  &&\quad =X_i ((-[\lambda])^{p^{i-1}}X_{i-1}+pY_{i-1})^p - ((-[\lambda])^{p^i}X_i+pY_i) X_{i-1}^p \\
  && \quad =(-[\lambda])^{p^i}X_iX_{i-1}^p-(-[\lambda])^{p^i}X_iX_{i-1}^p+O(p).
    \end{eqnarray*}
 Therefore, $\frac{1}{a_p}\theta_i (X_{i-1},-[\lambda]^{p^{i-1}}X_{i-1}+pY_{i-1},X_{i},-[\lambda]^{p^{i}}X_{i}+pY_{i})$ dies mod $p$. Hence $T^+f_i$ dies mod $p$. Also,
    \begin{eqnarray*}
      T^-f_i&=&[\left(\begin{smallmatrix}
        1&0\\0&p
    \end{smallmatrix} \right), \frac{1}{a_p}(pX_i Y_{i-1}^p - Y_i p^pX_{i-1}^p) P_i(pX_{i-1},Y_{i-1},pX_{i},Y_i)]  \\
    &=&[\left(\begin{smallmatrix}
        1&0\\0&p
    \end{smallmatrix} \right), \frac{p}{a_p}(X_i Y_{i-1}^p - Y_i p^{p-1}X_{i-1}^p) P_i(pX_{i-1},Y_{i-1},pX_{i},Y_i)]
    \end{eqnarray*}
    which dies mod $p.$ 
    Combining $T^+f_0$, $T^-f_0$ and $-a_pf_0$, we have  
    $$(T-a_p)f_i=[\mathrm{Id}, -\theta_iP_i]\text{ mod }p.$$ 
  Therefore, we have
    $$\mathrm{ind}_{KZ}^G \> \langle \theta_i \rangle \subset X_{k,a_p}$$ for every $i$.
Hence, we obtain $\text{ind}^G_{KZ}V_{r}^*\subseteq X_{k,a_p}.$
\end{proof}

It follows by \lemref{Xrinkernel} and \lemref{thetainkernel} that when $v(a_p) \in (0,1)$, the map \eqref{surjectivemap}  induces a surjective map $ \text{ind}_{KZ}^{G}~  Q \twoheadrightarrow \bar{\Theta}_{k,a_p},$ where 
\begin{equation*}
    Q=\frac{V_{r}}{V^*_{r}+X_{r}}\cong \frac{V_r/V_r*}{X_r/(X_r\cap V_r^*)}.
\end{equation*}

\noindent The following lemma tells about the structure of $X_{r}/(X_{r}\cap V_{r}^*).$ In particular, we show that $X_{r}/(X_{r}\cap V_{r}^*)$ is irreducible, which will help to understand the structure of $Q.$ For $f=1,$ this was proved by \cite[(4.5)]{G}.
 \begin{lem}\label{Xr}
     Let $r_0+pr_1+\cdots +p^{f-1}r_{f-1}\equiv a= a_0+pa_1+\cdots + p^{f-1}a_{f-1} \textnormal{ mod }(q-1),$ where $1\le a\le q-1.$ Then $X_{r}/(X_{r}\cap V_{r}^*)$ is isomorphic to $V_{a}=(a_0,a_1,\ldots, a_{f-1}),$ which is irreducible. 
  \end{lem}
  \begin{proof}
     First, note that $\left\{\left(\begin{smallmatrix}
         1&\lambda \\0&1
     \end{smallmatrix}\right)Y_0^{r_0}\cdots Y_{f-1}^{r_{f-1}}\right\}_{\lambda \in \mathbb F_q}\cup\left\{X_0^{r_0}\cdots X_{f-1}^{r_{f-1}} \right\}$ is a generating set of $X_{r}.$ Observe that 
     \begin{eqnarray*}
       \left(
\begin{smallmatrix}
a & b \\
c & d
\end{smallmatrix}
\right)\cdot \smallprod_{i=0}^{f-1}(\lambda^{p^{i}} X_{i}+Y_{i})^{r_i} 
&=&  \smallprod_{i=0}^{f-1}(\lambda^{p^{i}} (a^{p^{i}}X_{i}+c^{p^{i}}Y_{i})+(b^{p^{i}}X_{i}+d^{p^{i}}Y_{i}))^{r_{i}}\\
&=&\smallprod_{i=0}^{f-1}((\lambda a+b)^{p^{i}} X_{i}+(\lambda c+d)^{p^{i}}Y_{i})^{r_{i}}\\
&=&\begin{cases}
    (\lambda a+b)^{r}X_0^{r_0}\cdots X_{f-1}^{r_{f-1}} &\text{if } \lambda c+d=0,\\
    (\lambda c+d)^{r}\prod\limits_{i=0}^{f-1}((\lambda a+b)(\lambda c+d)^{-1})^{p^{i}} X_{i}+Y_{i})^{r_{i}} &\text{otherwise.}\\   
\end{cases}
  \end{eqnarray*}
  Similarly 
\begin{eqnarray*}
      \left(
\begin{smallmatrix}
a & b \\
c & d
\end{smallmatrix}
\right)\cdot X_0^{r_0} \cdots X_{f-1}^{r_{f-1}} &=&  (aX_0+cY_0)^{r_0} \cdots(a^{p^{f-1}}X_{f-1}+c^{p^{f-1}}Y_{f-1})^{r_{f-1}}\\
&=&\begin{cases}
     a^{r}X_0^{r_0}\cdots X_{f-1}^{r_{f-1}} & \text{ if    }~ c=0,\\
    c^{r}(( a/c) X_0+Y_0)^{r_0}\cdots((a/c)^{p^{f-1}} X_{f-1}+Y_{f-1})^{r_{f-1}} & \text{ otherwise.    }\\   
\end{cases}
  \end{eqnarray*}

  Let $X$ be an $\overline{\mathbb F}_p$-space of dimension $q+1$ with basis elements $\{f_\lambda\}_{\lambda \in \mathbb F_q}\cup \{f_{r}\}.$ For $(r_0,\ldots,r_{f-1}),$ we define the map
  \begin{equation}
      \rho_{r}:X \rightarrow X_{r}
  \end{equation}
   by $ f_{r} \mapsto X_0^{r_0}\cdots X_{f-1}^{r_{f-1}}$ and $f_{\lambda} \mapsto  (\lambda X_0+Y_0)^{r_0}\cdots(\lambda^{p^{f-1}} X_{f-1}+Y_{f-1})^{r_{f-1}}.$

  For $(a_0,\ldots,a_{f-1}),$ we define a similar map 
  \begin{equation}
      \rho_{a}:X \rightarrow X_{a}
  \end{equation}
   by $ f_{r} \mapsto X_0^{a_0}\cdots X_1^{a_{f-1}}$ and $f_{\lambda} \mapsto  (\lambda X_0+Y_0)^{a_0}\cdots(\lambda^{p^{f-1}} X_{f-1}+Y_{f-1})^{a_{f-1}}$.
Using these maps we define an action of $M$ (set of $2 \times 2$ matrices over $\mathbb F_q$) on $X$ by defining an action on basis elements. Let $\left(
\begin{smallmatrix}
a & b \\
c & d
\end{smallmatrix}
\right)\in M,$ we define 
\begin{eqnarray*}
 \left(
\begin{smallmatrix}
a & b \\
c & d
\end{smallmatrix}
\right)\cdot f_{r}&=&\begin{cases}
    a^{r}f_{r} &\text{ if }~c=0,\\
    c^{r}f_{a/c} &\text{ otherwise.}
\end{cases}   
\end{eqnarray*}
Also define
\begin{eqnarray*}
 \left(
\begin{smallmatrix}
a & b \\
c & d
\end{smallmatrix}
\right)\cdot f_\lambda&=&\begin{cases}
     (\lambda a+b)^{r}f_{r} &\text{if}~ \lambda c+d=0,\\
   (\lambda c+d)^{r}f_{ (\lambda a+b)/(\lambda c+d)} &\text{otherwise.}
\end{cases}   
\end{eqnarray*}
 This action makes $X$ into an $\overline{\mathbb F}_p[M]$-module. Moreover, the map $\rho_{r}$ becomes an $ \overline{\mathbb F}_p[M]$-homomorphism.
  We check this for one basis element $f_{\lambda}$:
  \begin{eqnarray*}
\rho_{r}(\left(
\begin{smallmatrix}
a & b \\
c & d
\end{smallmatrix}
\right)\cdot f_{\lambda})&=& \begin{cases}
     (\lambda a+b)^{r}\rho_{r}(f_{r}) &\text{if}~ \lambda c+d=0,\\
   (\lambda c+d)^{r}\rho_{r}(f_{ (\lambda a+b)/(\lambda c+d)}) & \text{otherwise}
   \end{cases}\\
   &=& \begin{cases}
     (\lambda a+b)^{r}X_0^{r_0}\cdots X_{f-1}^{r_{f-1}} &\text{   if    }~ \lambda c+d=0,\\
   (\lambda c+d)^{r}\prod\limits_{i=0}^{f-1}(((\lambda a+b)(\lambda c+d)^{-1})^{p^{i}} X_{i}+Y_{i})^{r_{i}} &\text{   otherwise    }
   \end{cases}\\
   &=& \left(
\begin{smallmatrix}
a & b \\
c & d
\end{smallmatrix}
\right)\cdot (\lambda X_0+Y_0)^{r_0}\cdots(\lambda^{p^{f-1}} X_{f-1}+Y_{f-1})^{r_{f-1}}\\
&=&\left(
\begin{smallmatrix}
a & b \\
c & d
\end{smallmatrix}
\right)\cdot  \rho_{r}(f_{\lambda}).
  \end{eqnarray*}
  One can similarly check that the condition $r \equiv a \> (q-1)$ makes the map $\rho_{a}$ an $\overline{\mathbb F}_q[M]$-homomorphism.

Let $X_{r}^*= X_{r} \cap V_{r}^*.$  Now we claim that 
  \begin{equation}\label{kernelinclusion}
   \rho_{r}(\mathrm{ker}(\rho_{a})) \subseteq X_{r}^*,
  \end{equation}
i.e., we claim that 
\begin{equation*}
\rho_{a}\biggl(\sum\limits_{\lambda \in \mathbb F_q}c_\lambda f_\lambda+c_{r}f_{r}\biggl)=0
\end{equation*}
for some $c_\lambda$, $c_r\in \bar{\mathbb F}_p$ implies \begin{equation*}
    t\cdot \rho_{r}\biggl(\sum\limits_{\lambda \in \mathbb F_q}c_\lambda f_\lambda+c_{r}f_{r}\biggl)=0
\end{equation*}
for every $t \> (\neq 0)=\left(
\begin{smallmatrix}
a & b \\
c & d
\end{smallmatrix}
\right) \in N.$ Fix $t \in N$. Since the rank of $t$ is $1$ (when considered as a linear map on $V_1$), we have  $\mathrm{Im}(t)=\langle v_t \rangle$ for some $v_t = a_1X_0+a_2Y_0 \in V_1$. Up to multiplication by a scalar, we can assume that $v_t 
\in \{\lambda X_0+Y_0, X_0\}_{\lambda \in \mathbb F_q}$.  Hence for every $\lambda\in \mathbb F_q$, we have $t\cdot (\lambda X_{0}+Y_{0})=k_{\lambda} v_t$ 
and $t\cdot X_0=k_{00} v_t$  
for $k_\lambda$, $k_{00}\in \bar{\mathbb F}_p$.  Therefore $t\cdot  (\lambda X_{0}+Y_{0})^{r_0} \cdots (\lambda^{p^{f-1}} X_{f-1}+Y_{f-1})^{r_{f-1}}=k_\lambda^{r} P_t
$ and $t\cdot X_0^{r_0}\cdots X_{f-1}^{r_{f-1}}=k_{00}^{r}
P_t$, 
where 
$P_t$ 
is one of the elements from the generating set of $X_{r}.$ Using this observation, we have
\begin{eqnarray*}
     t\cdot \rho_{r}\biggl(\sum\limits_{\lambda \in \mathbb F_q}c_\lambda f_\lambda+c_{r}f_{r}\biggl)&=& \sum\limits_{\lambda \in \mathbb F_q}c_\lambda t\cdot \rho_{r}(f_\lambda)+c_{r}t\cdot \rho_{r}(f_{r})\\
     &=&\biggl(\sum\limits_{\lambda \in \mathbb F_q}c_\lambda k_{\lambda}^{r}+c_{r}k_{00}^{r}\biggl)
     P_t 
     \\
   &=&\biggl(\sum\limits_{\lambda \in \mathbb F_q}c_\lambda k_{\lambda}^{a}+c_{r}k_{00}^{a}\biggl)
   P_t,  
\end{eqnarray*}
where in the last line we have used that $r \equiv a \text{ mod }(q-1).$ A similar calculation shows that
$$0=t\cdot \rho_{a}\biggl(\sum\limits_{\lambda \in \mathbb F_q}c_\lambda f_\lambda+c_{r}f_{r}\biggl)=\biggl(\sum\limits_{\lambda \in \mathbb F_q}c_\lambda k_{\lambda}^{a}+c_{r}k_{00}^{a}\biggl)
P'_t,$$
where $P'_t$ 
is an element of the generating set of $X_a$. Hence, $\sum\limits_{\lambda \in \mathbb F_q}c_\lambda k_{\lambda}^{a}+c_{r}k_{00}^{a} =0.$ Therefore, we have
$$ t\cdot \rho_{r}\biggl(\sum\limits_{\lambda \in \mathbb F_q}c_\lambda f_\lambda+c_{r}f_{r}\biggl)=0.$$

Note that $V_{a}=X_{a}$ as $V_{a}$ is irreducible. Now we define a map 
 \begin{equation*}
     \rho:V_{a}\rightarrow X_{r}/X_{r}^*
 \end{equation*} in the following way: let $v\in V_{a},$  choose a pre-image $x$ in $X$ under the map $\rho_{a}$ and map it to $\rho_{r}(x)+X_{r}^*.$
 The condition $ \rho_{r}(\mathrm{ker}(\rho_{a})) \subseteq X_{r}^*$ insures that  the map $\rho$ is well defined.
 Also, it  
 is $\overline{\mathbb F}_p[M]$-linear since $\rho_{a}$ and $\rho_{r}$ are $\overline{\mathbb F}_p[M]$-linear. Also, $\rho$ is surjective because $\rho_{r}$ is. Since $V_{a}$ is irreducible, the map $\rho$ becomes injective and hence
\begin{eqnarray*}
\quad\quad\quad\qquad\qquad \qquad \qquad \qquad&& V_{a}\cong X_{r}/X_{r}^*. \qquad\qquad\qquad \qquad \qquad \qquad\qquad\quad\qedhere 
 \end{eqnarray*}
 \end{proof}

Now we prove a necessary and sufficient condition for an element $v\in V_{r}$ to be in $V_{r}^*.$ We use the evaluation map from \cite{GJ} (Theorem $2.18$) that 
induces a map between $V_{r}$ and a mod $p$   principal series representation of $\Gamma$. This lemma is a generalization of \cite[Lemma 2.7]{GR} to $f>1$.

\begin{lem}
    Let $$P(X_0,Y_0,\ldots,X_{f-1},Y_{f-1})=\sum\limits_{i_0=0}^{r_0}\cdots \sum\limits_{i_{f-1}=0}^{r_{f-1}}a_{i_0,\ldots, i_{f-1}}X_0^{r_0-i_0}Y_0^{i_0}\cdots X_{f-1}^{r_{f-1}-i_{f-1}}Y_{f-1}^{i_{f-1}}\in V_{r}.$$ Then $P\in V_{r}^*$ if and only if the following conditions are satisfied:
    \begin{enumerate}[label=(\roman*)]
        \item $a_{0,\ldots, 0}=0=a_{r_0,\ldots, r_{f-1}},$
        \item $\sum\limits_{i \equiv j\textnormal{ mod }(q-1)}a_{i_0,\ldots, i_{f-1}}=0$ for every $1 \le j\le q-1.$
    \end{enumerate}
    \end{lem}
    \begin{proof}
        By the work of \cite{GJ} (see Theorem $2.18$), we know that $P\in V_{r}^*$ if and only if for every $\left(
\begin{smallmatrix}
a & b \\
c & d
\end{smallmatrix}
\right) \in \Gamma$ we have $P(c,d,\ldots,c^{p^{f-1}},d^{p^{f-1}})=0$ . We use this fact to prove the lemma. Assume that $P\in V_{r}^*.$ Now choose the matrix $\left(
\begin{smallmatrix}
0 & 1 \\
1 & 0
\end{smallmatrix}
\right)\in \Gamma,$ to get 
$$0=P(1,0,\ldots,1,0)=a_{0,\ldots, 0}.$$
Similarly, choose the matrix $\left(
\begin{smallmatrix}
1 & 0 \\
0 & 1
\end{smallmatrix}
\right)\in \Gamma,$ to get
$$0=P(0,1,\ldots,0,1)=a_{r_0,\ldots,  r_{f-1}}.$$
Hence condition $(i)$ is proved.
Now consider a matrix $\left(
\begin{smallmatrix}
a & b \\
c & d
\end{smallmatrix}
\right)\in \Gamma$ with both $c$ and $d$ nonzero. We have
\begin{eqnarray*}
 0=P(c,d,\ldots,c^{p^{f-1}},d^{p^{f-1}})&=& \sum\limits_{i_0=0}^{r_0}\cdots \sum\limits_{i_{f-1}=0}^{r_{f-1}}a_{i_0,\ldots, i_{f-1}} c^{r_0-i_0}d^{i_0}\cdots c^{{p^{f-1}}(r_{f-1}-i_{f-1})}d^{{p^{f-1}}i_{f-1}}\\
 &=&\sum\limits_{i_0=0}^{r_0}\cdots \sum\limits_{i_{f-1}=0}^{r_{f-1}}a_{i_0,\ldots, i_{f-1}}c^{r-i}d^{i}\\
&=&c^{r}\sum\limits_{i_0=0}^{r_0}\cdots \sum\limits_{i_{f-1}=0}^{r_{f-1}}a_{i_0,\ldots, i_{f-1}}\left(\frac{d}{c}\right)^{i}.
\end{eqnarray*}
Replacing $d/c$ by $\lambda$ and breaking the sum in the following way:
\begin{equation}
    0=\sum\limits_{1\le j\le q-1}\lambda^{j_0+\cdots+p^{f-1}j_{f-1}}\sum\limits_{i_0+\cdots +p^{f-1}i_{f-1} \equiv j \text{ mod }(q-1)}a_{i_0,\ldots, i_{f-1}}.
\end{equation} 
Observe that as $(c,d)$ vary, $\lambda$ varies over $\mathbb F_q^\times$. Hence the polynomial 
$$F(x)=\sum\limits_{1\le j\le q-1}x^{j}\sum\limits_{i_0+\cdots +p^{f-1}i_{f-1} \equiv j \text{ mod }(q-1)}a_{i_0,\ldots, i_{f-1}}$$
of degree at most $q-1$ has $q-1$ roots because every $\lambda \in \mathbb F_q^\times$ is a root of $F$. Moreover, $F$ has no constant term, so $0$ is also a root of $F.$ But, $F$ being of degree at most $q-1,$ cannot have $q$ roots. Hence $F=0$ and we get 
$$\sum\limits_{i_0+\cdots +p^{f-1}i_{f-1} \equiv j \text{ mod }(q-1)}a_{i_0,\ldots, i_{f-1}}=0$$
for every $1\le j\le q-1$. This proves condition $(ii)$.

Now we prove the converse. Assume that $P\in V_{r}$ and conditions $(i),(ii)$ are satisfied. We need to show that $P(c,d,\ldots,c^{p^{f-1}},d^{p^{f-1}})=0$ for every $\left(
\begin{smallmatrix}
a & b \\
c & d
\end{smallmatrix}
\right) \in \Gamma.$ For the matrices with $d=0,$ we have
 $$P(c,0,\ldots,c^{p^{f-1}},0)=a_{0,\ldots,0}c^{r},$$
 which is  $0$ by assumption $(i)$. Similarly for the matrices in $\Gamma$ such that $c=0$, we have
$$P(0,d,\ldots,0,d^{p^{f-1}})=a_{r_0,\ldots, r_{f-1}}d^{r},$$
 which is again $0$ by assumption $(i)$.  Now, consider the matrices such that $c$ and $d$ both are nonzero. A calculation similar to the one above gives 
 \begin{eqnarray*}
 P(c,d,\ldots,c^{p^{f-1}},d^{p^{f-1}})
&=&
c^r \sum\limits_{1\le j\le q-1}\left(\frac{d}{c}\right)^{j}\sum\limits_{i_0+\cdots+p^{f-1}i_{f-1} \equiv j \text{ mod }(q-1)}a_{i_0,\ldots, i_{f-1}}.
\end{eqnarray*}
By condition $(ii)$ we get $P(c,d,\ldots,c^{p^{f-1}},d^{p^{f-1}})=0.$ Therefore, $P(c,d,\ldots,c^{p^{f-1}},d^{p^{f-1}})=0$ for every $\left(
\begin{smallmatrix}
a & b \\
c & d
\end{smallmatrix}
\right)\in \Gamma.$ Hence by \cite[Theorem 2.18]{GJ}, 
we have that $P\in V_{r}^*.$ 
  \end{proof}

\section{Sums of products of binomial coefficients} In this section, we prove some lemmas about sums of products of binomial coefficients which will prove useful when computing the reduction of $\Theta_{k,a_p}.$ 
In this section, we allow $p \geq 2$. The following lemma is a special case of the next lemma. See also Chitrao-Soneji \cite[Lemma 1]{CS}.

\begin{lem}\label{binomialsum}
   Let $r_0+\cdots +p^{f-1}r_{f-1}\equiv a \textnormal{ mod }(q-1),$ where $1\le a \le q-1.$ Define
   $$S_{r}= \sideset{}{'}\sum\limits_{\substack{0 \leq i_0 \leq r_0,\ldots,0 \leq i_{f-1} \leq r_{f-1} \\ i \equiv a \textnormal{ mod }(q-1)}}\!\!\!\! \binom{r_0}{i_0}\cdots \binom{r_{f-1}}{i_{f-1}},$$
   where the primed sum means we exclude the terms $(i_0,\ldots,i_{f-1})=(0,\ldots,0)$ and $(r_0,\ldots,r_{f-1}).$ Then $S_{r}\equiv 0\textnormal{ mod }p.$ 
   \end{lem}
\begin{proof}
Consider the sum
\begin{eqnarray}\label{2.4.1}
  \sum\limits_{\lambda \in \mathbb F_q}(1+[\lambda])^{r_0}\cdots(1+[\lambda]^{p^{f-1}})^{r_{f-1}}\!\!&=&\!\!\sum\limits_{\substack{0 \leq i_0 \leq r_0,\cdots,0 \leq i_{f-1} \leq r_{f-1}  \nonumber\\ }}\binom{r_0}{i_0}\cdots \binom{r_{f-1}}{i_{f-1}} \cdot \\
  && \quad\sum\limits_{\lambda \in \mathbb F_q}[\lambda]^{(r_0-i_0)+\cdots+p^{f-1}(r_{f-1}-i_{f-1})}  \nonumber\\
  \!\!&= &\!\!\begin{cases}
      q+(q-1)S_{r} &
 \text{if } a<q-1\\
 q+(q-1)S_{r}+q-1 & \text{if }a=q-1,
 \end{cases} 
\end{eqnarray}
where $q$ is the sum corresponding to  $(i_0,\ldots,i_{f-1})=(r_0,\cdots,r_{f-1}).$ When $a=q-1$, the extra $q-1$ term comes from summing $\lambda^{(r_0-i_0)+\cdots+p^{f-1}(r_{f-1}-i_{f-1})}$ for $(i_0,\ldots,i_{f-1})=(0,\ldots,0)$ (this sum vanishes when $a<q-1$). 
On the other hand, we have 
\begin{eqnarray}\label{2.4.2}
  \sum\limits_{\lambda \in \mathbb F_q}(1+[\lambda])^{r_0}\cdots(1+[\lambda]^{p^{f-1}})^{r_{f-1}}&\equiv&\sum\limits_{\lambda \in \mathbb F_q}(1+[\lambda])^{r_0+\cdots+p^{f-1}r_{f-1}} \text{ mod }p \notag\\
  &\equiv&\sum\limits_{\lambda \in \mathbb F_q}[1+\lambda]^{r_0+\cdots+p^{f-1}r_{f-1}} \text{ mod }p \notag\\
  &\equiv&\begin{cases}
      0 ~\text{mod } p &
 \text{if } a<q-1\\
 -1 \text{ mod } p & \text{if }a=q-1.
 \end{cases}
\end{eqnarray}
From \eqref{2.4.1}
 and \eqref{2.4.2}, we get $S_{r}\equiv 0 \text{ mod }p.$ 
 \end{proof}
  
 The following lemma generalizes \cite[Lemma 2.14]{GR}. To state it we fix some notation. For an integer $n$, let $[n]\in \{1,2,\dots,q-1\}$ be such that $n\equiv [n]\text{ mod }(q-1).$ 
 \begin{lem}\label{gbinomialsum}
     For $0\le m_i \le r_i,
     m=m_0+\cdots +p^{f-1}m_{f-1}$, $0 \leq m < r$,  $r\equiv a \textnormal{ mod }(q-1)$ with $1\le a=a_0+\cdots +p^{f-1}a_{f-1} \le q-1$, and $1\le b\le q-1$,  we have
     \begin{eqnarray*}
     S_{r,b,m}&=&\sum\limits_{\substack{0 \leq l_0 \leq r_0,\ldots,0 \leq l_{f-1} \leq r_{f-1} \\ l=l_0+\cdots +p^{f-1}l_{f-1} \equiv b \textnormal{ mod }(q-1)}}\binom{r_0}{l_0}\cdots \binom{r_{f-1}}{l_{f-1}}\binom{l_0}{m_0}\cdots\binom{l_{f-1}}{m_{f-1}} \\
     &\equiv&\binom{r_0}{m_0}\cdots\binom{r_{f-1}}{m_{f-1}}\biggl(\binom{[a-m]}{[b-m]}+\delta_{(p-1,\ldots,p-1),(b_0',\ldots,b_{f-1}')} \biggl)\textnormal{ mod }p,
     \end{eqnarray*}
     where $[b-m]=b_0'+\ldots+p^{f-1}b_{f-1}'.$ 
 \end{lem}
 \begin{proof}
     We rewrite the sum $S_{r,b,m}$ as 
     \begin{eqnarray}
     \label{ravis trick}
      S_{r,b,m} &=&  \sum\limits_{\substack{m_0 \leq l_0 \leq r_0,\ldots,m_{f-1} \leq l_{f-1} \leq r_{f-1} \\ l \equiv b \text{ mod }(q-1)}}\binom{r_0}{l_0}\cdots\binom{r_{f-1}}{l_{f-1}}\binom{l_0}{m_0}\cdots\binom{l_{f-1}}{m_{f-1}} \nonumber \\
      &=&\binom{r_0}{m_0}\cdots \binom{r_{f-1}}{m_{f-1}}\sum\limits_{\substack{m_0 \leq l_0 \leq r_0,\ldots,m_{f-1} \leq l_{f-1} \leq r_{f-1} \\ l\equiv b \text{ mod }(q-1)}}\binom{r_0-m_0}{l_0-m_0}\cdots\binom{r_{f-1}-m_{f-1}}{l_{f-1}-m_{f-1}} \nonumber \\
      &=&\binom{r_0}{m_0}\cdots \binom{r_{f-1}}{m_{f-1}}\sum\limits_{\substack{0 \leq l_0' \leq r_0',\ldots,0 \leq l_{f-1}' \leq r_{f-1}' \\ l' \equiv [b-m] \text{ mod }(q-1)}}\binom{r_0'}{l_0'}\cdots\binom{r_{f-1}'}{l_{f-1}'} \nonumber \\
      &=&\binom{r_0}{m_0}\cdots\binom{r_{f-1}}{m_{f-1}} S_{r-m,[b-m],0}.
     \end{eqnarray}
     Now consider another sum 
     \begin{eqnarray}\label{binom1}
        T_{r-m,[b-m]}&=&\sum\limits_{\lambda \in \mathbb F_q}[\lambda]^{q-1-(b_0'+\cdots+p^{f-1}b_{f-1}')} (1+[\lambda])^{r_0-m_0}\cdots(1+[\lambda]^{p^{f-1}})^{r_{f-1}-m_{f-1}} \nonumber\\
        &=&\sum\limits_{\lambda \in \mathbb F_q^\times}[\lambda]^{q-1-(b_0'+\cdots+p^{f-1}b_{f-1}')} (1+[\lambda])^{r_0-m_0}\cdots(1+[\lambda]^{p^{f-1}})^{r_{f-1}-m_{f-1}}\nonumber\\
        && \quad +\delta_{(p-1,\ldots,p-1),(b_0',\ldots,b_{f-1}')} \nonumber\\        &=&\!\!\!\!\!\sum\limits_{\substack{0\le j_0 \le r_0-m_0,\ldots,\\0\le j_{f-1} \le r_{f-1}-m_{f-1}} }\!\!\!\!\!\!\!\!\binom{r_0-m_0}{j_0}\cdots\binom{r_{f-1}-m_{f-1}}{j_{f-1}}\!\!\sum\limits_{\lambda \in \mathbb F_q^\times} [\lambda]^{q-1+j_0+\cdots+p^{f-1}j_{f-1}-(b_0'+\cdots+p^{f-1}b_{f-1}')}\nonumber\\
        &&\quad +\delta_{(p-1,\ldots,p-1),(b_0',\ldots,b_{f-1}')}\nonumber\\
        &=&(q-1)S_{r-m,[b-m],0}+\delta_{(p-1,\ldots,p-1),(b_0',\ldots,b_{f-1}')}.
     \end{eqnarray}
     On the other hand,  we have
     \begin{eqnarray}\label{binom2}
        T_{r-m,[b-m]}&=& \sum\limits_{\lambda \in \mathbb F_q\setminus\{-1\}}[\lambda]^{q-1-(b_0'+\cdots+p^{f-1}b_{f-1}'}) (1+[\lambda])^{r_0-m_0}\cdots(1+[\lambda]^{p^{f-1}})^{r_{f-1}-m_{f-1}} \notag\\
    &\equiv&\sum\limits_{\lambda \in \mathbb F_q\setminus\{-1\}}[\lambda]^{q-1-(b_0'+\cdots+p^{f-1}b_{f-1}')} (1+[\lambda])^{r-m} \text{ mod } p \notag\\
        &\equiv& \sum\limits_{\lambda \in \mathbb F_q\setminus\{-1\}}[\lambda]^{q-1-(b_0'+\cdots+p^{f-1}b_{f-1}')} (1+[\lambda])^{a_0'+\cdots+p^{f-1}a_{f-1}'} \text{ mod } p \notag\\
        &\equiv& \sum\limits_{\substack{0\le j_0 \le a_0',\ldots,\\ 0\le j_{f-1} \le a_{f-1}' }}\binom{a_0'}{j_0}\cdots  \binom{a_{f-1}'}{j_{f-1}}\sum\limits_{\lambda \in \mathbb F_q\setminus\{-1\}} [\lambda]^{q-1+j_0+\cdots+p^{f-1}j_{f-1}-(b_0'+\cdots+p^{f-1}b_{f-1}')}\text{ mod }p \notag\\
        &\equiv&\sum\limits_{\substack{0\le j_0 \le a_0',\ldots,\\ 0\le j_{f-1} \le a_{f-1}' }}\binom{a_0'}{j_0}\cdots \binom{a_{f-1}'}{j_{f-1}}\sum\limits_{\lambda \in \mathbb F_q} [\lambda]^{q-1+j_0+\cdots+p^{f-1}j_{f-1}-(b_0'+\cdots+p^{f-1}b_{f-1}')} \notag\\
        &&-\sum\limits_{\substack{0\le j_0 \le a_0',\ldots,\\ 0\le j_{f-1} \le a_{f-1}' }}\binom{a_0'}{j_0}\cdots \binom{a_{f-1}'}{j_{f-1}} (-1)^{q-1+j_0+\cdots+p^{f-1}j_{f-1}-(b_0'+\cdots+p^{f-1}b_{f-1}')}\text{ mod }p \notag\\
        &\equiv&-\binom{a_0'}{b_0'}\cdots\binom{a_{f-1}'}{b_{f-1}'}-(-1)^{q-1-b_0'-\cdots-p^{f-1}b_{f-1}'}(1-1)^{a_0'}\cdots(1-1^{p^{f-1}})^{a_{f-1}'} \text{ mod }p\notag\\
        &\equiv&-\binom{a_0'}{b_0'}\cdots \binom{a_{f-1}'}{b_{f-1}'}\text{ mod }p \notag \\
        &\equiv&-\binom{[a-m]}{[b-m]} \text{ mod }p,
     \end{eqnarray}
     where $a_0'+\cdots +p^{f-1}a_{f-1}'=[r-m]$ and in the last line we have used Lucas's theorem. Now comparing \eqref{binom1} and \eqref{binom2}, we get
$$S_{r-m,[b-m],0}\equiv \binom{[a-m]}{[b-m]}+\delta_{(p-1,\ldots,p-1),(b_0',\ldots,b_{f-1}')}\text{ mod }p.$$

\noindent Plugging this
congruence into \eqref{ravis trick} yields the lemma.  
\end{proof}

The following lemma will be used to smoothen certain non-integral functions. 
It extends \cite[Lemma 7.1, 7.2, 7.3]{BG}, 
\cite[Lemma 2.11]{BGR}
to the case of general
$f$.

\begin{lem}\label{vanishingexistence}
    Let $t \ge 0$ and 
    $1 \leq c = c_0 + p c_1 + \cdots  + p^{f-1} c_{f-1} \leq q-1$. For $0 \leq i \leq f-1$, 
    let $0 \le m_i<p$, 
    $r_i > c_i + m_i(q-1)$. Let $\beta_{j_0,\ldots,j_{f-1}}$ for $0 \leq j_i \leq r_i$ be integers such that 
    \begin{enumerate}
        \item $\beta_{j_0,\ldots,j_{f-1}}=0$ if $j_0+\cdots +p^{f-1}j_{f-1}\not\equiv c \textnormal{ mod }(q-1)$, and 
        \item $\sum\limits_{\substack{0 \leq j_0 \leq r_0,\ldots,0 \leq j_{f-1} \leq r_{f-1} \\ j_0+\cdots +p^{f-1}j_{f-1} \equiv c \textnormal{ mod }(q-1)}}\beta_{j_0,\ldots,j_{f-1}}\binom{j_0}{l_0}\cdots\binom{j_{f-1}}{l_{f-1}}\equiv 0 \textnormal{ mod }p^t$ for 
        all $0 \leq l_i \leq m_i.$
    \end{enumerate}
    Then for every $n\ge0,$ $\exists$ integers $\alpha_{j_0,\ldots,j_{f-1}}$ s.t. 
    \begin{enumerate}
        \item $\alpha_{j_0,\ldots,j_{f-1}}=0$ if $j_0+\cdots+p^{f-1}j_{f-1}\not\equiv c \textnormal{ mod }(q-1)$ and $$\alpha_{j_0,\ldots,j_{f-1}} \equiv \beta_{j_0,\ldots,j_{f-1}}\textnormal{ mod }p^t$$ if $j_0+\cdots+p^{f-1}j_{f-1} \equiv c \textnormal{ mod } (q-1)$.
        \item  $\sum\limits_{\substack{0 \leq j_0 \leq r_0,\ldots,0 \leq j_{f-1} \leq r_{f-1} \\ j_0+\cdots +p^{f-1}j_{f-1} \equiv c \textnormal{ mod }(q-1)}} \alpha_{j_0,\ldots,j_{f-1}}\binom{j_0}{l_0}\cdots\binom{j_{f-1}}{l_{f-1}}\equiv 0 \textnormal{ mod }p^n$ for every $0 \leq l_i \leq m_i.$
    \end{enumerate}
\end{lem}
\begin{proof}
    If $t\ge n,$ we can just take $\alpha_{j_0,\ldots,j_{f-1}}=\beta_{j_0,\ldots,j_{f-1}}.$ So, we can assume $t<n.$ Let 
    $$S_{l_0,\ldots,l_{f-1}=}\sum\limits_{\substack{0 \leq j_0 \leq r_0,\ldots,0 \leq j_{f-1} \leq r_{f-1} \\ j_0+\cdots+p^{f-1}j_{f-1} \equiv c \textnormal{ mod }(q-1)}}\beta_{j_0,\ldots,j_{f-1}}\binom{j_0}{l_0}\cdots\binom{j_{f-1}}{l_{f-1}}.$$
    We set $$\alpha_{j_0,\ldots,j_{f-1}}=\begin{cases}
        \beta_{j_0,\ldots,j_{f-1}}+p^t\gamma_{j_0,\ldots,j_{f-1}}&\text{if}~~j_0+\cdots+p^{f-1}j_{f-1} \equiv c \text{ mod }(q-1)\\  0&\text{otherwise}
    \end{cases}$$
for some  $\gamma_{j_0,\ldots,j_{f-1}} \in {\mathbb Z}$. It is enough to solve the following system of equations
for $\gamma_{j_0,\ldots,j_{f-1}}$:
\begin{equation}\label{vanishingeqn}
  \sum\limits_{\substack{0 \leq j_0 \leq r_0,\ldots,0 \leq j_{f-1} \leq r_{f-1} \\ j_0+\cdots+p^{f-1}j_{f-1} \equiv c \textnormal{ mod }(q-1)}}\gamma_{j_0,\ldots,j_{f-1}}\binom{j_0}{l_0}\cdots\binom{j_{f-1}}{l_{f-1}} \equiv -p^{-t}S_{l_0,\ldots,l_{f-1}} \text{ mod } p^{n-t}
\end{equation}
where there is one equation for each tuple of integers $0 \leq l_i \leq m_i$.
To get a square matrix, we solve \eqref{vanishingeqn} for certain special
tuples $(j_0,\ldots,j_{f-1})$
and set $\gamma_{j_0,\ldots,j_{f-1}} = 0$ outside
this special set.
Write $c= c_0 + c_1 p + \cdots + c_{f-1} p^{f-1}$. We take
\begin{eqnarray*}
  j_{k_0}&=&c_0+k_0(q-1) \text{ for } 0\le k_0 \le m_0\\&\vdots&\\
    j_{k_{f-1}}&=& c_{f-1} + k_{f-1}(q-1) \text{ for } 0\le k_{f-1} \le m_{f-1}.
\end{eqnarray*}
We solve 
$$\sum\limits_{\substack{0 \leq k_0 \leq m_0,\ldots,0 \leq k_{f-1} \leq m_{f-1} \\ }}\gamma_{j_{k_0},\ldots,j_{k_{f-1}}}\binom{c_0 +k_0(q-1)}{l_0}\cdots\binom{c_{f-1}+ k_{f-1}(q-1)}{l_{f-1}}\equiv -p^{-t}S_{l_0,\ldots,l_{f-1}}\text{ mod }p^{n-t}.$$ 
Write $\gamma_{k_0,\dots,k_{f-1}}$ in place of $\gamma_{j_{k_0},\ldots,j_{k_{f-1}}}$ and arrange these variables in lexicographic order. Set
$$\vec{X}=\begin{bmatrix}
\gamma_{0,\ldots,0} \\
\vdots \\
\gamma_{0,\ldots,0,m_{f-1}}\\
\gamma_{0,\ldots,0,1,0}\\
\vdots
\\
\gamma_{m_0,\ldots,m_{f-1}}
\end{bmatrix},~ \vec{B}=\begin{bmatrix}
-p^{-t}S_{0,\ldots,0} \\
\vdots \\
-p^{-t}S_{0,\ldots,0,m_{f-1}}\\
-p^{-t}S_{0,\ldots,0,1,0}\\
\vdots
\\
-p^{-t}S_{m_0,\ldots,m_{f-1}}
\end{bmatrix}$$ and 
$$A_0=\left( \begin{matrix}
    \binom{c_0+k_0(q-1)}{l_0}
\end{matrix}\right)_{0\le k_0,l_0 \le m_0},
\ldots, A_{f-1}=\left( \begin{matrix}
    \binom{c_{f-1} + k_{f-1}(q-1)}{l_{f-1}}
\end{matrix}\right)_{0\le k_{f-1},l_{f-1} \le m_{f-1}}.$$
Set ${A}=A_0 \otimes A_1\otimes \cdots\otimes A_{f-1}$. We want to show that $ A\cdot \vec X=\vec B$ has a  solution mod $p^{n-t}$. It suffices to show that $\det A=\det A_0^{m_0}\times \det A_1^{m_1}\times \cdots \times \det A_{f-1}^{m_{f-1}}\neq 0 \text{ mod } p$. But by \cite[Lemma 2.10]{BGR}, 
we have $\det A_i\neq 0 \text{ mod } p$ for $i=0,1,\ldots,f-1$. 
Hence the lemma is proved.
\end{proof}

\begin{rmk}
  \label{extreme values}
  It follows from the proof
  that $\alpha_{0,\ldots, 0} = \beta_{0,\ldots,0}$  and $\alpha_{r_0,\ldots, r_{f-1}} = \beta_{r_0,\ldots,r_{f-1}}$ since some
  $c_i \geq 1$ and each $r_i > c_i + m_i(q-1)$, so $\gamma_{0,\ldots, 0} = \gamma_{r_0,\ldots,r_{f-1}} = 0$.
\end{rmk}

\section{Weights $r\equiv 0 \textnormal{ mod }(q-1)$}

We can now start
computing the reduction
of $\Theta_{k,a_p}$
for $0 < v(a_p) < 1$.
Recall $r_i = k_i-2$.

In this section, we assume that $r = r_0+pr_1+\cdots +p^{f-1}r_{f-1}\equiv 0 \textnormal{ mod }(q-1).$ Recall that (see Section~\ref{kernel}) the map \eqref{surjectivemap} factors through $$\text{ind}^G_{KZ}  Q
\twoheadrightarrow \bar{\Theta}_{k,a_p}.$$ Ghate and Jana (\cite{GJ}, Theorem $2.18)$ showed that $V_{r}/V_{r}^*$ is the principal series 
$\mathrm{ind}_B^\Gamma d^r = \mathrm{ind}_B^\Gamma 1$ where
$B$ is the Borel subgroup of $\Gamma$ of upper-triangular matrices and $d$ is the character which takes a matrix in $B$ to the lower right corner entry.  Breuil \cite{Br} describes the structure of all principal series (see also Bardoe-Sin \cite{BS} and Diamond \cite{dia}). Putting
these results together, we see that 
if 
$r 
\equiv0\text{ mod } (q-1),$ we have  
\begin{equation}\label{0ps}
   \frac{V_{r}}{V_{r}^*} \cong (0,\ldots,0)\oplus (p-1,\ldots,p-1). 
\end{equation}
By \lemref{Xr}, we obtain that
$$\frac{X_{r}}{X_{r}^*}\cong V_{p-1,\ldots,p-1}=(p-1,\ldots,p-1).$$ 
Hence the quotient
$$Q \cong \frac{V_{r}/V_{r}^*}{X_{r}/X_r^*}\cong (0,\ldots,0)=\mathbbm{1}$$
is the trivial representation. Thus, we obtain a surjective map 
$$\text{ind}^G_{KZ}\mathbbm 1\twoheadrightarrow \bar{\Theta}_{k,a_p}.$$
The next theorem shows that the above map factors through the Hecke operator $T.$
\begin{thm}\label{requivzero}
Assume that $p > 2$. Let $r_i>q+p-2$,  $r
\equiv 0 \textnormal{ mod }(q-1)$ and $v(a_p) \in (0,1)$. Then the map $\textnormal{ind}^G_{KZ}\mathbbm 1\twoheadrightarrow \bar{\Theta}_{k,a_p}$ factors through $T$. In particular, the reduction of $\Theta_{k,a_p}$ is a quotient of a universal supersingular
representation.
\end{thm}
\begin{proof}
  Consider the function 
$$f_1=\sum\limits_{\mu \in \mathbb F_q}\bigg[\left(\begin{smallmatrix}
      p&[\mu]\\0&1
  \end{smallmatrix}\right),\frac{Y_0^{r_0}\cdots Y_{f-1}^{r_{f-1}}-X_0^{r_0-p+1}Y_0^{p-1}\cdots X_{f-1}^{r_{f-1}-p+1}Y_{f-1}^{p-1}}{a_p}\bigg].$$ 
  We calculate $(T-a_p)f_{1}.$ As in 
  the proof of Lemma~\ref{Xrinkernel}, we see that $T^+f_1$ dies mod $p.$ Also
  $$-a_pf_1=-\sum\limits_{\mu \in \mathbb F_q}\bigg[\left(\begin{smallmatrix}
      p&[\mu]\\0&1
  \end{smallmatrix}\right),Y_0^{r_0}\cdots Y_{f-1}^{r_{f-1}}-X_0^{r_0-p+1}Y_0^{p-1}\cdots X_{f-1}^{r_{f-1}-p+1}Y_{f-1}^{p-1}\bigg].$$
  Now we calculate $T^-f_1$ mod $p$. By \eqref{T-}, we have
  \begin{eqnarray*}
  T^-f_1&=&\frac{1}{a_p}    \sum\limits_{\mu \in \mathbb F_q}\bigg[\left(\begin{smallmatrix}
      p&p[\mu]\\0&p
  \end{smallmatrix}\right),Y_0^{r_0}\cdots Y_{f-1}^{r_{f-1}}-p^{\sum\limits_{i=0}^{f-1}r_i-p+1}X_0^{r_0-p+1}Y_0^{p-1}\cdots X_{f-1}^{r_{f-1}-p+1}Y_{f-1}^{p-1}\bigg]\\
 &=&\frac{1}{a_p} \sum\limits_{\mu \in \mathbb F_q}\bigg[\left(\begin{smallmatrix}
      1&[\mu]\\0&1
  \end{smallmatrix}\right),Y_0^{r_0}\cdots Y_{f-1}^{r_{f-1}}\bigg]\\
  &=&\frac{1}{a_p}\bigg[\mathrm{Id},\sum\limits_{\mu \in \mathbb F_q}([\mu] X_0+Y_0)^{r_0}\cdots ([\mu]^{p^{f-1}}X_{f-1}+Y_{f-1})^{r_{f-1}}\bigg]\\
  &=&\frac{1}{a_p}\bigg[\mathrm{Id},\sum\limits_{0\le j_0 \le r_0,\ldots, 0\le j_{f-1}\le r_{f-1}}\binom{r_0}{j_0}\cdots \binom{r_{f-1}}{j_{f-1}}X_0^{r_0-j_0}Y_0^{j_0} \cdots X_{f-1}^{r_{f-1}-j_{f-1}}Y_{f-1}^{j_{f-1}}\\
  && \quad \times \sum\limits_{\mu \in \mathbb F_q}[\mu]^{r_0+\cdots p^{f-1}r_{f-1}-(j_0+\cdots+p^{f-1}j_{f-1})}\bigg]\\
  &=&\frac{q-1}{a_p}\bigg[\mathrm{Id},\sideset{}{'}\sum\limits_{\substack{0\le j_0 \le r_0, \ldots,0\le j_{f-1}\le r_{f-1}\\ j_0+\cdots+p^{f-1}j_{f-1} \equiv r \text{ mod }(q-1)}}\binom{r_0}{j_0}\cdots \binom{r_{f-1}}{j_{f-1}}X_0^{r_0-j_0}Y_0^{j_0} \cdots X_{f-1}^{r_{f-1}-j_{f-1}}Y_{f-1}^{j_{f-1}}\bigg]\\
  &&\quad +\frac{q-1}{a_p}\bigg[\mathrm{Id},X_0^{r_0}\cdots X_{f-1}^{r_{f-1}}\bigg]+\frac{q}{a_p}\bigg[\mathrm{Id},Y_0^{r_0}\cdots Y_{f-1}^{r_{f-1}}\bigg].
  \end{eqnarray*}
  In the second equality of the above calculation, we have ignored the second term as $\frac{p^{\sum\limits_{i=0}^{f-1}r_1-p+1}}{a_p}$ dies mod $p.$ In the last equality we have used the fact that
  $$\sum\limits_{\mu \in \mathbb F_q}[\mu]^{r_0+\cdots +p^{f-1}r_{f-1}-(j_0+\cdots+p^{f-1}j_{f-1})} =\begin{cases}
     q & \text{ for } (j_0,\ldots,j_{f-1})= (r_0,\ldots,r_{f-1})\\
       q-1 & \text{ for } (j_0,\ldots,j_{f-1})=(0,\ldots,0),
  \end{cases}$$
  which further
  allows us to drop
  the last term as $\frac{q}{a_p}$ dies mod $p$. 
  
  Consider another function $f_{-1}=[\left(\begin{smallmatrix}
      1&0\\0&p
  \end{smallmatrix} \right),\frac{1}{a_p} (X_0^{r_0}\cdots X_{f-1}^{r_{f-1}}-X_0^{r_0-p+1}Y_0^{p-1}\cdots X_{f-1}^{r_{f-1}-p+1}Y_{f-1}^{p-1})].$ We have
  \begin{eqnarray*}
      T^{+}f_{-1}&=&\bigg[\mathrm{Id},\frac{1}{a_p}X_0^{r_0}\cdots X_{f-1}^{r_{f-1}}\bigg]\\
      T^{-}f_{-1}&=&0\\
      -a_pf_{-1}&=&-[\left(\begin{smallmatrix}
      1&0\\0&p
  \end{smallmatrix} \right), X_0^{r_0}\cdots X_{f-1}^{r_{f-1}}-X_0^{r_0-p+1}Y_0^{p-1}\cdots X_{f-1}^{r_{f-1}-p+1}Y_{f-1}^{p-1}].
  \end{eqnarray*}
  Note that $T^+f_{-1}$ kills the second last term
  in $T^-f_1$ mod $p$,
  which is the reason to introduce $f_{-1}$.
  
  The integers   $\beta_{j_0,\ldots,j_{f-1}}=\binom{r_0}{j_0}\cdots \binom{r_{f-1}}{j_{f-1}}$ for $0 \leq j_i \leq r_i$ and $0,r \neq j_0+\cdots+p^{f-1}j_{f-1} \equiv c \text{ mod } (q-1)$ with $c = q-1$
  and all other $\beta_{j_0,\ldots,j_{f-1}}= 0$ (so
  $\beta_{0,\ldots,0} = \beta_{r_0,\ldots,r_{f-1}} = 0$)
  satisfy the hypotheses of \lemref{vanishingexistence} with $t = 1$ and $0\le m_0+\cdots+m_{f-1} \le 1$.  Indeed,  the following sums of products of binomial coefficients vanish
  mod $p$:

  \begin{eqnarray*}
     &&\sideset{}{'}\sum\limits_{\substack{0\le j_0 \le r_0,\ldots, 0\le j_{f-1}\le r_{f-1}\\ j_0+\cdots p^{f-1}j_{f-1} \equiv r \text{ mod }(q-1)}}\binom{r_0}{j_0}\cdots \binom{r_{f-1}}{j_{f-1}}=0 \text{ mod }p \\
     &&\sideset{}{'}\sum\limits_{\substack{0\le j_0 \le r_0,\ldots, 0\le j_{f-1}\le r_{f-1}\\ j_0+\cdots p^{f-1}j_{f-1} \equiv r \text{ mod }(q-1)}}\binom{r_0}{j_0}\cdots \binom{r_{f-1}}{j_{f-1}}\binom{j_i}{1} =0 \text{ mod }p.
     \end{eqnarray*} For the first sum ($l_i=0$ for all $i$), use \lemref{binomialsum}. For the other sums ($l_0+\cdots+l_{f-1}=1$), we use \lemref{gbinomialsum}. Indeed
     \begin{eqnarray*}
        \sideset{}{'}\sum\limits_{\substack{0\le j_i \le r_i\\ j \equiv r \text{ mod }(q-1)}}\binom{r_0}{j_0}\cdots \binom{r_{f-1}}{j_{f-1}}\binom{j_i}{1}&\equiv& \binom{r_i}{1}\left[\binom{[q-1-p^i]}{[q-1-p^i]}+\delta_{q-1,[q-1-p^i]}\right]-\binom{r_i}{1}\\
     &=& r_i-r_i \> = \> 0 \text{ mod }p.
     \end{eqnarray*}
     By \lemref{vanishingexistence}, $\exists$ integers $\alpha_{j_0,\ldots,j_{f-1}}$ for 
     $j \equiv q-1 \text{ mod } (q-1)$ such that $\alpha_{j_0,\ldots,j_{f-1}}\equiv \beta_{j_0,\ldots,j_{f-1}}
     \text{ mod }p$ and \begin{equation}\label{sumalpha1}
      \sideset{}{}\sum\limits_{\substack{0\le j_0 \le r_0,\ldots, 0\le j_{f-1}\le r_{f-1}\\ {j_0+\cdots+p^{f-1}j_{f-1}} \equiv r \text{ mod }(q-1)}}\alpha_{j_0,\ldots,j_{f-1}}\binom{j_0}{l_0}\cdots \binom{j_{f-1}}{l_{f-1}} =0 \text{ mod }p^2
  \end{equation} for $l_0+\cdots +l_{f-1}\le 1$.
  By Remark~\ref{extreme values}, 
  $\alpha_{0,\ldots,0} = \alpha_{r_0,\ldots,r_{f-1}} = 0$.  
  Now consider the function
  $$f_0=-\bigg[\mathrm{Id}, \frac{1}{a_p^2}\sideset{}{'}\sum\limits_{\substack{0\le j_0\le r_0,\ldots,0\le j_{f-1}\le r_{f-1}\\j\equiv r \text{ mod }(q-1)}}\alpha_{j_0,\ldots,j_{f-1}}X_0^{r_0-j_0}Y_0^{j_0}\cdots X_{f-1}^{r_{f-1}-j_{f-1}}Y_{f-1}^{j_{f-1}}\bigg].$$
  The congruence \eqref{sumalpha1} and \eqref{T+} 
  implies that $T^+f_0 = 0 \text{ mod } p$. Clearly $T^-f_0$ also dies mod $p$ since $\alpha_{r_0,\ldots,r_{f-1}}=0$. 
   Finally, modulo $\text{ind}_{KZ}^G\>X_r,$ we obtain
   \begin{eqnarray*}
    &&   (T-a_p )(f_{-1}+f_{0}+f_1)\\
    &&\quad = -\sum\limits_{\mu \in \mathbb F_q}\bigg[\left(\begin{smallmatrix}
      p&[\mu]\\0&1
  \end{smallmatrix}\right),-X_0^{r_0-p+1}Y_0^{p-1}\cdots X_{f-1}^{r_{f-1}-p+1}Y_{f-1}^{p-1}\bigg]\\ 
  &&\quad\quad
  -\bigg[\left(\begin{smallmatrix}
      1&0\\0&p
  \end{smallmatrix} \right), -X_0^{r_0-p+1}Y_0^{p-1}\cdots X_{f-1}^{r_{f-1}-p+1}Y_{f-1}^{p-1}\bigg] \\
  &&\quad \quad-\frac{1}{a_p}\bigg[\mathrm{Id},\!\!\!\!\! \!\!\!\!\!
 \sideset{}{'}\sum\limits_{\substack{0\le j_0 \le r_0, \ldots,0\le j_{f-1}\le r_{f-1}\\ j_0+\cdots+p^{f-1}j_{f-1} \equiv r \text{ mod }(q-1)}}\!\!\!\!\!\!\!\!\!\!\!\!\! \left(\binom{r_0}{j_0}\cdots \binom{r_{f-1}}{j_{f-1}}-\alpha_{j_0,\ldots,j_{f-1}}\right)X_0^{r_0-j_0}Y_0^{j_0}\cdots  X_{f-1}^{r_{f-1}-j_{f-1}}Y_{f-1}^{j_{f-1}}\bigg].
   \end{eqnarray*}
   But $\beta_{j_0,\ldots, j_{f-1}}\equiv\alpha_{j_0,\dots,j_{f-1}}\text{ mod }p$ implies that the last function dies mod $p$, so that \begin{eqnarray}
   \label{heckeintegralformula}
    (T-a_p )(f_{-1}+f_{0}+f_1)&=& -\sum\limits_{\mu \in \mathbb F_q}[\left(\begin{smallmatrix}
      p&[\mu]\\0&1
  \end{smallmatrix}\right),-X_0^{r_0-p+1}Y_0^{p-1}\cdots X_{f-1}^{r_{f-1}-p+1}Y_{f-1}^{p-1}]\nonumber\\
  &&-[\left(\begin{smallmatrix}
      1&0\\0&p
  \end{smallmatrix} \right), -X_0^{r_0-p+1}Y_0^{p-1}\cdots X_{f-1}^{r_{f-1}-p+1}Y_{f-1}^{p-1}] \quad \text{ mod } p.   
   \end{eqnarray}
   
    Now observe that the isomorphism
  from $V_r/V_r^*$ to $\mathrm{ind}_B^\Gamma 1$ in \cite[Theorem 2.18]{GJ} is induced by the evaluation map $v\mapsto(\left(\begin{smallmatrix}
      a&b\\c&d
\end{smallmatrix}\right)\mapsto v(c,d,\ldots,c^{p^{f-1}},d^{p^{f-1}}))$.
Thus $$v=X_0^{r_0}\cdots X_{f-1}^{r_{f-1}}+Y_0^{r_0}\cdots Y_{f-1}^{r_{f-1}}-X_0^{r_0-p+1}Y_0^{p-1}\cdots X_{f-1}^{r_{f-1}-p+1}Y_{f-1}^{p-1}$$ maps to the constant function $1_{B\backslash \Gamma}$ in the principal series. Indeed 
  \begin{eqnarray*}
     \text{if }c=0, d\neq 0 \text{ then } v(c,d,\ldots,c^{p^{f-1}},d^{p^{f-1}})&=&d^{r_0+pr_1+\cdots+p^{f-1}r_{f-1}}=1\\ 
      \text{if }c\neq0, d=0 \text{ then } v(c,d,\ldots,c^{p^{f-1}},d^{p^{f-1}})&=&c^{r_0+pr_1+\cdots+p^{f-1}r_{f-1}}=1\\
      \text{if }c\neq0, d\neq0 \text{ then } v(c,d,\ldots,c^{p^{f-1}},d^{p^{f-1}})&=&c^{r_0+pr_1+\cdots+p^{f-1}r_{f-1}}+d^{r_0+pr_1+\cdots+p^{f-1}r_{f-1}}\\&&-c^{r_0+pr_1+\cdots+p^{f-1}r_{f-1}-(q-1)}d^{q-1}=1+1-1=1.
  \end{eqnarray*}
  But $1_{B\backslash \Gamma}$ generates the trivial representation (the constant functions)
  in the principal series $\mathrm{ind}_B^\Gamma 1$. We deduce that the
  polynomial $v$ above generates 
  the trivial
  representation
  $\mathbbm 1$ 
  in $V_r/V_r^*$.
  This is also proved in  \cite{CS}.
  
   Since the polynomial 
   $v$ above is the same as that occurring in \eqref{heckeintegralformula} modulo $X_r$, we 
   see that the
   under the projection from $V_r/V_r^*$ to $\mathbbm 1$ (and using \eqref{T+} and \eqref{T-}), the right hand side
   of \eqref{heckeintegralformula} 
   projects to 
   $$T([\mathrm{Id, 1}]) =  \sum\limits_{\mu \in \mathbb F_q}[\left(\begin{smallmatrix}
      p&[\mu]\\0&1
  \end{smallmatrix}\right),1]+[\left(\begin{smallmatrix}
      1&0\\0&p
  \end{smallmatrix} \right), 1]  \in \text{ind}^G_{KZ}\mathbbm 1,$$
  for a generator $1 \in \mathbbm 1$. But for
  any function $f$
  for which $(T-a_p)f$ integral,
  the function $(T-a_p)f$ dies in $\Theta_{k,a_p}$, hence in $\bar\Theta_{k,a_p}$. 
  We conclude that the map $\textnormal{ind}^G_{KZ}\mathbbm 1\twoheadrightarrow \bar{\Theta}_{k,a_p}$ factors through $T.$
\end{proof}

\section{Cosocle of $Q$}

In the previous section, we studied the mod $p$ reduction of 
$\Theta_{k,a_p}$
for  
$r=r_0+pr_1+\cdots+ p^{f-1}r_{f-1}\equiv 0 \text{ mod } (q-1)$ where 
$r_i = k_i -2$ and
$0 < v(a_p) < 1$.
In this section, we start our study for other congruence classes of $r$ mod $(q-1)$, by studying the map $\mathrm{ind}^G_{KZ}Q \rightarrow \bar{\Theta}_{k,a_p}$ for the top subquotient of $Q$, namely $\text{cosoc}(Q)$ the cosocle of $Q$. 

Let $r\equiv a \> \text{ mod }(q-1)$, where $1 \leq a = a_0+pa_1+\cdots+a_{f-1}p^{f-1}  \leq q-1$. Recall that 
$$Q=\frac{V_{r}}{V^*_{r}+X_{r}}.$$
The socle filtration of $V_r/V_r^* = \mathrm{ind}^\Gamma_B d^r$ (cf. \cite[Theorem 2.18]{GJ}) has up to $2^f$ Jordan-H\"older factors arranged in $f+1$ layers with the cosocle  given by $(p-1-a_0,p-1-a_1,\ldots,p-1-a_{f-1})\otimes D^{a}.$ 
Going mod $X_r$ kills the socle $V_a$ of $V_r/V_r^*$  by Lemma~\ref{Xr}. 

For example, when 
$f =2$, the principal series $V_{r_0,r_1}/V_{r_0,r_1}^*$ has 
up to $4$ Jordan-H\"older factors and
the socle filtration is given by the following diagram:
 \begin{center}
\begin{tikzpicture}[baseline=(current bounding box.center)]
  \node (A) at (0,2) {$(p-1-a_0,p-1-a_1)\otimes D^{a_0+pa_1}$}; 
  \node (B) at (-3,0) {$(a_0 - 1, p - 2 - a_1) \otimes D^{(1 + a_1)p}$};
  \node (C) at (3,0) {$(p-2-a_0,a_1-1)\otimes D^{1+a_0}$}; 
  \node (D) at (0,-2) {$(a_0,a_1)$.};

  \draw[solid] (A) -- (B) -- (D) -- (C) -- (A);

  \node at (0,0) {$\oplus$};
\end{tikzpicture}
\end{center}
  Hence the structure of $Q$ when $f = 2$ is given by the following diagram:
\begin{center}
\begin{tikzpicture}[baseline=(current bounding box.center)]

  \node (A) at (-2,0) {$ (a_0 - 1, p - 2 - a_1) \otimes D^{(1 + a_1)p}$};
  \node (B) at (2,2) {$(p-1-a_0,p-1-a_1)\otimes D^{a_0+pa_1}$};
  \node (D) at (6,0) {$(p-2-a_0,a_1-1)\otimes D^{1+a_0}$.};

  \draw[solid] (A) -- (B) -- (D);

  \node at (2,0) {$\oplus$};
\end{tikzpicture}
\captionof{figure}{}
\label{fig:socle}
\end{center}
In this case, we study the image
of the induction of the cosocle
under the map \eqref{surjectivemap}
modulo the images of the inductions of the bottom two Jordan-H\"older factors. 

Let us return
to the case of general $f \geq 1$.
Write $Q_{i}$ for 
$1 \leq i \leq f$ for the filtration on 
$Q$ such that 
$Q_i/Q_{i+1}$ gives
the $i$-ith layer
in the socle filtration of $Q$.
So $Q_f = Q$ and $Q_1 = \mathrm{soc}(Q)$. 
Let $$F_i = \mathrm{Im}(\mathrm{ind}_{KZ}^GQ_i \rightarrow\bar{\Theta}_{k,a_p})
$$ 
so that $F_i$ for $1 \leq i \leq f$ defines
a filtration on $\bar{\Theta}_{k,a_p}$. Finally 
define $F_t$
($t$ for `top') to be $$F_t : = F_f/F_{f-1}.$$
Note that there is a surjection $$\mathrm{ind}_{KZ}^G
\mathrm{cosoc}(Q) =
\mathrm{ind}_{KZ}^G Q_f/Q_{f-1} \twoheadrightarrow F_t.
$$

\subsection{Weights $r\nequiv  0,1,p,\ldots,p^{f-1}$ mod $(q-1)$}
\begin{thm}\label{nonexcep}
Let $p > 2$.
Let $r_i>q+p-2$ and $r\nequiv  0,1,p,\ldots,p^{f-1}$ \textnormal{mod} $(q-1)$. Let $v(a_p) \in (0,1)$. Then the map $\textnormal{ind}^G_{KZ}\mathrm{cosoc}(Q) \rightarrow F_t
$ factors through $T.$ 
\end{thm}
\begin{proof}
  We consider the function 
  $$f_1=\sum\limits_{\mu \in \mathbb F_q}\bigg[\left(\begin{smallmatrix}
      p&[\mu]\\0&1
  \end{smallmatrix}\right),\frac{Y_0^{r_0}\cdots Y_{f-1}^{r_{f-1}}-X_0^{r_0-a_0}Y_0^{a_0}\cdots X_{f-1}^{r_{f-1}-a_{f-1}}Y_{f-1}^{a_{f-1}}}{a_p}\bigg]$$ 
  and calculate $(T-a_p)f_{1}.$ We obtain that $T^+f_1$ dies mod $p$ (we use here
  that at least one $a_i \geq 1$). Also
  $$-a_pf_1=-\sum\limits_{\mu \in \mathbb F_q}[\left(\begin{smallmatrix}
      p&[\mu]\\0&1
  \end{smallmatrix}\right),{Y_0^{r_0}\cdots Y_{f-1}^{r_{f-1}}-X_0^{r_0-a_0}Y_0^{a_0}\cdots X_{f-1}^{r_{f-1}-a_{f-1}}Y_{f-1}^{a_{f-1}}}].$$
  Now we calculate $T^-f_1.$ A  calculation similar to the one in the last section gives
  \begin{eqnarray*}
  T^-f_1&=&\frac{1}{a_p}    \sum\limits_{\mu \in \mathbb F_q}\bigg[\left(\begin{smallmatrix}
      p&p[\mu]\\0&p
  \end{smallmatrix}\right),Y_0^{r_0}\cdots Y_{f-1}^{r_{f-1}}-p^{\sum\limits_{i=0}^{f-1}r_i-a_i}X_0^{r_0-a_0}Y_0^{a_0}\cdots X_{f-1}^{r_{f-1}-a_{f-1}}Y_{f-1}^{a_{f-1}}\bigg]\\ 
  &=&\frac{q-1}{a_p}\bigg[\mathrm{Id},\sideset{}{'}\sum\limits_{\substack{0\le j_i \le r_i\\ j \equiv r \text{ mod }(q-1)}}\binom{r_0}{j_0}\cdots \binom{r_{f-1}}{j_{f-1}}X_0^{r_0-j_0}Y_0^{j_0}\cdots X_{f-1}^{r_{f-1}-j_{f-1}}Y_{f-1}^{j_{f-1}}\bigg].
  \end{eqnarray*}
   The extreme
   terms are excluded
   since $r \not\equiv 0 \text{ mod } (q-1)$ and $\frac{q}{a_p} = 0 \text{ mod } p$.
  
   The integers   $\beta_{j_0,\ldots,j_{f-1}}=\binom{r_0}{j_0}\cdots \binom{r_{f-1}}{j_{f-1}}$ for $0 \leq j_i \leq r_i$ and $r \neq j_0+\cdots+p^{f-1}j_{f-1} \equiv c \text{ mod } (q-1)$ with $c = a$
  and all other $\beta_{j_0,\ldots,j_{f-1}}= 0$ (so
  $\beta_{r_0,\ldots,r_{f-1}} = 0$)
  satisfy the hypotheses of \lemref{vanishingexistence} with $t = 1$ and $0\le m_0+\cdots+m_{f-1} \le 1$.  Indeed, by Lemma~\ref{binomialsum}, resp. Lemma~\ref{gbinomialsum}, the following sums of products of binomial coefficients vanish
  mod $p$:
   \begin{eqnarray*}
    \sideset{}{'}\sum\limits_{\substack{0\le j_i \le r_i\\ j \equiv r \text{ mod }(q-1)}}\binom{r_0}{j_0}\cdots \binom{r_{f-1}}{j_{f-1}}&\equiv&0 \text{ mod }p \\
     \sideset{}{'}\sum\limits_{\substack{0\le j_i \le r_i\\ j \equiv r \text{ mod }(q-1)}}\binom{r_0}{j_0}\cdots \binom{r_{f-1}}{j_{f-1}}\binom{j_i}{1}&\equiv& \binom{r_i}{1}\left[\binom{[a-p^i]}{[a-p^i]}+\delta_{q-1,[a-p^i]}\right]-\binom{r_i}{1}\\
     &=& r_i-r_i = 0 \text{ mod }p.
  \end{eqnarray*}
  Here we have used $\delta_{q-1,[a-p^i]} = 0 \iff a \neq  p^i$, which is true. By \lemref{vanishingexistence}, $\exists$ integers $\alpha_{j_0,\ldots,j_{f-1}}$ for 
     $j \equiv a \text{ mod } (q-1)$ such that $\alpha_{j_0,\ldots,j_{f-1}}\equiv \beta_{j_0,\ldots,j_{f-1}}
     \text{ mod }p$ and
     \begin{equation}\label{sumalpha}
      \sideset{}{}\sum\limits_{\substack{0\le j_i \le r_i,\\ j \equiv r \text{ mod }(q-1)}}\alpha_{j_0,\ldots,j_{f-1}}\binom{j_0}{l_0}\cdots \binom{j_{f-1}}{l_{f-1}} \equiv 0 \text{ mod }p^2
  \end{equation}
  whenever $l_0+\cdots +l_{f-1}\le 1$.  By Remark~\ref{extreme values}, 
  $\alpha_{r_0,\ldots,r_{f-1}} = 0$. Now consider the function
  $$f_0=-\bigg[\mathrm{Id}, \frac{1}{a_p^2}\sideset{}{'}\sum\limits_{\substack{0\le j_i\le r_i\\j \equiv r \text{ mod }(q-1)}}\alpha_{j_0,\ldots,j_{f-1}}X_0^{r_0-j_0}Y_0^{j_0}\cdots X_{f-1}^{r_{f-1}-j_{f-1}}Y_{f-1}^{j_{f-1}}\bigg].$$
   The formulas \eqref{sumalpha} and  \eqref{T+}   imply that $T^+f_0$ dies mod $p.$ Also $T^-f_0$ dies mod $p$ since $\alpha_{r_0,\ldots, r_{f-1}}=0$. Modulo $\text{ind}_{KZ}^G X_r$, we obtain
   \begin{eqnarray*}
       (T-a_p)(f_0+f_1)&=& -\sum\limits_{\mu \in \mathbb F_q}\bigg[\left(\begin{smallmatrix}
      p&[\mu]\\0&1
  \end{smallmatrix}\right),-X_0^{r_0-a_0}Y_0^{a_0}\cdots X_{f-1}^{r_{f-1}-a_{f-1}}Y_{f-1}^{a_{f-1}}\bigg]\\
  &&-\frac{1}{a_p}\bigg[\mathrm{Id},\!\!\!\!\!\!\!\!\!\!\! \sideset{}{'}\sum\limits_{\substack{0\le j_i \le r_i\\ j \equiv r \text{ mod }(q-1)}}\!\!\!\!\!\!\!\! \left(\binom{r_0}{j_0}\cdots \binom{r_{f-1}}{j_{f-1}}-\alpha_{j_0,\ldots,j_{f-1}}\right)X_0^{r_0-j_0}Y_0^{j_0}\cdots X_{f-1}^{r_{f-1}-j_{f-1}}Y_{f-1}^{j_{f-1}}\bigg].
   \end{eqnarray*}
   But the congruence $\alpha_{j_0,\ldots,j_{f-1}}\equiv \beta_{j_0,\ldots,j_{f-1}}$ mod $p$ implies that the last function dies mod $p$, hence 
   \begin{eqnarray} \label{heckeformula}
   (T-a_p)(f_0+f_1)= \sum\limits_{\mu \in \mathbb F_q}[\left(\begin{smallmatrix}
      p&[\mu]\\0&1
  \end{smallmatrix}\right),X_0^{r_0-a_0}Y_0^{a_0}\cdots X_{f-1}^{r_{f-1}-a_{f-1}}Y_{f-1}^{a_{f-1}}] \quad \text{mod } p.
  \end{eqnarray}
  
   Under the map $V_r/V_r^* \rightarrow \mathrm{cosoc}(Q)$,
  the polynomial $X_0^{r_0-a_0}Y_0^{a_0}\cdots X_{f-1}^{r_{f-1}-a_{f-1}}Y_{f-1}^{a_{f-1}}$
  maps to
  $X_0^{p-1-a_0}\cdots X_{f-1}^{p-1-a_{f-1}}$ (as can
  easily be checked; for the
  case $f = 2$, see \eqref{cosocproj} in the next section). Thus
  $$[\left(\begin{smallmatrix}
      p&[\mu]\\0&1
  \end{smallmatrix}\right),X_0^{r_0-a_0}Y_0^{a_0}\cdots X_{f-1}^{r_{f-1}-a_{f-1}}Y_{f-1}^{a_{f-1}}] \in \text{ind}^G_{KZ} V_r$$
  maps to
  $$[\left(\begin{smallmatrix}
      p&[\mu]\\0&1
  \end{smallmatrix}\right),X_0^{p-1-a_0}\cdots X_{f-1}^{p-1-a_{f-1}}]\in \text{ind}^G_{KZ}(p-1-a_0,\ldots,p-1-a_1)\otimes D^{a}.$$
Using \eqref{T+} and \eqref{T-}, the right hand side of \eqref{heckeformula} under the projection of $V_r/V_r^*$ to $\mathrm{cosoc}$ maps to 
  $$T([\text{Id},X_0^{p-1-a_0}\cdots X_{f-1}^{p-1-a_{f-1}}])=\sum\limits_{\mu \in \mathbb F_q}[\left(\begin{smallmatrix}
      p&[\mu]\\0&1
  \end{smallmatrix}\right),X_0^{p-1-a_0}\cdots X_{f-1}^{p-1-a_{f-1}}].$$
  
  Hence the map $\textnormal{ind}^G_{KZ}\mathrm{cosoc}(Q) \rightarrow F_t$ factors through $T$. 
\end{proof}

  \begin{rmk}
  The two previous theorems generalize the
  work of Buzzard-Gee for the case 
  $f = 1$ and $p > 2$. Indeed, it is shown in \cite{BGEE} that
$\bar\Theta_{k,a_p}$
is always irreducible and supercuspidal if
$r \not\equiv 1 \text{ mod } (p-1)$.
In this case, the
conclusion follows immediately from 
the mod $p$ local
Langlands correspondence, since
the weight $V_{p-a_0-1} \otimes D^{a_0}$ is not a twist of $V_{p-2}$ so does not
support any pairs
of principal series
representations, which is the only other possibility.  Such a correspondence is not (yet) available for $f \geq 2$, and so in some sense our proof for $f \geq 1$ here is new even in the case of $f = 1$. 
\end{rmk}

  \subsection{Weights $r\equiv1,p,\dots,p^{f-1}$ mod $(q-1)$ (exceptional cases)}
  
\begin{thm}
\label{thmexcept}
Let $p > 2$.
Let $r_i>q+p-2$, $r\equiv p^h \textnormal{ mod } (q-1)$ for
$0 \leq h \leq f-1$ and $v(a_p) \in (0,1)$. The map $\textnormal{ind}^G_{KZ}\mathrm{cosoc}(Q) \rightarrow  
F_t$ factors through $T$ except possibly if 
$v(a_p) = \frac{1}{2}$ and
$\frac{a_p^2 - pr_h}{p}$ is not a unit.

\end{thm}

\begin{proof}
We break the proof
into the cases:
\begin{itemize}
    \item 
  $p \mid r_h$, 
 or $0 < v(a_p)< \frac{1}{2}$, or $v(a_p) = \frac{1}{2}$ and $\frac{a_p^2 - pr_h}{a_p^2}$ is a unit
 \item $p\nmid r_h$ and $\frac{1}{2} < v(a_p) < 1$, or 
 $v(a_p) = \frac{1}{2}$ and $\frac{a_p^2 - pr_h}{p}$ is a unit.
\end{itemize}
Note that  when $v(a_p) = \frac{1}{2}$, the constant in the first 
bullet point is a unit
if and only if the constant in the second bullet point is a unit. So we obtain nothing new
when $v(a_p) = \frac{1}{2}$ in the second bullet point,
but prove it for reasons of symmetry.
 
First assume that $p \mid r_h$.
Consider the function 
  $$f_1=\sum\limits_{\mu \in \mathbb F_q}\bigg[\left(\begin{smallmatrix}
      p&[\mu]\\0&1
  \end{smallmatrix}\right),\frac{Y_0^{r_0}\cdots Y_{f-1}^{r_{f-1}}-X_0^{r_0}X_1^{r_1}\cdots X_h^{r_h-1}Y_h \cdots X_{f-1}^{r_{f-1}}}{a_p}\bigg]$$
  and calculate $(T-a_p)f_{1}.$ We obtain that $T^+f_1$ dies mod $p.$ Also
  $$-a_pf_1=-\sum\limits_{\mu \in \mathbb F_q}[\left(\begin{smallmatrix}
      p&[\mu]\\0&1
  \end{smallmatrix}\right),{Y_0^{r_0}\cdots Y_{f-1}^{r_{f-1}}-X_0^{r_0}X_1^{r_1}\cdots X_h^{r_h-1}Y_h \cdots X_{f-1}^{r_{f-1}}}].$$
  Now we calculate $T^-f_1.$ A calculation similar to the one in the previous subsection gives
  \begin{eqnarray*}
  T^-f_1&=&\frac{1}{a_p}    \sum\limits_{\mu \in \mathbb F_q} \bigg[\left(\begin{smallmatrix}
      p&p[\mu]\\0&p
  \end{smallmatrix}\right),Y_0^{r_0}\cdots Y_{f-1}^{r_{f-1}}-p^{(\sum\limits_{i=0}^{f-1}r_i)-1}X_0^{r_0}X_1^{r_1}\cdots X_h^{r_h-1}Y_h \cdots X_{f-1}^{r_{f-1}}\bigg]\\ 
  &=&\frac{q-1}{a_p}\bigg[\mathrm{Id},\sideset{}{'}\sum\limits_{\substack{0\le j_i \le r_i\\ j \equiv p^h \text{ mod }(q-1)}}\binom{r_0}{j_0}\cdots \binom{r_{f-1}}{j_{f-1}}X_0^{r_0-j_0}Y_0^{j_0} \cdots X_{f-1}^{r_{f-1}-j_{f-1}}Y_{f-1}^{j_{f-1}}\bigg].
  \end{eqnarray*}

   The integers   $\beta_{j_0,\ldots,j_{f-1}}=\binom{r_0}{j_0}\cdots \binom{r_{f-1}}{j_{f-1}}$ for $0 \leq j_i \leq r_i$ and $r \neq j_0+\cdots+p^{f-1}j_{f-1} \equiv c \text{ mod } (q-1)$ with $c = p^h$
  and all other $\beta_{j_0,\ldots,j_{f-1}}= 0$ (so
  $\beta_{r_0,\ldots,r_{f-1}} = 0$)
  satisfy the hypotheses of \lemref{vanishingexistence} with $t = 1$ and $0\le m_0+\cdots+m_{f-1} \le 1$.  Indeed, by Lemma~\ref{binomialsum}, resp. Lemma~\ref{gbinomialsum}, the following sums of products of binomial coefficients vanish
  mod $p$:
   \begin{eqnarray*}
    \sideset{}{'}\sum\limits_{\substack{0\le j_i \le r_i\\ j \equiv r \text{ mod }(q-1)}}\binom{r_0}{j_0}\cdots \binom{r_{f-1}}{j_{f-1}}&\equiv&0 \text{ mod }p \\
     \sideset{}{'}\sum\limits_{\substack{0\le j_i \le r_i\\ j \equiv r \text{ mod }(q-1)}}\binom{r_0}{j_0}\cdots \binom{r_{f-1}}{j_{f-1}}\binom{j_i}{1}&\equiv& \binom{r_i}{1}\left[\binom{[p^h-p^i]}{[p^h-p^i]}+\delta_{q-1,[p^h-p^i]}\right]-\binom{r_i}{1}\\
     &=& 
     \begin{cases}
        r_i - r_ i = 0 & \text{if } i \neq h \\
        2r_i-r_i = r_i =  0 \text{ mod }p & \text{if } i =h,
      \end{cases}
  \end{eqnarray*}
  by assumption. Here we have used $\delta_{q-1,[p^h-p^i]} = 1 \iff i = h$. By \lemref{vanishingexistence}, $\exists$ integers $\alpha_{j_0,\ldots,j_{f-1}}$ for 
     $j \equiv p^h \text{ mod } (q-1)$ such that $\alpha_{j_0,\ldots,j_{f-1}}\equiv \beta_{j_0,\ldots,j_{f-1}}
     \text{ mod }p$ and
     \begin{equation}\label{sumalphaexcep}
      \sideset{}{}\sum\limits_{\substack{0\le j_i \le r_i,\\ j \equiv r \text{ mod }(q-1)}}\alpha_{j_0,\ldots,j_{f-1}}\binom{j_0}{l_0}\cdots \binom{j_{f-1}}{l_{f-1}} \equiv 0 \text{ mod }p^2
  \end{equation}
  whenever $l_0+\cdots +l_{f-1}\le 1$.  By Remark~\ref{extreme values}, 
  $\alpha_{r_0,\ldots,r_{f-1}} = 0$. Now consider the function
  $$f_0=-\bigg[\mathrm{Id}, \frac{1}{a_p^2}\sideset{}{'}\sum\limits_{\substack{0\le j_i\le r_i\\j \equiv r \text{ mod }(q-1)}}\alpha_{j_0,\ldots,j_{f-1}}X_0^{r_0-j_0}Y_0^{j_0}\cdots X_{f-1}^{r_{f-1}-j_{f-1}}Y_{f-1}^{j_{f-1}}\bigg].$$
  Note that $T^+f_0 = 0$ by \eqref{T+} and \eqref{sumalphaexcep} and the fact
  that $\alpha_{j_0,\ldots,j_{f-1}} \equiv
  r_h \text{ mod } p$ for $(j_0,\ldots,j_h, \ldots, j_{f-1}) = (0, \ldots, 1, \ldots, 0)$. Also
  $T^-f_0 = 0$ since
$\alpha_{r_0,\ldots,r_{f-1}} = 0$ and
the case where $(j_0,\ldots,j_i, \ldots, j_{f-1}) = (r_0, \ldots, r_i-1, \ldots, r_{f-1})$ does not
occur in the sum since $r-p^i \not \equiv r \text{ mod } (q-1)$ for $0 \leq i \leq f-1$.

Going modulo $\text{ind}_{KZ}^G X_r$ as well, we obtain
   \begin{eqnarray*}
       (T-a_p)(f_0+f_1)&=& -\sum\limits_{\mu \in \mathbb F_q}\bigg[\left(\begin{smallmatrix}
      p&[\mu]\\0&1
  \end{smallmatrix}\right),-X_0^{r_0} X_1^{r_1} \cdots X_{h}^{r_{h}-1} Y_h \cdots X_{f-1}^{r_{f-1}} \bigg]\\
  &&-\frac{1}{a_p}\bigg[\mathrm{Id},\!\!\!\!\!\!\!\!\!\!\! \sideset{}{'}\sum\limits_{\substack{0\le j_i \le r_i\\ j \equiv r \text{ mod }(q-1)}}\!\!\!\!\!\!\!\! \left(\binom{r_0}{j_0}\cdots \binom{r_{f-1}}{j_{f-1}}-\alpha_{j_0,\ldots,j_{f-1}}\right)X_0^{r_0-j_0}Y_0^{j_0}\cdots X_{f-1}^{r_{f-1}-j_{f-1}}Y_{f-1}^{j_{f-1}}\bigg].
   \end{eqnarray*}
   But the congruence $\alpha_{j_0,\ldots,j_{f-1}}\equiv \beta_{j_0,\ldots,j_{f-1}}$ mod $p$ implies that the last function dies mod $p$, hence 
   \begin{eqnarray} \label{heckeformulaexcep}
   (T-a_p)(f_0+f_1)= \sum\limits_{\mu \in \mathbb F_q}[\left(\begin{smallmatrix}
      p&[\mu]\\0&1
  \end{smallmatrix}\right),X_0^{r_0-a_0}Y_0^{a_0}\cdots X_{f-1}^{r_{f-1}-a_{f-1}}Y_{f-1}^{a_{f-1}}] \quad \text{mod } p
  \end{eqnarray}
  for $(a_0, \ldots, a_h, \ldots, a_{f-1}) = (0, \ldots, 1, \ldots, 0)$.
  The rest of the proof now proceeds
  exactly as in the previous subsection, to show that the map
  $\mathrm{ind}_{KZ}^G \text{cosoc}(Q) \twoheadrightarrow
  F_t$ factors through $T$. This finishes the case 
  $p \mid r_h$.

  The above argument
  simplifies if we
  assume $v(a_p) < \frac{1}{2}$. We keep
  the same $f_1$ as
  above but simply choose $f_0=-\frac{1}{a_p^2}[\mathrm{Id},\sideset{}{'}\sum\limits_{\substack{0\le j_i \le r_i\\ j \equiv r \text{ mod }(q-1)}}\binom{r_0}{j_0}\cdots \binom{r_{f-1}}{j_{f-1}}X_0^{r_0-j_0}Y_0^{j_0}\cdots  X_{f-1}^{r_{f-1}-j_{f-1}}Y_{f-1}^{j_{f-1}}].$
Since the first sum 
of products of binomial coefficients above vanishes and  $v(a_p^2)<1$,  
we see that
$T^+f_0$ dies mod $p$. Similarly $T^-f_0$ dies mod $p$. Working modulo $\text{ind}_{KZ}^{G}\text{soc}(Q)$, we again arrive at \eqref{heckeformulaexcep}, and the proof
proceeds as above.

In fact, we can extend the proof in
the previous paragraph to the case
of $v(a_p) = \frac{1}{2}$. This time 
$$T^+f_0 =  -\frac{pr_h}{a_p^2}
\sum\limits_{\mu \in \mathbb F_q}\bigg[\left(\begin{smallmatrix}
      p&[\mu]\\0&1
  \end{smallmatrix}\right), X_0^{r_0} X_1^{r_1} \cdots X_{h}^{r_{h}-1} Y_h \cdots X_{f-1}^{r_{f-1}} \bigg]$$
contributes, so that we obtain
\eqref{heckeformulaexcep} but with the
leading constant $1$ replaced
by $1 - \frac{pr_h}{a_p^2}$.
This completes the proof of the first bullet point.

Now let us assume $v(a_p)>\frac{1}{2}$ and $p \nmid r_h$.
Consider $f_1'=\frac{a_p^2}{p}f_1$ and $f_0'=\frac{a_p^2}{p}f_0$ where $f_1$ and $f_0$ are as in the previous paragraph:
 $$f'_1= \frac{a_p}{p} \sum\limits_{\mu \in \mathbb F_q}[\left(\begin{smallmatrix}
      p&[\mu]\\0&1
  \end{smallmatrix}\right),Y_0^{r_0}\cdots Y_{f-1}^{r_{f-1}}-X_0^{r_0}X_1^{r_1}\cdots X_h^{r_h-1}Y_h \cdots X_{f-1}^{r_{f-1}}]$$
  $$f'_0=-\frac{1}{p}\bigg[\mathrm{Id},\sideset{}{'}\sum\limits_{\substack{0\le j_i \le r_i\\ j \equiv r \text{ mod }(q-1)}}\binom{r_0}{j_0}\cdots \binom{r_{f-1}}{j_{f-1}}X_0^{r_0-j_0}Y_0^{j_0}\cdots  X_{f-1}^{r_{f-1}-j_{f-1}}Y_{f-1}^{j_{f-1}}\bigg].$$
  
One can easily show that $T^+f_1'$  and $-a_pf_1'$ ($v(a_p)>\frac{1}{2}$) die mod $p,$ and $T^-f_1'$ given by 
\begin{eqnarray*}
  T^-f_1'&=&\frac{(q-1)a_p}{p}\bigg[\mathrm{Id},\sideset{}{'}\sum\limits_{\substack{0\le j_i \le r_i\\ j \equiv r\text{ mod }(q-1)}}\binom{r_0}{j_0}\cdots \binom{r_{f-1}}{j_{f-1}}X_0^{r_0-j_0}Y_0^{j_0} \cdots X_{f-1}^{r_{f-1}-j_{f-1}}Y_{f-1}^{j_{f-1}}\bigg].
  \end{eqnarray*}
  Also $$-a_pf_0'=\frac{a_p}{p}\bigg[\mathrm{Id},\sideset{}{'}\sum\limits_{\substack{0\le j_i \le r_i\\ j \equiv r \text{ mod }(q-1)}}\binom{r_0}{j_0}\cdots \binom{r_{f-1}}{j_{f-1}}X_0^{r_0-j_0}Y_0^{j_0}\cdots X_{f-1}^{r_{f-1}-j_{f-1}}Y_{f-1}^{j_{f-1}}\bigg]$$ 
  and $T^-f_0'=0\text{ mod }p$ as was
  argued for the original $f_0$ above. To calculate $T^+f_0'$ we compute the usual
  sums of products of binomial coefficients:
\begin{eqnarray*}
     \sideset{}{'}\sum\limits_{\substack{0\le j_i \le r_i\\ j\equiv r\text{ mod }(q-1)}}\binom{r_0}{j_0}\cdots \binom{r_{f-1}}{j_{f-1}}&\equiv& 0 \text{ mod }p \\
     \sideset{}{'}\sum\limits_{\substack{0\le j_i \le r_i\\ j \equiv r\text{ mod }(q-1)}}\binom{r_0}{j_0}\cdots \binom{r_{f-1}}{j_{f-1}}\binom{j_i}{1} &\equiv&\binom{r_i}{1}\left[\binom{[p^h-p^i]}{[p^h-p^i]}+\delta_{q-1,[p^h-p^i]}\right]-\binom{r_i}{1}\\
     &= & \begin{cases}
         0 & \text{if } i \neq h \\
         r_h & \text{if } i =h,
    \end{cases} \quad  \text{ mod }p.\\
  \end{eqnarray*}
  Going modulo
  $\mathrm{ind}_{KZ}^G X_r$, we get 
  $$T^+f_0'=-r_h\sum\limits_{\mu \in \mathbb F_q}[\left(\begin{smallmatrix}
      p&[\mu]\\0&1
  \end{smallmatrix}\right),{X_0^{r_0}X_1^{r_1}\cdots X_h^{r_h-1}Y_h \cdots X_{f-1}^{r_{f-1}}}] \quad \text{mod } p.$$
  
  Hence $(T-a_p)(f_0'+f_1')=-r_h \sum\limits_{\mu \in \mathbb F_q}[\left(\begin{smallmatrix}
      p&[\mu]\\0&1
  \end{smallmatrix}\right),{X_0^{r_0}X_1^{r_1}\cdots X_h^{r_h-1}Y_h \cdots X_{f-1}^{r_{f-1}}}]$
  which is nonzero
  since we have assumed that $p\nmid r_h$.
As above, we conclude that the map $\textnormal{ind}^G_{KZ}\mathrm{cosoc}(Q) \twoheadrightarrow F_t$ 
factors through $T$.

We remark that in
the case of the first bullet point, we can extend  the proof just above to the case $v(a_p) = \frac{1}{2}$ and $\frac{a_p^2 - pr_h}{p}$ is a unit.
This time 
$$-a_pf_1' =  \frac{a_p^2}{p}
\sum\limits_{\mu \in \mathbb F_q}\bigg[\left(\begin{smallmatrix}
      p&[\mu]\\0&1
  \end{smallmatrix}\right), X_0^{r_0} X_1^{r_1} \cdots X_{h}^{r_{h}-1} Y_h \cdots X_{f-1}^{r_{f-1}} \bigg]$$
contributes, so that we obtain
\eqref{heckeformulaexcep} but with the
leading constant $-r_h$ replaced
by $-r_h + \frac{a_p^2}{p}$.
\end{proof}

\begin{rmk}
  Again the theorem recovers some 
  work of Buzzard-Gee for the case 
  $f = 1$ and $p > 2$. Indeed,
  in a second paper \cite{BGEE2}, it is shown that
$\bar\Theta_{k,a_p}$
is supercuspidal when
$r \equiv 1 \text{ mod } (p-1)$
under the conditions 
$v(a_p) \neq  \frac{1}{2}$ 
or $\frac{a_p^2 - pr_h}{p}$ is a unit.
Indeed
in these cases, in the context of 
the first author's zig-zag conjecture, one has
$$\tau := v\left(\frac{a_p^2-pr_0}{pa_p}\right) \lneq 0 \leq v(r_0-1).$$ When $r \equiv 1 \text{ mod } (p-1)$, $v(a_p) = 
\frac{1}{2}$
and $\frac{a_p^2 - pr_h}{p}$ is not a unit, 
it is possible that $\tau \geq v(r_0 -1)$, in which case it is shown in \cite{BGEE2} that 
$\bar\Theta_{k,a_p}$ is not supercuspidal. It would be interesting to understand the analogous behaviour of $\bar\Theta_{k,a_p}$ when $f \geq 2$. 
This would be
tantamount to  understanding the beginnings of a `Hilbert modular version' of the
zig-zag conjecture.
\end{rmk}

\begin{rmk}
The case of $f = 1$, slopes in $(0,1)$ and $p = 2$ has recently been treated by
Bhattacharya-Venugopal \cite{BV}. Their main result shows that Theorem~\ref{thmexcept} is also valid for $f = 1$ and $p = 2$. 
\end{rmk}

\section{Socle of $Q$ for $f = 2$} In this section, we turn
to the next deepest Jordan-H\"older factors in the socle
filtration of $Q$. These factors arise only if $f > 1$. 
For simplicity, we shall
assume that $f = 2$
so that these factors also form the socle of $Q$.
The case of the second deepest 
layer of the socle filtration for general $f$ is deferred to the next section.

\subsection{Generators of Jordan-H\"older factors}

Recall that the socle filtration of $V_{r_0,r_1}/V_{r_0,r_1}^*$ (for $f=2$) is given by the following diagram:
 \begin{center}
\begin{tikzpicture}[baseline=(current bounding box.center)]
 
  \node (A) at (0,2) {$(p-1-a_0,p-1-a_1)\otimes D^{a_0+pa_1}$}; 
  \node (B) at (-3,0) {$(a_0 - 1, p - 2 - a_1) \otimes D^{(1 + a_1)p}$}; 
  \node (C) at (3,0) {$(p-2-a_0,a_1-1)\otimes D^{1+a_0}$}; 
  \node (D) at (0,-2) {$(a_0,a_1)$.};

  \draw[solid] (A) -- (B) -- (D) -- (C) -- (A);

  \node at (0,0) {$\oplus$};
\end{tikzpicture}
\end{center}
A basis of $V_{r_0,r_1}/V_{r_0,r_1}^*$ is given by 
$$\{Y_0^{r_0}Y_1^{r_1}, {X_0^{r_0-j_0}Y_0^{j_0}X_1^{r_1-j_1}Y_1^{j_1}}: 0\le j_0,j_1\le p-1\}.$$  In this section we write down the monomials which generate the above Jordan-H\"older factors of $V_{r_0,r_1}/V_{r_0,r_1}^*$. The computation can be generalized to arbitrary $f$.

We first compute the monomials 
corresponding to the Jordan-H\"older factor $(a_0,a_1).$ Clearly by the proof of \lemref{Xr},  we have $X_0^{r_0}X_1^{r_1}\text{ mod } V_{r_0,r_1}^* \mapsto X_0^{a_0}X_1^{a_1} $ and hence $Y_0^{r_0}Y_1^{r_1} \text{ mod } V_{r_0,r_1}^* \mapsto Y_0^{a_0}Y_1^{a_1}.$ 
We now write down the polynomial (it turns to be a monomial) which corresponds to $X_0^{a_0-i_0}Y_0^{i_0}X_1^{a_1-i_1}Y_1^{i_1} \>(\neq X_0^{a_0}X_1^{a_1}, Y_0^{a_0}Y_1^{a_1}).$
Suppose 
$$\sum\limits_{0\le j_i\le p-1}c_{j_0,j_1}X_0^{r_0-j_0}Y_0^{j_0}X_1^{r_1-j_1}Y_1^{j_1} +c_{r_0,r_1} Y_0^{r_0}Y_1^{r_1} \text{ mod } V_{r_0,r_1}^* \mapsto X_0^{a_0-i_0}Y_0^{i_0}X_1^{a_1-i_1}Y_1^{i_1}.$$ 
Then by $\Gamma$-equivariance, for $\alpha \in {\mathbb F}_q^\times$, we have
{\small
\begin{eqnarray*}
  \left(\begin{smallmatrix}
    1&0\\0&\alpha
\end{smallmatrix}\right)\cdot\bigg(\sum\limits_{0\le j_i\le p-1}c_{j_0,j_1}X_0^{r_0-j_0}Y_0^{j_0}X_1^{r_1-j_1}Y_1^{j_1} +c_{r_0,r_1} Y_0^{r_0}Y_1^{r_1} \bigg) \text{ mod } V_{r_0,r_1}^*    &\mapsto& \left(\begin{smallmatrix}
    1&0\\0&\alpha
\end{smallmatrix}\right)\cdot X_0^{a_0-i_0}Y_0^{i_0}X_1^{a_1-i_1}Y_1^{i_1} , 
\end{eqnarray*}
}
that is,
{\small
\begin{eqnarray*}
    \sum\limits_{0\le j_i\le p-1}c_{j_0,j_1}\alpha^{j_0+pj_1}X_0^{r_0-j_0}Y_0^{j_0}X_1^{r_1-j_1}Y_1^{j_1} +c_{r_0,r_1} \alpha^{r_0+pr_1}Y_0^{r_0}Y_1^{r_1} \text{ mod } V_{r_0,r_1}^* 
& \mapsto &\alpha^{i_0+pi_1}X_0^{a_0-i_0}Y_0^{i_0}X_1^{a_1-i_1}Y_1^{i_1}.
\end{eqnarray*}
}
Also, trivially, we have 
{\small
$$\alpha^{i_0+pi_1}\bigg(\sum\limits_{0\le j_i\le p-1}c_{j_0,j_1}X_0^{r_0-j_0}Y_0^{j_0}X_1^{r_1-j_1}Y_1^{j_1} +c_{r_0,r_1} Y_0^{r_0}Y_1^{r_1}  \bigg) \text{ mod } V_{r_0,r_1}^* 
\mapsto \alpha^{i_0+pi_1} X_0^{a_0-i_0}Y_0^{i_0}X_1^{a_1-i_1}Y_1^{i_1}.$$ 
 }
Hence, since 
we are working with
isomorphisms, we have
\begin{eqnarray*}
&&\sum\limits_{0\le j_i\le p-1}c_{j_0,j_1}\alpha^{j_0+pj_1}X_0^{r_0-j_0}Y_0^{j_0}X_1^{r_1-j_1}Y_1^{j_1} +c_{r_0,r_1} \alpha^{r_0+pr_1}Y_0^{r_0}Y_1^{r_1}\\
&&\qquad=\alpha^{i_0+pi_1}\bigg(\sum\limits_{0\le j_i\le p-1}c_{j_0,j_1}X_0^{r_0-j_0}Y_0^{j_0}X_1^{r_1-j_1}Y_1^{j_1}+c_{r_0,r_1} Y_0^{r_0}Y_1^{r_1} \bigg)\quad \text{ mod } V_{r_0,r_1}^*.
\end{eqnarray*}
Hence, for all $\alpha \in {\mathbb F}_q^\times$, we have 
$$\sum\limits_{0\le j_i\le p-1}c_{j_0,j_1}(\alpha^{j_0+pj_1}-\alpha^{i_0+pi_1})X_0^{r_0-j_0}Y_0^{j_0}X_1^{r_1-j_1}Y_1^{j_1} +c_{r_0,r_1}(\alpha^{r_0+pr_1}-\alpha^{i_0+pi_1})Y_0^{r_0}Y_1^{r_1}=0\text{ mod } V_{r_0,r_1}^*.$$
Note that the monomials in the above equation forms a basis of $V_{r_0,r_1}/V_{r_0,r_1}^*$.
Recall that $r_0+pr_1 \equiv a_0 + pa_1 \text{ mod } (q-1)$ and we have assumed that $i_0+pi_1 \neq a_0+pa_1$. Hence
$c_{r_0,r_1} = 0$. 
Similarly $c_{j_0,j_1}=0$ if $(j_0,j_1)\neq (i_0,i_1)$. Therefore $X_0^{r_0-i_0}Y_0^{i_0}X_1^{r_1-i_1}Y_1^{i_1} \text{ mod } V_{r_0,r_1}^*\mapsto X_0^{a_0-i_0}Y_0^{i_0}X_1^{a_1-i_1}Y_1^{i_1}. $
Summarizing, the generators of socle are given by 
\begin{eqnarray*}
    X_0^{r_0}X_1^{r_1} \text{ mod } V_{r_0,r_1}^*&\mapsto &X_0^{a_0}X_1^{a_1}\\
     Y_0^{r_0}Y_1^{r_1} \text{ mod } V_{r_0,r_1}^*&\mapsto &Y_0^{a_0}Y_1^{a_1}\\
  X_0^{r_0-i_0}Y_0^{i_0}X_1^{r_1-i_1}Y_1^{i_1} \text{ mod } V_{r_0,r_1}^* &\mapsto&      X_0^{a_0-i_0}Y_0^{i_0}X_1^{a_1-i_1}Y_1^{i_1} .
\end{eqnarray*}

A similar computation can be done to compute the generators of say $(a_0-1,p-2-a_1)\otimes D^{(1+a_1)p}$ (going modulo the socle). Consider the monomial  $X_{0}^{a_0-1-i_0}Y_0^{i_0}X_1^{p-2-a_1-i_1}Y_1^{i_1}\in (a_0-1,p-2-a_1)\otimes D^{(1+a_1)p}$.
One can show that 
\begin{eqnarray}
\label{middleproj}
X_0^{r_0-i_0}Y_0^{i_0}X_1^{r_1-i_1-a_1-1}Y_1^{i_1+1+a_1} \text{ mod } V_{r_0,r_1}^*\mapsto X_0^{a_0-1-i_0}Y_0^{i_0}X_1^{p-2-a_1-i_1}Y_1^{i_1}.
\end{eqnarray}

Similarly for any $X_0^{p-2-a_0-i_0}Y_0^{i_0}X_{1}^{a_1-1-i_1}Y_1^{i_1}\in (p-2-a_0,a_1-1)\otimes D^{(1+a_0)},$ one can show that 
$$X_0^{r_0-i_0-a_0-1}Y_0^{i_0+1+a_0}X_{1}^{r_1-i_1}Y_1^{i_1} \text{ mod } V_{r_0,r_1}^* \mapsto X_0^{p-2-a_0-i_0}Y_0^{i_0}X_{1}^{a_1-1-i_1}Y_1^{i_1}.$$

Finally for any $X_0^{p-1-a_0-i_0}Y_0^{i_0}X_1^{p-1-a_1-i_1}Y_1^{i_1}\in (p-1-a_0,p-1-a_1)\otimes D^{a_0+pa_1},$ we have
\begin{eqnarray}
\label{cosocproj}
X_0^{r_0-i_0-a_0}Y_0^{i_0+a_0}X_1^{r_1-i_1-a_1}Y_1^{i_1+a_1} \text{ mod } V_{r_0,r_1}^* \mapsto X_0^{p-1-a_0-i_0}Y_0^{i_0}X_1^{p-1-a_1-i_1}Y_1^{i_1}.
\end{eqnarray}

\subsection{Weights $r \equiv 1, 2, \ldots, p-1 \text{ mod } (q-1)$}  

We shall assume
that there is only one factor in $\mathrm{soc}(Q)$. 
To this end, we assume without loss of generality, that $a_1 = 0$. Then the structure of $Q = \frac{V_{r_0,r_1}}{V_{r_0,r_1}^*+X_{r_0,r_1}}$ is given by 
\begin{center}
\begin{tikzpicture}
\node (L) at (0,2) {$(p-1-a_0,p-1)\otimes D^{a_0}$};
\node (F) at (0,0) {$(a_0-1,p-2)\otimes D^{p}$.};
\draw (L) -- (F);
\end{tikzpicture}
\end{center}
The following theorem studies 
the image in 
$\bar{\Theta}_{k,a_p}$
of the
bottom Jordan-H\"older
factor under the map
\eqref{surjectivemap} under some additional assumptions.

\begin{thm}\label{middle1}
     Let $p >2$. Let $r = r_0 +pr_1 \equiv a = a_0 + p a_1 \textnormal{ mod } (q-1)$ with
     $1 \leq a_0 \leq p-1$ and $a_1 =0$. 
     Assume that
     $p \nmid r_1$ and
     $p \mid r_0$ if
     $a_0 = 1$ and 
     that $v(a_p) \in (\frac{1}{2}, 1)$.
    Then the map $\textnormal{ind}^G_{KZ}
    \mathrm{soc}(Q) \rightarrow \bar{\Theta}_{k_0,k_1,a_p}$ factors through $T.$ 
\end{thm}
\begin{proof}
   Consider the function 
   $$f_0= \frac{1}{p}\bigg[\mathrm{Id},\sideset{}{'}\sum\limits_{\substack{0\le j_0 \le r_0, 0\le j_1\le r_1\\ j_0+pj_1 \equiv p \text{ mod }(q-1)}}\binom{r_0}{j_0}\binom{r_1}{j_1}X_0^{r_0-j_0}Y_0^{j_0} X_1^{r_1-j_1}Y_1^{j_1}\bigg].$$
   Then mod $p$, we have $T^-f_0 =0$ since $a_0 \notequiv 2p \text{ mod } (q-1)$ if $p > 2$. Clearly
   \begin{eqnarray*}
       -a_pf_0&=&-\frac{a_p}{p}\bigg[\mathrm{Id},\sideset{}{'}\sum\limits_{\substack{0\le j_0 \le r_0, 0\le j_1\le r_1\\ j_0+pj_1 \equiv p \text{ mod }(q-1)}}\binom{r_0}{j_0}\binom{r_1}{j_1}X_0^{r_0-j_0}Y_0^{j_0} X_1^{r_1-j_1}Y_1^{j_1}\bigg].
   \end{eqnarray*}
   Now we calculate $T^+f_0.$ By  
 \lemref{gbinomialsum}, the fact that $\binom{q-1}{n} \equiv (-1)^n \text{ mod } p$ and the assumption that $p \mid r_0$ if $a_0 = 1$, we have
   \begin{eqnarray*}
     \sideset{}{'}\sum\limits_{\substack{0\le j_0 \le r_0, 0\le j_1\le r_1\\ j_0+pj_1 \equiv p\text{ mod }(q-1)}}\binom{r_0}{j_0}\binom{r_1}{j_1}& \equiv & \binom{[a_0]}{[p]}+\delta_{q-1,[p]} \>  = 0 \> \text{ mod }p \\
     \sideset{}{'}\sum\limits_{\substack{0\le j_0 \le r_0, 0\le j_1\le r_1\\ j_0+pj_1 \equiv p\text{ mod }(q-1)}}\binom{r_0}{j_0}\binom{r_1}{j_1}\binom{j_0}{1} &\equiv&\binom{r_0}{1}\left(\binom{[a_0-1]}{[p-1]}+\delta_{q-1,[p-1]}\right)\\
     &=&\begin{cases}
         r_0((-1)^{p-1} +0)=r_0 \equiv 0 & \text{ if } a_0= 1 \\
         r_0(0+0)=0 & \text{ if } a_0\neq 1 
     \end{cases} \quad \text{ mod }p\\
    \sideset{}{'}\sum\limits_{\substack{0\le j_0 \le r_0, 0\le j_1\le r_1\\ j_0+pj_1 \equiv p\text{ mod }(q-1)}}\binom{r_0}{j_0}\binom{r_1}{j_1}\binom{j_1}{1}&\equiv&\binom{r_1}{1}\left(\binom{[a_0-p]}{[p-p]}+\delta_{q-1,[p-p]}\right)\\
    &=& r_1(0+1) = r_1 \quad \text{ mod } p.
  \end{eqnarray*}
  The general $\mu$ power in formula \eqref{T+} is $(-[\mu])^{j_0+pj_1 - (l_0+pl_1)} = (-[\mu])^{p-p} = 1$ in the last case. Thus going mod $\mathrm{ind}_{KZ}^G X_{r_0,r_1}$ and mod $p$, by \eqref{T+} we obtain
  $$T^+f_0=r_1\sum\limits_{\mu \in \mathbb F_q}[\left(\begin{smallmatrix}
      p&[\mu]\\0&1
  \end{smallmatrix}\right),
  {X_0^{r_0}X_1^{r_1-1}Y_1}].$$
  Now consider 
  $$f_1=-\frac{a_p}{p}\sum\limits_{\lambda \in \mathbb F_q}[\left(\begin{smallmatrix}
      p&[\lambda]\\0&1
  \end{smallmatrix} \right),[\lambda]^{p-a_0}(Y_0^{r_0}Y_1^{r_1}-X_0^{r_0-a_0}Y_0^{a_0}X_1^{r_1})].$$
  We obtain that $T^+f_1$ and $-a_pf_1$ dies mod $p$ since $v(a_p)>\frac{1}{2}.$
  Now we calculate $T^-f_1$. By \eqref{T-}, we have that mod $p$
  \begin{eqnarray*}
      T^-f_1&=&-\frac{a_p}{p}\sum\limits_{\lambda \in \mathbb F_q}\bigg[\left(\begin{smallmatrix}
      p&[\lambda]\\0&1
  \end{smallmatrix} \right)\left(\begin{smallmatrix}
      1&0\\0&p
  \end{smallmatrix}\right),[\lambda]^{p-a_0}Y_0^{r_0}Y_1^{r_1}\bigg]\\
  &=&-\frac{a_p}{p}\sum\limits_{\lambda \in \mathbb F_q}\bigg[p\left(\begin{smallmatrix}
      1&[\lambda]\\0&1
  \end{smallmatrix} \right),[\lambda]^{p-a_0}Y_0^{r_0}Y_1^{r_1}\bigg]\\
  &=&-\frac{a_p}{p}\bigg[\mathrm{Id},\sum\limits_{\lambda \in \mathbb F_q}[\lambda]^{p-a_0}(\lambda X_0+Y_0)^{r_0}(\lambda^p X_1+Y_1)^{r_1}\bigg]\\
  &=&-\frac{a_p}{p}\bigg[\mathrm{Id},\sum\limits_{\substack{0\le j_0 \le r_0, 0\le j_1\le r_1\\ }}\sum\limits_{\lambda \in \mathbb F_q}[\lambda]^{p-a_0+r_0+pr_1-(j_0+pj_1)}\binom{r_0}{j_0}\binom{r_1}{j_1}X_0^{r_0-j_0}Y_0^{j_0}X_1^{r_1-j_1}Y_1^{j_1}\bigg]\\
  &=&-\frac{(q-1)a_p}{p}\bigg[\mathrm{Id},\sum\limits_{\substack{0\le j_0 \le r_0, 0\le j_1\le r_1\\ j_0+pj_1 \equiv p\text{ mod }(q-1)}}\binom{r_0}{j_0}\binom{r_1}{j_1}X_0^{r_0-j_0}Y_0^{j_0}X_1^{r_1-j_1}Y_1^{j_1}\bigg].
  \end{eqnarray*}
  Hence modulo $\mathrm{ind}_{KZ}^G X_{r_0,r_1}$ and mod $p$, we have 
  \begin{eqnarray*}
      (T-a_p)(f_0+f_1)=r_1\sum\limits_{\mu \in \mathbb F_q}[\left(\begin{smallmatrix}
      p&[\mu]\\0&1
  \end{smallmatrix}\right),{X_0^{r_0}X_1^{r_1-1}Y_1}].
  \end{eqnarray*}
  
  As proved in \eqref{middleproj} in the previous subsection,
  we have
\begin{eqnarray*}
    X_0^{r_0}X_1^{r_1-1}Y_1 & \mapsto &  X_0^{a_0-1}X_1^{p-2}\in (a_0-1,p-2)\otimes D^{p}.
\end{eqnarray*}
Thus $(T-a_p)(f_0+f_1)$ projects onto 
  $$T([\mathrm{Id},X_0^{a_0-1}X_1^{p-2}])=r_1\sum\limits_{\mu \in \mathbb F_q}[\left(\begin{smallmatrix}
      p&[\mu]\\0&1
  \end{smallmatrix}\right), X_0^{a_0-1}X_1^{p-2}].$$
       Hence the theorem follows.
\end{proof}

Now lets assume, $v(a_p)>\frac{1}{2}, r_0+pr_1\equiv pa_1 \text{ mod }(q-1),a_1\neq 1$, $p\nmid r_0$ and $p\mid r_1$ if $a_1=1$. The structure of $Q=\frac{V_{r_0,r_1}}{V_{r_0,r_1}^*+X_{r_0,r_1}}$ is given by 
\begin{center}
\begin{tikzpicture}
\node (L) at (0,2) {$(p-1,p-1-a_1)\otimes D^{a_1p}$};
\node (F) at (0,0) {$(p-2,a_1-1)\otimes D^{}.$};
\draw (L) -- (F);
\end{tikzpicture}
\end{center}

\begin{rmk}
    The cases when $r_0+pr_1\equiv a_1 p$ are Frobenius twist of the cases treated above and a similar calculation can be done.
\end{rmk}

\begin{thm}\label{middle2}
   Let $p >2$. Let $r = r_0 +pr_1 \equiv a = a_0 + p a_1 \textnormal{ mod } (q-1)$ with
     $1 \leq a_1 \leq p-1$ and $a_0 =0$. 
     Assume that
     $p \nmid r_0$ and
     $p \mid r_1$ if
     $a_1 = 1$ and 
     that $v(a_p) \in (\frac{1}{2}, 1)$.
    Then the map $\textnormal{ind}^G_{KZ}
    \mathrm{soc}(Q) \rightarrow \bar{\Theta}_{k_0,k_1,a_p}$ factors through $T.$     
\end{thm}

\section{Weights in the second layer of the socle filtration of $Q$}

We now study the
contribution of the weights in the second layer (just below
the cosocle) in 
the socle filtration 
of $Q$ for general $f \geq 2$.

Again, for simplicity, we assume that $a_1 = \cdots= a_{f-1} = 0$ so that $r\equiv a_0 ~\text{mod } (q-1)$, where $a_0\in \{1,2,\ldots, p-1\}$. The socle filtration of  $V_r/V_r^*$ in this case is given by the following diagram:

\begin{center}
\begin{tikzpicture}[every node/.style={align=center}]
\node (T1) at (0,4) {$(p-1-a_0,p-1,p-1,\ldots,p-1)\otimes D^{a_0}$};
\node (T2) at (0,3) {$(a_0-1,p-2,p-1,p-1,\ldots,p-1)\otimes D^p$};
\node (T3) at (0,2) {$(a_0-1,0,p-2,p-1,\ldots,p-1)\otimes D^{p^2}$};

\node (vd) at (0,1) {$\vdots$};

\node (B1) at (0,0) {$(a_0-1,0,0, \ldots,0,p-2,p-1)\otimes D^{p^{f-2}}$};
\node (B2) at (0,-1) {$(a_0-1,0,0,\ldots,0,p-2)\otimes D^{p^{f-1}}$};
\node (B3) at (0,-2) {$(a_0,0,0,0,\ldots,0)$};

\draw (T1) -- (T2);
\draw (T2) -- (T3);
\draw (T3) -- (vd);
\draw (vd) -- (B1);
\draw (B1) -- (B2);
\draw (B2) -- (B3);
\end{tikzpicture}
\end{center}

The socle filtration $Q_i$ for
$1 \leq i \leq f$ of $Q$ is the same as that given above except that
we drop the bottommost term (it is killed by 
$X_r$). Recall
$F_i$ is the image of $Q_i$ under the
map \eqref{surjectivemap}.  Let $F_s = 
F_{f-1}/F_{f-2}$
(where $s$ stands for `second'). 

\begin{thm}\label{gmiddle1}
Let $p>2$. Let $r\equiv a_0 ~\textnormal{mod }(q-1),$ where $a_0\in \{1,2,3,\ldots, p-1\}$. Assume that $p\nmid r_1$ and $p\mid r_0$ if $a_0=1$ and that $v(a_p)\in (\frac{1}{2},1)$. Then the map 
    $\mathrm{ind}^G_{KZ}Q_{f-1}/Q_{f-2} \twoheadrightarrow F_s$
    factors through $T.$
\end{thm}
\begin{proof}
   Consider the function 
   $$f_0=\frac{1}{p}\bigg[\mathrm{Id},\sideset{}{'}\sum\limits_{\substack{0\le j_i \le r_i\\ j \equiv p \text{ mod }(q-1)}}\binom{r_0}{j_0}\cdots\binom{r_{f-1}}{j_{f-1}}X_0^{r_0-j_0}Y_0^{j_0}\cdots X_{f-1}^{r_{f-1}-j_{f-1}}Y_{f-1}^{j_{f-1}}\bigg].$$
   Then mod $p$, we have $T^-f_0 =0$ since by uniqueness of digits in base $p$ expansions of elements in $[1,q-1]$, we see that $a_0 \notequiv 
p^h+p \text{ mod } (q-1)$ if $p > 2$
for any $0 \leq h \leq f-1$. Clearly
   \begin{eqnarray*}
       -a_pf_0&=&-\frac{a_p}{p}\bigg[\mathrm{Id},\sideset{}{'}\sum\limits_{\substack{0\le j_i \le r_i\\ j \equiv p \text{ mod }(q-1)}}\binom{r_0}{j_0}\cdots\binom{r_{f-1}}{j_{f-1}}X_0^{r_0-j_0}Y_0^{j_0}\cdots X_{f-1}^{r_{f-1}-j_{f-1}}Y_{f-1}^{j_{f-1}}\bigg].
   \end{eqnarray*}
By 
 \lemref{gbinomialsum}, the fact that $\binom{q-1}{n} \equiv (-1)^n \text{ mod } p$ and the assumption that $p \mid r_0$ if $a_0 = 1$, we have
   \begin{eqnarray*}
     \!\!\!\!\!\!\sideset{}{'}\sum\limits_{\substack{0\le j_i \le r_i\\ j \equiv p\text{ mod }(q-1)}}\binom{r_0}{j_0}\cdots\binom{r_{f-1}}{j_{f-1}}&\equiv& \binom{[a_0]}{[p]}+\delta_{q-1,[p]}\equiv0 \text{ mod }p, \\
     \!\! \!\!\!\! \sideset{}{'}\sum\limits_{\substack{0\le j_i \le r_i\\ j \equiv p\text{ mod }(q-1)}}\binom{r_0}{j_0}\cdots\binom{r_{f-1}}{j_{f-1}}\binom{j_0}{1}&\equiv& \binom{r_0}{1}\left[\binom{[a_0-1]}{[p-1]}+\delta_{q-1,[p-1]} \right]\\
   &=&\begin{cases}
   r_0((-1)^{p-1}+0)=r_0 \equiv 0 &\text{if }a_0=1\\
   r_0(0+0)=0 &\text{if }a_0\neq 1
   \end{cases}\text{ mod }p\\
   \!\! \!\!\!\! \sideset{}{'}\sum\limits_{\substack{0\le j_i \le r_i\\ j \equiv p\text{ mod }(q-1)}}\binom{r_0}{j_0}\cdots\binom{r_{f-1}}{j_{f-1}}\binom{j_1}{1}&\equiv& \binom{r_1}{1}\left[\binom{[a_0-p]}{[p-p]}+\delta_{q-1,[p-p]} \right]\\
   &=&r_1(0+1)=r_1 \text{ mod }p\\
   \!\! \!\!\!\! \sideset{}{'}\sum\limits_{\substack{0\le j_i \le r_i\\ j \equiv p\text{ mod }(q-1)}}\binom{r_0}{j_0}\cdots\binom{r_{f-1}}{j_{f-1}}\binom{j_h}{1}&\equiv& \binom{r_h}{1}\left[\binom{[a_0-p^h]}{[p-p^h]}+\delta_{q-1,[p-p^h]} \right]\\
   &=&r_h(0+0)=0 \quad\text { for } 2 \leq h\leq f-1.
  \end{eqnarray*}
    The general $\mu$ power in the formula \eqref{T+} is $(-[\mu])^{j - l} = (-[\mu])^{p-p} = 1$ when $l = p$.  Thus going mod $\mathrm{ind}_{KZ}^G X_{r}$ and mod $p$, by \eqref{T+} we obtain
  $$T^+f_0=r_1\sum\limits_{\mu \in \mathbb F_q}[\left(\begin{smallmatrix}
      p&[\mu]\\0&1
  \end{smallmatrix}\right),X_0^{r_0}X_1^{r_1-1}Y_1X_2^{r_2}\cdots X_{f-1}^{r_{f-1}}].$$
  Now consider $$f_1=-\frac{a_p}{p}\sum\limits_{\lambda \in \mathbb F_q}[\left(\begin{smallmatrix}
      p&[\lambda]\\0&1
  \end{smallmatrix} \right),\lambda^{p-a_0}(Y_0^{r_0}\cdots Y_{f-1}^{r_{f-1}}-X_0^{r_0-a_0}Y_0^{a_0}X_1^{r_1}\cdots X_{f-1}^{r_{f-1}})].$$
  Clearly, $T^+f_1$ and $a_pf_1$ dies mod $p$ since $v(a_p)>\frac{1}{2}$. Now we calculate $T^-f_1.$ By \eqref{T-}, we have
  \begin{eqnarray*}
      T^-f_1&=&-\frac{a_p}{p}\sum\limits_{\lambda \in \mathbb F_q}\bigg[\left(\begin{smallmatrix}
      p&[\lambda]\\0&1
  \end{smallmatrix} \right)\left(\begin{smallmatrix}
      1&0\\0&p
  \end{smallmatrix}\right),\lambda^{p-a_0}Y_0^{r_0}\cdots Y_{f-1}^{r_{f-1}}\bigg]\\
  &=&-\frac{a_p}{p}\sum\limits_{\lambda \in \mathbb F_q}\bigg[p\left(\begin{smallmatrix}
      1&[\lambda]\\0&1
  \end{smallmatrix} \right),\lambda^{p-a_0}Y_0^{r_0}\cdots Y_{f-1}^{r_{f-1}}\bigg]\\
  &=&-\frac{a_p}{p}\bigg[\mathrm{Id},\sum\limits_{\lambda \in \mathbb F_q}\lambda^{p-a_0}(\lambda X_0+Y_0)^{r_0}\cdots(\lambda^{p^{f-1}} X_{f-1}+Y_{f-1})^{r_1}\bigg]\\
  &=&-\frac{a_p}{p}\bigg[\mathrm{Id},\sum\limits_{\lambda \in \mathbb F_q}\sum\limits_{\substack{0\le j_i \le r_i\\ }}\lambda^{p-a_0+r-j}\binom{r_0}{j_0}\cdots \binom{r_{f-1}}{j_{f-1}}X_0^{r_0-j_0}Y_0^{j_0}\cdots X_{f-1}^{r_{f-1}-j_{f-1}}Y_{f-1}^{j_{f-1}}\bigg]\\
  &=&-\frac{(q-1)a_p}{p}\bigg[\mathrm{Id},\sum\limits_{\substack{0\le j_i \le r_i\\ j \equiv p\text{ mod }(q-1)}}\binom{r_0}{j_0}\cdots \binom{r_{f-1}}{j_{f-1}}X_0^{r_0-j_0}Y_0^{j_0}\cdots X_{f-1}^{r_{f-1}-j_{f-1}}Y_{f-1}^{j_{f-1}}\bigg].
  \end{eqnarray*}
   Hence modulo $\mathrm{ind}_{KZ}^G X_{r}$ and mod $p$, we have 
  \begin{eqnarray*}
      (T-a_p)(f_0+f_1)&=&r_1\sum\limits_{\mu \in \mathbb F_q}[\left(\begin{smallmatrix}
      p&[\mu]\\0&1
  \end{smallmatrix}\right),X_0^{r_0}X_1^{r_1-1}Y_1X_2^{r_2}\cdots X_{f-1}^{r_{f-1}}].
  \end{eqnarray*}
 It can be checked  (see \eqref{middleproj} for the case $f = 2$) that
\begin{eqnarray*}
X_0^{r_0}X_1^{r_1-1}Y_1X_2^{r_2}\cdots X_{f-1}^{r_{f-1}} & \mapsto & X_0^{a_0-1}X_1^{p-2}X_2^{p-1}\cdots X_{f-1}^{p-1} \in
(a_0-1,p-2,p-1,\ldots,p-1)\otimes D^p.
\end{eqnarray*}
  Thus $(T-a_p)(f_0+f_1)$ projects onto 
  \begin{eqnarray*}
 T([\mathrm{Id},X_0^{a_0-1}X_1^{p-2}X_2^{p-1}\cdots X_{f-1}^{p-1}) &=&r_1\sum\limits_{\mu \in \mathbb F_q}[\left(\begin{smallmatrix}
      p&[\mu]\\0&1
  \end{smallmatrix}\right), X_0^{a_0-1}X_1^{p-2}X_2^{p-1}\cdots X_{f-1}^{p-1}].
       \end{eqnarray*}
     
Hence the theorem follows.   
\end{proof}

\begin{thm}\label{gmiddle2}
Let $p>2$. Let $r\equiv a_hp^h ~\textnormal{mod }(q-1),$ where $a_h\in \{1,2,3,\ldots, p-1\}$. Assume that $p\nmid r_{h+1}$ and $p\mid r_{h}$ if $a_h=1$ and that $v(a_p)\in (\frac{1}{2},1)$. Then the map 
    $\mathrm{ind}^G_{KZ} Q_{f-1}/Q_{f-2} \twoheadrightarrow F_s$
    factors through $T$.
\end{thm}

\section{Appendix}

The goal of this section is to show that 
the socle of the universal supersingular representation is zero if $F$ is a totally ramified extension of ${\mathbb Q}_p$, not equal to ${\mathbb Q}_p$. 
This was proved by Schein \cite[Remark 3.10]{SchJNT}. We 
provide some details. We thank A. Chitrao, A. Jana and M. Sheth for several useful inputs. We shall assume that the underlying weight is trivial. 

Let $\pi$ be a uniformizer of $F$.
Recall that for $m\ge 0$ and $\lambda \in I_m$, we have $g^0_{m,\lambda}=\left(\begin{matrix}
    \pi^m&\lambda\\0& 1
\end{matrix}\right)$ and $g^1_{m,\lambda}=\left(\begin{matrix}
    1&0\\ \pi \lambda& \pi^{m+1}
\end{matrix}\right)$. In this appendix, let us call the part
of the tree emanating
from $g^0_{0,0} = \mathrm{Id}$ to be the 
positive part of the tree,
and the part emanating
away from $g_{0,0}^1 =\left(\begin{matrix}
    1&0\\0& \pi
\end{matrix}\right)=: \alpha$ to be the negative part of the tree. Thus,
both $\mathrm{id}$ ($=0^+$) and $\alpha$ ($=0^-$) are considered as origins of the tree!
Let $\beta:=\left(\begin{matrix}
    0&1\\ \pi &0
\end{matrix}\right)=\alpha w$. 

\begin{lem}
    Let $\sigma$ be a weight of $\mathrm{GL}_2(\mathbb F_q)$. Then the $G$-socle of $(\mathrm{ind}_{KZ}^G\sigma)/T$ is zero.
\end{lem}
\begin{proof}
It is enough to show that if $U$ is a nonzero $G$-submodule of $(\mathrm{ind}_{KZ}^G\sigma)/T$, then $U$ is reducible. Let $0\neq \tau \subset \mathrm{soc}_K(U)$. Then there exists a $K$-map $\tau \hookrightarrow U|_K$. By Frobenius reciprocity, we get a nonzero $G$-map
$\rho: \mathrm{ind}_{KZ}^G\tau \rightarrow U$. But $\mathrm{ind}_{KZ}^G\tau$ is generated as a $G$-module by the $I(1)$-invariant function $[\mathrm{Id}, X_0^{r_0}\cdots X_{f-1}^{r_{f-1}}]$, where $\tau=V_{{r}}\otimes D^s$ (note this $f = 1$ for
Schein). Therefore, the image of $\mathrm{ind}_{KZ}^G\tau$ is also generated by a single $I(1)$-invariant vector. Hence, if we can show that, given a (nonzero) $I(1)$-invariant vector $f$ (in $U$), there exists another (nonzero) $I(1)$-invariant vector $f'$ such that the $G$-module generated by $f'$ is strictly contained in the $G$-module generated by $f$, we will be done. 

We claim we may assume that $I$ acts on $f$ by a character of $I/I(1)$. Indeed 
$$f=(f_{\chi})_{\chi}\in U^{I(1)}=\oplus_{\chi}U_{\chi}$$
where $U_\chi$ is the 
eigenspace corresponding
to $\chi \in \widehat{I/I(1)}$.
Since the $G$-module generated by $f$ contains $f_\chi$ (by the formula for the projector corresponding to $\chi$), and since $f \neq 0$ implies $f_{\chi}\neq 0$ for some $\chi$,  we may replace $f$ by $f_\chi$ for such a $\chi$, proving the claim. 

For simplicity, we assume that $\sigma=\mathbbm 1$. Then by Schein \cite[Theorem 3.8]{SchJNT} with $r = 0$  (see also \cite[Theorem 1.2]{Hen} and \cite[Theorem 
2.1]{CJS}) a basis of $I(1)$-invariant vectors of $(\mathrm{ind}_{KZ}^G\sigma)/T$ is given by
$$\{[\mathrm{Id},1], [\alpha,1] \} \cup \{ s_n^{p^l}, \beta s_n^{p^l}\}_{\substack{n\ge 2\\0\le l\le f-1}},$$
where $s_n^k=\sum\limits_{\mu \in I_n}\left[\left(\begin{matrix}
    \pi^n&\mu\\0&1
\end{matrix}\right), \mu_{n-1}^k\right]$, $\mu=[\mu_0]+\cdots +[\mu_{n-1}]\pi^{n-1}$ and where $I$ acts on $[\mathrm{Id},1]$ and $[\alpha,1]$ by $\chi = 1$ and $I$ acts on $s_n^{p^l}$ by $\chi = a^{-p^l}d^{p^l}$ (and $I$ acts on $\beta s_n^{p^l}$ by $\chi = d^{-p^l}a^{p^l}$) by
\cite[Lemma 3.6]{Hen} or \cite[Lemma 3.13]{CJS}.  Clearly, all these characters $\chi$ are distinct.

Hence by the remarks above, we may assume that either
\begin{enumerate}[label=(\roman*)]
    \item $f$ is a linear combination of $[\mathrm{Id},1]$ and $[\alpha,1]$, or 
    \item $f$ is a linear combination of $s_n^{p^l}$ for fixed $l$. 
    \item $f$ is a linear combination of $\beta s_n^{p^l}$ for fixed $l$. 
\end{enumerate}
Case (iii) follows from case (ii) since 
if $f$ is a linear combination of $\beta s_n^{p^l}$, then $\beta f$ is as in case (ii),
so if there exists such
a $f'$ for $\beta f$, then $\beta f'$ works for $f$. We now treat cases 
(i) and (ii).

\vspace{.2cm}
{\bf \noindent Case (i) \>($\chi=1$)}: If $f$ is a scalar multiple of $[\mathrm{Id},1]$ or $[\alpha,1]$, we are done since these vectors generate $(\mathrm{ind}_{KZ}^G {\mathbbm 1})/T$. Hence $U=(\mathrm{ind}_{KZ}^G{\mathbbm 1})/T$ is reducible since $F \neq {\mathbb Q}_p$. Thus, we may assume $f=[\alpha,1]+c[\mathrm{Id},1]$ for some $c\in \bar{\mathbb F}_p^{\times}$. Since $G$ acts transitively on  the oriented edges of the tree, by translating by an appropriate $g \in G$, we may assume that $[\mathrm{Id},1]+c[g_{1,\lambda}^0,1]$ 
 belongs to the image of $\rho$, $\forall \lambda \in I_1$. Hence $-c[\mathrm{Id},1]-c^2[g_{1,\lambda}^0,1]$
 belongs to the image of $\rho$, $\forall \lambda \in I_1$. Adding, we get that $[\alpha,1]-c^2[g_{1,\lambda}^0,1]$ belongs to the image of $\rho$,  $ \forall \lambda \in I_1 $. Adding each of these functions to $c^2T([\mathrm{Id},1])=c^2[\alpha,1]+c^2\sum\limits_{\lambda\in I_1}[g_{1,\lambda}^0,1]$, we get $c^2[\alpha,1]$. Thus, modulo $T$, we see that $[\alpha,1]$ belongs to the image of $\rho$. So again $U=(\mathrm{ind}_{KZ}^G\sigma)/T$ is reducible.

\noindent {\bf Case (ii) \> ($\chi\neq 1$)}: In this case $f=\sum\limits_{n=2}^N c_ns_{n}^{p^l}$ for some constants $c_n \in \bar{\mathbb F}_p$ with $c_N$ nonzero. We claim that the $G$-module generated by 
$$f'=\sum\limits_{n=2}^N c_ns_{n+1}^{p^l}$$
is strictly contained in the $G$-module generated by $f$. The containment is clear since
$$s_{n+1}^{p^l}=\sum\limits_{\lambda \in I_1}\left(\begin{matrix}
    \pi&\lambda\\0&1
\end{matrix}\right)s_{n}^{p^l}.$$
If equality holds, then we may write 
$$f=\sum\limits_{i=1}^td_i g_i f' \quad (\text{mod }T) ~~~~~\text{for some  }d_i\in \bar{\mathbb F}_p^{\times} \text{ and }g_i\in G.$$
Therefore 
\begin{eqnarray}
    \label{image of T} 
    \sum\limits_{i=1}^td_i g_i f' -f\in \mathrm{Im}(T).
\end{eqnarray}
Suppose $t = 1$ and $g_i=g_1=g_{m,\lambda}^0$ for some $\lambda \in I_m$ with $m\geq 0$. Then $g_if'$ has maximal radius (distance from $\mathrm{Id}$) $m+N+1$ in the positive part of the tree. But $f$ has maximal radius $N <m+N+1$. Consider the function 
\begin{eqnarray*}
g_i s_{N+1}^{p^l}&=&\sum\limits_{\mu\in I_{N+1}}\left[\left(\begin{matrix}\pi^m&\lambda \\0&1\end{matrix}\right)\left(\begin{matrix}\pi^{N+1}&\mu \\0&1\end{matrix}\right),\mu_{N}^{p^l}  \right]\\
&=&\sum\limits_{\mu\in I_{N+1}}\left[\left(\begin{matrix}\pi^{m+N+1}&\lambda+\pi^m\mu \\0&1\end{matrix}\right),\mu_{N}^{p^l}  \right].
\end{eqnarray*}
This function is supported on radius $m+N+1$. We have the following elementary lemma.

\begin{lem}\label{maximalradius}
    Let $h\in\mathrm{ind}_{KZ}^G\mathbbm 1$. If $h$ is in the image of $T$, then the values of $h$ at the vertices of  maximal radius (both in the positive and negative directions) are constant on all vertices that are connected to a given vertex of one  radius less.
\end{lem}
\begin{proof}
    This is clear from the definition of the Hecke operator $T$ as a sum-of-neighbours operator.
\end{proof}

The value of $d_ig_ic_Ns_{N+1}^{p^l}$ at the vertices of maximal radius is $d_ic_N\mu_N^{p^l}$. By \lemref{maximalradius}, 
$$c=d_ic_N\mu_N^{p^l}$$
is a constant $\forall \mu_N\in \mathbb F_q$.
Putting $\mu_N=0$, we get $c=0$ and by putting $\mu_N=1$, we  get $d_ic_N=0$. This is a contradiction since neither $d_i$ nor $c_N$ is zero.

Continuing with the case 
$t = 1$, we now assume that  $g_i=g^0_{m,\lambda}k$ for some $k\in K$. Though
we wish to treat the case
$n = N$ below, we will need
the discussion below later for any $n$, so 
assume $n \geq 2$.
Since $s^{p^l}_{n+1}$ is an $I$-eigenvector, as far as the radius computation is concerned, the element $k$ only matters modulo $I$. Thus, we may assume that
$$
k=
\begin{pmatrix}
1 & 0\\
\kappa & 1
\end{pmatrix} \text{ for $0 \neq \kappa \in I_1$ or } k =  w.$$
To compute the 
radius of $g_i s_{n+1}^{p^l}$, we need to compute
$$k \begin{pmatrix}
\pi^{n+1} & \mu\\
0 & 1
\end{pmatrix}
= k \begin{pmatrix} \pi & [\mu_0]\\ 0 & 1 \end{pmatrix} \begin{pmatrix} \pi^n & \frac{\mu-[\mu_0]}{\pi}\\ 0 & 1
\end{pmatrix}$$
for $\mu \in I_{n+1}$.

First assume that
\[
k=
\begin{pmatrix}
1 & 0\\
\kappa & 1
\end{pmatrix}.
\]
By \cite[p.~16]{Hen}, we have
$$
\begin{pmatrix}
1 & 0\\
\kappa & 1
\end{pmatrix} \begin{pmatrix} 
\pi & [\mu_0]\\
0 & 1
\end{pmatrix}
=
\begin{cases}
\begin{pmatrix}
\pi & \nu\\
0 & 1
\end{pmatrix}
\begin{pmatrix}
1-\kappa\nu & 0\\
\kappa \pi & (1-\kappa\nu)^{-1}
\end{pmatrix}
& \text{if } [\mu_0]\neq -\kappa^{-1},\\[1.5em]
\alpha w
\begin{pmatrix}
\kappa & 0\\
\pi & -\kappa^{-1}
\end{pmatrix}
& \text{if } [\mu_0]=-\kappa^{-1},
\end{cases}
$$
where $\nu=[\mu_0](1+\kappa[\mu_0])^{-1}$. Note that $[\mu_0]\neq -\kappa^{-1} \iff \nu \neq \kappa^{-1}$. 
 Thus, 
 we may write
\begin{eqnarray*}
g_i s_{n+1}^{p^l}&=&g^{0}_{m,\lambda}
\sum_{[\mu_0]\neq -\kappa^{-1}}
\sum_{\frac{\mu-[\mu_0]}{\pi}\in I_n}
\left[
\begin{pmatrix}
\pi & \nu\\
0 & 1
\end{pmatrix}
\begin{pmatrix}
1-\kappa\nu & 0\\
\kappa \pi & (1-\kappa\nu)^{-1}
\end{pmatrix}
\begin{pmatrix}
\pi^{n}&\frac{\mu-[\mu_0]}{\pi} \\
0 & 1
\end{pmatrix},
\mu_n^{p^{l}}
\right]\\
&&+g^{0}_{m,\lambda} \sum_{\frac{\mu-[\mu_0]}{\pi}\in I_n}
\left[
\alpha w
\begin{pmatrix}
\kappa & 0\\
\pi & -\kappa^{-1}
\end{pmatrix}
\begin{pmatrix}
\pi^{n}&\frac{\mu-[\mu_0]}{\pi} \\
0 & 1
\end{pmatrix}, \mu_n^{p^l}
\right]\\
&=& g^{0}_{m,\lambda}
\sum_{
[\mu_0]\neq -\kappa^{-1}}
\begin{pmatrix}
\pi & \nu\\
0 & 1
\end{pmatrix}
\begin{pmatrix}
1-\kappa\nu & 0\\
\kappa \pi & (1-\kappa\nu)^{-1}
\end{pmatrix}
\sum_{\mu'\in I_n}
\left[
\begin{pmatrix}
\pi^{n}&\mu' \\
0 & 1
\end{pmatrix}, {\mu'}_{n-1}^{p^l}
\right]\\
&& +g^{0}_{m,\lambda}\alpha w
\begin{pmatrix}
\kappa & 0\\
\pi & -\kappa^{-1}
\end{pmatrix}
\sum_{\mu'\in I_n}
\left[
\begin{pmatrix}
\pi^{n}&\mu'\\
0 & 1
\end{pmatrix}, {\mu'}_{n-1}^{p^l}
\right]\\
&=&g^{0}_{m,\lambda}
\sum_{
[\mu_0]\neq -\kappa^{-1}}
\begin{pmatrix}
\pi & \nu\\
0 & 1
\end{pmatrix}
\begin{pmatrix}
1-\kappa\nu & 0\\
\kappa \pi & (1-\kappa\nu)^{-1}
\end{pmatrix} s_n^{p^l}+ g^{0}_{m,\lambda}\alpha w
\begin{pmatrix}
\kappa & 0\\
\pi & -\kappa^{-1}
\end{pmatrix}s_n^{p^l}.
\end{eqnarray*}
But $\left(\begin{matrix}
    a&b\\c&d
\end{matrix}\right)\in I$ acts on $s_n^{p^l}$ by $\left(d/a\right)^{p^l}.$ So we have  
\begin{equation}\label{s_{n+1}}
    g_i s_{n+1}^{p^l}=\sum\limits_{\substack{\nu\in I_1\\ [\mu_0]\neq -\kappa^{-1}}} (1-\kappa \nu)^{-2p^l} g^0_{m+1,\lambda+\pi^m\nu}s_n^{p^l} + (-\kappa^{-2p^l})g^0_{m,\lambda}\beta s_n^{p^l}.
\end{equation}
Observe that in the above equation, the first summation on the right-hand side has support in radius $m+n+1$. We claim that the second term is supported on radius $\le m+n-1$ or on the negative part of the tree. Indeed, this is obvious if $m = 0$, so assume that $m > 0$. Noting that
\begin{eqnarray}
\label{alpha-trick}
g^0_{m,\lambda}\alpha = g^{0}_{m-1,[\lambda]_{m-1}} \begin{pmatrix}
1 & [\lambda_{m-1}]\\
0 & 1
\end{pmatrix} \pi,
\end{eqnarray}
we have
\begin{eqnarray*}g^0_{m,\lambda}\beta s_n^{p^l}&=&g^{0}_{m-1,[\lambda]_{m-1}}
\begin{pmatrix}
1 & [\lambda_{m-1}]\\
0 & 1
\end{pmatrix}
ws_n^{p^l}\\
&=&g^{0}_{m-1,[\lambda]_{m-1}}
\begin{pmatrix}
1 & [\lambda_{m-1}]\\
0 & 1
\end{pmatrix}\left(\sum\limits_{0 \neq \mu \in I_1}w\begin{pmatrix}
    \pi&\mu\\0 &1
\end{pmatrix}s_{n-1}^{p^l}+w\begin{pmatrix}
    \pi&0\\0 &1
\end{pmatrix}s_{n-1}^{p^l}\right).\end{eqnarray*}
Using the identities from \cite[p. 16]{Hen}

\begin{eqnarray}
\label{Hendel2}
w g^{0}_{1,\mu}
=
\begin{cases}
g^{0}_{1,\mu^{-1}}
\begin{pmatrix}
-\mu^{-1} & 0\\
\pi & \mu
\end{pmatrix} & \text{if } \mu \neq 0,\\
\beta & \text{if } \mu = 0,
\end{cases}
\end{eqnarray}
we get
\begin{eqnarray*}\label{gmbetasn}
    g^0_{m,\lambda}\beta s_n^{p^l}&=&g^{0}_{m-1,[\lambda]_{m-1}}
\left(\begin{matrix}
1 & [\lambda_{m-1}]\\
0 & 1
\end{matrix}\right)\sum\limits_{\mu\neq 0}g^0_{1,\mu^{-1}}\left(\begin{matrix}
    -\mu^{-1}&0\\ \pi &\mu
\end{matrix}\right) s_{n-1}^{p^l}+g^{0}_{m-1,[\lambda]_{m-1}}
\begin{pmatrix}
1 & [\lambda_{m-1}]\\
0 & 1
\end{pmatrix}\beta s_{n-1}^{p^l}.
\end{eqnarray*}
If $n = 2$, then we continue the 
argument from \eqref{n=2} with $m'$ there equals $m$. So assume $n \gneq 2$.
Noting that in general for
$* \in I_1$ 
and $\lambda \in I_m$, there
is a $*' \in I_1$ and a $\lambda' \in I_m$ such that
\begin{equation}\label{g^0identity}
 \begin{pmatrix}
1 & *\\
0 & 1
\end{pmatrix} g^0_{m,\lambda}=  g^0_{m,\lambda'} \begin{pmatrix}
1 & *'\\
0 & 1
\end{pmatrix} 
\end{equation}
and that $I$ acts on $s_{n-1}^{p^l}$ by the character $(d/a)^{p^l}$ (since $n > 2$)
and $\beta s_{n-1}^{p^l}$ is $I(1)$-invariant,
we get that 
\begin{eqnarray*}
  g^0_{m,\lambda}\beta s_n^{p^l}&=& g^{0}_{m-1,[\lambda]_{m-1}} \sum\limits_{\mu\neq 0}g^0_{1,(\mu^{-1})'}(-\mu^{2p^l})s_{n-1}^{p^l}+g^{0}_{m-1,[\lambda]_{m-1}}
\beta s_{n-1}^{p^l}.
\end{eqnarray*}
The first sum on the right-hand side of the above equation is supported on radius $m+n-1$, and we are reduced to proving the claim for $g^0_{m-1,[\lambda]_{m-1}}\beta s_{n-1}^{p^l}.$ Repeating this argument, we see that $ g^0_{m,\lambda}\beta s_n^{p^l}$ is a linear combination of functions supported on radius $\le m+n-1$ and functions of the form $g^0_{m-t,[\lambda]_{m-t}}\beta s_{n-t}^{p^l}$ for some $t \geq 0$. If $m\le n-2$, then $g^0_{m-t,[\lambda]_{m-t}}\beta s_{n-t}^{p^l}$ eventually becomes $\beta s_{n-m}^{p^l}$ with $n-m\ge 2$, which is in the 
 negative part of the tree (of radius at most $n$). Otherwise $m > n-2$ and eventually at $t=n-2$, we have  
$g^0_{m-t,[\lambda]_{m-t}}\beta s_{n-t}^{p^l}=g^0_{m',[\lambda]_{m'}}\beta s_{2}^{p^l}$
for $m' = m-n+2$. By \eqref{alpha-trick}, \eqref{Hendel2}, 
\begin{eqnarray}\label{n=2}
    g^0_{m',[\lambda]_{m'}}\alpha w s_{2}^{p^l}&=& g^{0}_{m'-1,[\lambda]_{m'-1}}
\begin{pmatrix}
1 & [\lambda_{m'-1}]\\
0 & 1
\end{pmatrix}\left(\sum\limits_{\mu\neq 0}w\begin{pmatrix}
    \pi&\mu\\0 &1
\end{pmatrix}s_{1}^{p^l}+w\begin{pmatrix}
    \pi&0\\0 &1
\end{pmatrix}s_{1}^{p^l}\right)\nonumber \\
&=&g^{0}_{m'-1,[\lambda]_{m'-1}}
\left(\begin{matrix}
1 & [\lambda_{m'-1}]\\
0 & 1
\end{matrix}\right)\sum\limits_{\mu\neq 0}g^0_{1,\mu^{-1}}\left(\begin{matrix}
    -\mu^{-1}&0\\ \pi &\mu
\end{matrix}\right) s_{1}^{p^l}+g^{0}_{m'-1,[\lambda]_{m'-1}}
\begin{pmatrix}
1 & [\lambda_{m'-1}]\\
0 & 1
\end{pmatrix}\beta s_{1}^{p^l}.\nonumber \\
\end{eqnarray}
First, we prove that the first sum is supported on radius $\le m+n-1$. We note the following identity. For $\nu \in I_1$,
\begin{eqnarray*}
\left(\begin{matrix}
    -\mu^{-1}&0\\ \pi &\mu
\end{matrix}\right) \begin{pmatrix}
    \pi &\nu\\0 &1
\end{pmatrix} &=&\left(\begin{matrix}
    -\mu^{-1}&0\\0 &\mu
\end{matrix}\right)\left(\begin{matrix}
    1&0\\\mu^{-1}\pi &1
\end{matrix}\right)\begin{pmatrix}
    \pi &\nu\\0 &1
\end{pmatrix}\\
&=&\left(\begin{matrix}
    -\mu^{-1}&0\\0 &\mu
\end{matrix}\right)\begin{pmatrix}
    \pi&\nu\\0 &1
\end{pmatrix}\begin{pmatrix}
    1-\pi\mu^{-1}\nu &-\mu^{-1}\nu^2\\ \pi^2\mu^{-1} &1+\pi\mu^{-1}\nu
\end{pmatrix} =  \begin{pmatrix}
    \pi&-\nu\mu^{-2}\\0 &1
\end{pmatrix} k'
\end{eqnarray*}
for some $k' \in K$.
Using the above identity, we see that the first sum equals
\begin{eqnarray*}
    g^{0}_{m'-1,[\lambda]_{m'-1}}
\left(\begin{matrix}
1 & [\lambda_{m'-1}]\\
0 & 1
\end{matrix}\right)\sum\limits_{\mu\neq 0}g^0_{1,\mu^{-1}}\sum\limits_{\nu\in I_1}\left[\left(\begin{matrix}
    \pi&-\nu\mu^{-2}\\ 0&1
\end{matrix}\right), \nu^{p^l}\right].
\end{eqnarray*}
Using \eqref{g^0identity} twice, we see that the above sum is supported on radius $m'+1 \leq m+n-1$. Returning to the second term, we first note the following two identities:
\begin{eqnarray*}
  \left(\begin{matrix}
1 & [\lambda_{m'-1}]\\
0 & 1
\end{matrix}\right)\beta&=& \beta  \left(\begin{matrix}
1 & 0\\
[\lambda_{m'-1}] \pi & 1
\end{matrix}\right),\\
\left(\begin{matrix}
1 & 0\\
[\lambda_{m'-1}]\pi & 1
\end{matrix}\right)\begin{pmatrix}
    \pi&\nu\\0 &1
\end{pmatrix}&=&\begin{pmatrix}
    \pi&\nu\\0 &1
\end{pmatrix}\begin{pmatrix}
    1-\pi[\lambda_{m'-1}]\nu &-[\lambda_{m'-1}]\nu^2\\ \pi^2[\lambda_{m'-1}] &1+\pi[\lambda_{m'-1}]\nu
\end{pmatrix},
\end{eqnarray*}
for $\nu \in I_1$.
Using these two identities, \eqref{alpha-trick}, \eqref{Hendel2} and \eqref{g^0identity}, we have 
\begin{eqnarray*}
    g^{0}_{m'-1,[\lambda]_{m'-1}}
\begin{pmatrix}
1 & [\lambda_{m'-1}]\\
0 & 1
\end{pmatrix}\beta s_{1}^{p^l}&=&  g^{0}_{m'-1,[\lambda]_{m'-1}}
\beta  \left(\begin{matrix}
1 & 0\\
[\lambda_{m'-1}]\pi & 1
\end{matrix}\right) \sum\limits_{\nu\in I_1}\left[\left(\begin{matrix}
    \pi&\nu\\ 0&1
\end{matrix}\right), \nu^{p^l}\right]\\
&=&g^{0}_{m'-1,[\lambda]_{m'-1}} \alpha w\sum\limits_{\nu\in I_1}\left[\left(\begin{matrix}
    \pi&\nu\\ 0&1
\end{matrix}\right), \nu^{p^l}\right]\\
&=&g^{0}_{m'-2,[\lambda]_{m'-2}} \begin{pmatrix}
1 & [\lambda_{m'-2}]\\
0 & 1
\end{pmatrix} \sum\limits_{0 \neq \nu\in I_1} \left[g^0_{1,\nu^{-1}}\left(\begin{matrix}
    -\nu^{-1}&0\\ \pi &\nu
\end{matrix}\right), \nu^{p^l}\right]\\
&=&g^{0}_{m'-2,[\lambda]_{m'-2}} \sum\limits_{0 \neq \nu \in I_1} \left[g^0_{1,(\nu^{-1})'}, \nu^{p^l}\right],
\end{eqnarray*}
which is supported in radius $m'-1$ (at least if $m' \geq 2$; if $m' = 1$, then already the function in the second line in the display above is in the negative part of the tree, of radius $1$). This proves the claim that $g^0_{m,\lambda}\beta s_n^{p^l}$ is supported in radius $\le m+n-1$ or in the negative part of the tree.

If $k=w$, by \eqref{Hendel2}, we have
\begin{eqnarray}\label{ws_{n+1}}
g_is_{n+1}^{p^l}&=&g^0_{m,\lambda}w\sum\limits_{\nu \in I_1} \left(\begin{matrix}
    \pi&\nu\\ 0&1
\end{matrix}\right)s_n^{p^l} \nonumber\\
&=& g^0_{m,\lambda} \sum\limits_{0 \neq \nu \in I_1} g^0_{1,\nu^{-1}}\left(\begin{matrix}
    -\nu^{-1}&0\\ \pi&\nu
\end{matrix}\right)s_n^{p^l}+g^0_{m,\lambda}\beta s_n^{p^l}.
\end{eqnarray}
The first sum is supported in radius $m+n+1$ and as just shown above the second term $g^0_{m,\lambda}\beta s_n^{p^l}$ is supported in radius $\le m+n-1$ or in the negative part of the tree. 

Let us now return to the case of maximal radius when $t = 1$
and there is a unique $g_i$ of the form $g_{m,\lambda}^0 k$ with
 $k \in K$. 
Take $n = N$ in the formulas above
and follow the earlier argument for $k = \mathrm{Id}$ which uses the constancy of the value of $
\sum_{i=1}^{t} d_i g_i f'-f$
with respect to the final digit $[\mu_{N-1}]$ at a  vertex with radius $m+N+1$. If 
$k = \left(\begin{matrix}
    1&0\\\kappa&1
\end{matrix}\right)$ with $\kappa \neq 0$, for fixed
$\nu \neq \kappa^{-1}$
we obtain from \eqref{s_{n+1}} that
$$ d_ic_N(1-\kappa \nu)^{-2p^l}\mu_{N-1}^{p^l}=c$$
is a constant, which again leads to a contradiction by varying
$\mu_{N-1}$. If $k=w$,
then for $\nu \neq 0$, one obtains from 
\eqref{ws_{n+1}} that
$$ d_ic_N(- \nu^{2p^l})\mu_{N-1}^{p^l}=c $$
is a constant, leading to a similar
contradiction.

It is not so clear what happens if there is more than one $g_i$. In this case, it is possible that everything in radius $m+N+1$ cancels. This cannot be ruled out. (Though if two or three of the $g_i$ are
of the form $g_{m,\lambda}^0 k_i$
with the $k_i$ non-trivial and lower unipotent, then we can rule this out. Here is the proof when there are three $d_i \neq 0$. Say $k_i = \begin{pmatrix} 1 &  0 \\  \kappa_i & 1 \end{pmatrix}$
with $\kappa_i \neq 0$.
Then by \eqref{image of T} and \eqref{s_{n+1}}, we
have that the function
$$\sum_{i=1}^3 d_i \sum_{\kappa_i^{-1} \neq \nu_i \in I_1} (1- \kappa_i \nu_i)^{-2p^l} \sum_{\mu_i \in I_N} [g^0_{m+N+1, \lambda + \pi^m \nu_i + \pi^{m+1} \mu_i}, \mu_{i,N-1}^{p^l}]$$
plus some terms of lesser (or negative) radius is in the image of $T$.
The value of this function
at $g^0_{m+N+1, \lambda + \pi^m \nu + \pi^{m+1} \mu}$ is given
by 
$$ \sum_{i=1}^3 d_i (1- \kappa_i \nu)^{-2p^l} \mu_{N-1}^{p^l}$$
where we have taken $\nu_i = \nu$ and $\mu_i = \mu$
and where we drop the terms
where $\nu = \kappa_i^{-1}$.
As above, this must be constant
as $\mu_{N-1}$ varies, hence
we obtain the system of 
equations
$$\sum_{i=1}^3 d_i (1- \kappa_i \nu)^{-2p^l} = 0 $$
for all $\nu$ and in particular for  $\nu = \kappa_j^{-1}$ for
$j = 1,2,3$. The coefficient
matrix essentially has the 
form
$$\begin{pmatrix} 
0 & \frac{1}{(a-b)^2} & \frac{1}{(a-c)^2} \\ 
\frac{1}{(a-b)^2} & 0 & \frac{1}{(b-c)^2} \\ 
\frac{1}{(a-c)^2} & \frac{1}{(b-c)^2} & 0 
\end{pmatrix}$$
which turns out to be invertible if $p \neq 2$, forcing
the $d_i$ to be all $0$, a contradiction.)

Let us turn to the
case of many $g_i$.
We make a remark. Suppose that some $g_i=g_{m',\lambda'}^1k$
for some $k \in K$. Note that $g_{m',\lambda'}^1k=\beta g_{m',\lambda'}^0wk$. As above, $g_{m',\lambda'}^0wk s_{N+1}^{p^l}$ is supported on positive radius, and negative radius $\le N.$ Hence, $\beta g_{m',\lambda'}^0wk s_{N+1}^{p^l}$ is supported on negative radius, and positive radius $\le N$. Thus such $g_i$ do not intervene in any computation involving vertices of positive radius
$> N$.

 Now, fix $m$ maximal among all $g_i$ of the form $g^0_{m,\lambda} k$. Suppose that all the terms at all vertices in radius $m+N+1$ cancel each other in the sum 
 \begin{eqnarray}\label{sumdi}
    c_N \sum\limits_{i}d_ig_i s_{N+1}^{p^l}&=&c_Nd_{0} g^0_{m,\lambda}s^{p^l}_{N+1}\nonumber \\
    &&+  \>\> c_N \sum\limits_{i}{}'d_i \bigg( \sum\limits_{\substack{\nu_j\in I_1\\  }} (1-\kappa_i \nu_j)^{q-1-2p^l} g^0_{m+1,\lambda+\pi^m\nu_j}s_N^{p^l} + (-\kappa_i^{-2p^l})g^0_{m,\lambda}\beta s_N^{p^l}\bigg) \nonumber \\
     &&+ \>\> c_N d_w\bigg(g^0_{m,\lambda} \sum\limits_{0 \neq \nu_j \in I_1} g^0_{1,\nu_j^{-1}}(-\nu_j^{2p^l})s_N^{p^l}+g^0_{m,\lambda}\beta s_N^{p^l}\bigg)  
 \end{eqnarray}
  where (by the remark in the previous paragraph) the $g_i$'s
 are of the form $g^0_{m,\lambda}k$
 for $k \in K$ and some $\lambda \in I_m$ (which we may assume is fixed, since the computation below does not depend on $\lambda$). Here the term on the right-hand side in the first line is meant to correspond to $k = \mathrm{Id}$, the terms on the second line to $k = \left(\begin{matrix}
    1&0\\ \kappa_i &1 
 \end{matrix}\right)$ with $\kappa_i \neq 0$ and the term on the third line to $k = w$. Note that  since $I$ acts by a character, by modifying the $d_i$ on the right-hand side if necessary, and dropping any that vanish, we may and do assume that the $k$'s are these coset representatives mod $I$.

 Note that $q = p^f = p$. The cancellation means that $$c_N\bigg(d_0+\sum\limits_{i}^{}{}'d_i(1-\kappa_i\nu)^{q-1-2p^l}+d_w (-\nu^{q-1-2p^l})\bigg)\mu_{N-1}^{p^l}=0$$
 for every $\mu_{N-1}\in I_1$ and $\nu\in I_1$. Taking $\mu_{N-1}=1$, we obtain 
 \begin{equation}\label{firstsum}
 d_0+\sum\limits_{i}^{}{}'d_i(1-\kappa_i\nu)^{q-1-2p^l}+d_w (-\nu^{q-1-2p^l})=0
 \end{equation}
 for every $\nu\in I_1$. 
 Assume that $p \geq 5$. Thinking of this as a polynomial in $\nu$ of degree strictly less than $q-1$ with $q$ roots, we see that the (leading coefficient of the) polynomial vanishes. Hence, we obtain 
 \begin{equation*}
    \sum\limits_{i}^{}{}'d_i(-\kappa_i)^{q-1-2p^l}-d_w =0 
    \end{equation*}
    and so multiplying by a minus sign, we have
    \begin{equation*}
    \sum\limits_{i}^{}{}'d_i(-\kappa_i^{-2p^l})+d_w =0.
 \end{equation*}
 This shows that in \eqref{sumdi}, the  second terms on the second and third lines also cancel, and so we have
$$\sum\limits_{i}d_ig_i s_{N+1}^{p^l}=0.$$
  Note that \eqref{firstsum} is independent of $n$. This shows that $c_n\sum\limits_{i}d_ig_i s_{n}^{p^l}=0$ for every $n$ when the $g_i$'s are of the form $g^0_{m,\lambda}k$ for
  any $\lambda \in I_m$, $k \in K$. Next, we consider the sum $\sum\limits_{i}d_ig_i s_{N+1}^{p^l}$ where the $g_i$'s are of the form $g^0_{m-1,\lambda}k$. Repeating the above calculation, we get that either for some $m'$ the functions supported on the maximal radius $m'+N+1$ do not cancel each other or $\sum\limits_{i}d_ig_i f'=0$ where the sum is restricted to those $g_i$ in the positive part of the tree. 
 
 In the first case, the usual  argument works. 
 In the second case, we obtain that in \eqref{image of T} all $g_i$'s are of the form $g^1_{m,\lambda}k$ for some $m, \lambda, k$ (since all the $g^0$ terms are assumed to vanish). We multiply by $\beta$ and obtain
 
 $$\sum\limits_{i}d_i\beta g_if'-\beta f\in \text{Im}(T)$$
 with $g_i$'s of the form $g^1_{m,\lambda}k=\beta g^0_{m,\lambda} wk$. 
 We have 
  \begin{eqnarray*}\sum\limits_{i}d_i\beta g^1_{m_i,\lambda_i}k_if'-\beta f &\in& \text{Im}(T), \text{ i.e.},\\
  \sum\limits_{i}d_i g^0_{m_i,\lambda_i}wk_i f'-\beta f &\in & \text{Im}(T).
  \end{eqnarray*}
 Note that
 $\beta f$ is in negative radius. Now, we do the whole calculation again. We conclude that either we get a contradiction using the maximal radius argument above or that the whole function $\sum\limits_{i}d_ig_i f'=0$. The latter implies $f\in \text{Im}(T)$ by \eqref{image of T}. This contradicts the fact that $f$ was assumed to be nonzero in $U$.
\end{proof}

\section*{Acknowledgements} We thank CB, AC, PC, SM, MS and RV for useful comments, corrections and
questions. We acknowledge the support of project 1303/9/2025-R\&D-II-DAE/TIFR-17312.
 \\

\end{document}